\documentclass[12pt]{amsart}

\usepackage{amsmath,amssymb,amsfonts,amsthm}
\usepackage{ifpdf}
\usepackage{enumerate}
\usepackage{ulem}
\usepackage{leftidx}

\usepackage{breqn}

\usepackage{tikz}
\usetikzlibrary{calc}
\usepackage{tikz-qtree}

\ifpdf
\usepackage[pdftex,plainpages=false,hypertexnames=false,pdfpagelabels]{hyperref}
\else
\usepackage[dvips,plainpages=false,hypertexnames=false]{hyperref}
\fi

\newcommand{\arxiv}[1]{\href{http://arxiv.org/abs/#1}{\tt arXiv:\nolinkurl{#1}}}
\newcommand{\doi}[1]{\href{http://dx.doi.org/#1}{{\tt DOI:#1}}}

\newcommand{\googlebooks}[1]{(preview at \href{http://books.google.com/books?id=#1}{google books})}

\theoremstyle{plain}
\newtheorem{prop}{Proposition}[section]

\newtheorem{thm}[prop]{Theorem}

\newtheorem{lem}[prop]{Lemma}
\newtheorem{cor}[prop]{Corollary}
\newtheorem*{cor*}{Corollary}
\numberwithin{equation}{section}
\theoremstyle{remark}
\newtheorem{ex}[prop]{Example}

\newtheorem{remark}[prop]{Remark}           
\newtheorem*{rem*}{Remark}               %unnumbered remark
\newtheorem*{ex*}{Example}                %unnumbered exercise

\theoremstyle{definition}
\newtheorem{defn}[prop]{Definition}         % numbered definition
   
\newtheorem{nota}[prop]{Notation}   
\newtheorem*{defn*}{Definition}             % unnumbered definition

\theoremstyle{plain}

\newcounter{comment}
\newcommand{\noop}[1]{}

\def\clap#1{\hbox to 0pt{\hss#1\hss}}

\newcommand{\Integer}{\mathbb Z}

\newcommand{\Complex}{\mathbb C}

\def\semicolon{;}
\def\applytolist#1{
    \expandafter\def\csname multi#1\endcsname##1{
        \def\multiack{##1}\ifx\multiack\semicolon
            \def\next{\relax}
        \else
            \csname #1\endcsname{##1}
            \def\next{\csname multi#1\endcsname}
        \fi
        \next}
    \csname multi#1\endcsname}

\def\calc#1{\expandafter\def\csname c#1\endcsname{{\mathcal #1}}}
\applytolist{calc}QWERTYUIOPLKJHGFDSAZXCVBNM;
\def\bbc#1{\expandafter\def\csname bb#1\endcsname{{\mathbb #1}}}
\applytolist{bbc}QWERTYUIOPLKJHGFDSAZXCVBNM;
\def\bfc#1{\expandafter\def\csname bf#1\endcsname{{\mathbf #1}}}
\applytolist{bfc}QWERTYUIOPLKJHGFDSAZXCVBNM;

\newcommand{\id}{\boldsymbol{1}}
\renewcommand{\imath}{\mathfrak{i}}
\renewcommand{\jmath}{\mathfrak{j}}

\newcommand{\norm}[1]{\left|\left|#1\right|\right|}

\newcommand{\set}[1]{\left\{#1\right\}}

\newcommand{\code}[1]{{\tt #1}}

\makeatletter
\newcommand{\hashdef}[2]{\@namedef{#1}{#2}}
\newcommand{\hashlookup}[1]{\@nameuse{#1}}
\makeatother

\IfFileExists{../../graphs/lookup.tex}%
	{\newcommand{\pathtographs}{../../graphs/}}%
	{\newcommand{\pathtographs}{diagrams/graphs/}}

\hashdef{bwd1v1v1p1p1v1x0x0p0x1x0duals1v1v2x1}{63FDB4A04F409318}%
\hashdef{bwd1v1v1v1p1p1v1x0x0p0x1x0p0x0x1v1x0x0p0x1x0p0x0x1duals1v1v1x2x3v1x2x3}{772136FAF425B059}%
\hashdef{bwd1v1v1v1p1p1v1x0x0p0x1x0p0x0x1v1x0x0p0x1x0p0x0x1duals1v1v2x1x3v2x1x3}{24CD197B973929C1}%
\hashdef{bwd1v1v1v1p1p1v1x0x0v1duals1v1v1x2x3v1}{C4B6B3927CB75C0E}%
\hashdef{bwd1v1v1v1v1p1p1v1x0x0p0x1x0p0x0x1v1x0x0p0x1x0v1x0p0x1duals1v1v1v1x2x3v1x2}{E380D74CB4F0C4DF}%
\hashdef{bwd1v1v1v1v1p1p1v1x0x0p0x1x0p0x0x1v1x0x0p0x1x0v1x0p0x1duals1v1v1v2x1x3v2x1}{ACE47192690819DF}%
\hashdef{bwd1v1v1v1v1v1v1v1p1v1x0p0x1v1x0p0x1duals1v1v1v1v1x2v2x1}{C53E32050C862F37}%
\hashdef{bwd1v1v1v1v1v1v1v1p1v1x0p1x0duals1v1v1v1v1x2}{CA1952BED6C8EB18}%
\hashdef{gbg1v1v1p1p1v0x2x0}{F76652CFE2E42AB6}%
\hashdef{gbg1v1v1p1p1v1x0x0p0x1x0}{688C471F1900888A}%
\hashdef{gbg1v1v1p1p1v1x0x0p1x0x0}{D88CF710A1035A81}%
\hashdef{gbg1v1v1p1p1v1x1x0}{A421F9600BA24C44}%
\hashdef{gbg1v1v1v1p1p1v0x2x0}{E515F4CA1E6B335F}%
\hashdef{gbg1v1v1v1p1p1v1x0x0p0x1x0p0x0x1v0x2x0}{A49455E189858D54}%
\hashdef{gbg1v1v1v1p1p1v1x0x0p0x1x0p0x0x1v1x0x0p0x1x0p0x0x1}{BAF17D013FA99673}%
\hashdef{gbg1v1v1v1p1p1v1x0x0p0x1x0p0x0x1v1x0x0p1x0x0p0x1x0}{3DB6BC12EFDFADA8}%
\hashdef{gbg1v1v1v1p1p1v1x0x0p0x1x0p0x0x1v1x1x0}{A783CABB9C190B57}%
\hashdef{gbg1v1v1v1p1p1v1x0x0p1x0x0}{11487D705F5724BC}%
\hashdef{gbg1v1v1v1p1p1v1x0x0p1x0x0p0x1x0p0x0x1}{C7E321B622ABE5DB}%
\hashdef{gbg1v1v1v1p1p1v1x0x0v1}{D9E55F89702ABF01}%
\hashdef{gbg1v1v1v1p1p1v1x1x0}{4324D5E7527BA229}%
\hashdef{gbg1v1v1v1v1p1p1v0x2x0}{B12303BEF122C936}%
\hashdef{gbg1v1v1v1v1p1p1v1x0x0p0x1x0p0x0x1v0x1x0p0x0x1v1x0p1x0}{F882266BAAD68B6C}%
\hashdef{gbg1v1v1v1v1p1p1v1x0x0p0x1x0p0x0x1v0x1x0p0x0x1v1x1}{13B793E5B71ACFCD}%
\hashdef{gbg1v1v1v1v1p1p1v1x0x0p0x1x0p0x0x1v0x1x0p0x0x1v2x0}{C223F95512D4D194}%
\hashdef{gbg1v1v1v1v1p1p1v1x0x0p0x1x0p0x0x1v0x2x0}{FBCF443668B88E8B}%
\hashdef{gbg1v1v1v1v1p1p1v1x0x0p0x1x0p0x0x1v1x0x0p0x1x0v1x0p0x1}{D0CB6A067E64294F}%
\hashdef{gbg1v1v1v1v1p1p1v1x0x0p0x1x0p0x0x1v1x0x0p1x0x0}{CA8951CB81C89B55}%
\hashdef{gbg1v1v1v1v1p1p1v1x0x0p0x1x0p0x0x1v1x1x0}{A88B4510E32688FB}%
\hashdef{gbg1v1v1v1v1p1p1v1x0x0p1x0x0p0x1x0p0x0x1}{7F62DC1CFF7DA57C}%
\hashdef{gbg1v1v1v1v1p1p1v1x1x0}{80FF9EDCC0F539DD}%

\newcommand{\bigraph}[1]{{\hspace{-3pt}\begin{array}{c}%
  \raisebox{-2.5pt}{\includegraphics[height=6mm]{\pathtographs \hashlookup{#1}}}% 
\end{array}\hspace{-3pt}}}

\newcommand{\Hom}[3]{\operatorname{Hom_{#1}}\!\!\left(#2 \to #3\right)}

\newcommand{\tr}[1]{\text{tr}(#1)}

\usepackage{xifthen}  % http://ctan.org/pkg/xifthen
\usepackage{breqn}
\usepackage{yfonts}
\usepackage{ulem}

\tikzstyle{shaded}=[fill=red!10!blue!20!gray!30!white]
\tikzstyle{unshaded}=[fill=white]
\tikzstyle{empty box}=[circle, draw, thick, fill=white, opaque, inner sep=2mm]
\tikzstyle{annular}=[scale=.7, inner sep=1mm, baseline]
\tikzstyle{rectangular}=[scale=.75, inner sep=1mm, baseline=-.1cm]

\usepackage{xcolor}
\definecolor{dark-red}{rgb}{0.7,0.25,0.25}
\definecolor{dark-blue}{rgb}{0.15,0.15,0.55}
\definecolor{medium-blue}{rgb}{0,0,0.65}
\hypersetup{
   colorlinks, linkcolor={purple},
   citecolor={medium-blue}, urlcolor={medium-blue}
}

\DeclareMathOperator{\spann}{span}
\DeclareMathOperator{\Tr}{Tr}

\renewcommand{\set}[2]{\left\{#1\middle|#2\right\}}

\newcommand{\quadraticDiagram}[3]{
\begin{tikzpicture}[baseline = -.2cm]
	\clip (-.5,-.8)--(-.5,.7)--(2.4,.7)--(2.4,-.8);
	\draw (0,0)--(1.6,0)--(1.6,-1)--(0,-1)--(0,0);
	\node at (.8,.15) {{\scriptsize{$#1-1$}}};
	\node at (.4,-.5) {{\scriptsize{$#1+1$}}};
	\node at (2,-.5) {{\scriptsize{$#1+1$}}};
	\draw[thick, unshaded] (0,0) circle (.4);
	\node at (0,0) {$#2$};
	\draw[thick, unshaded] (1.6,0) circle (.4);
	\node at (1.6,0) {$#3$};
	\node at (1.6,.55) {$\star$};
	\node at (0,.55) {$\star$};
\end{tikzpicture}
}
\newcommand{\boxquadraticDiagram}[3]{
\begin{tikzpicture}[baseline = -.9cm]
	\clip (-.7,-1.7)--(-.7,.7)--(2.4,.7)--(2.4,-1.7);
	\draw (0,-2)--(0,0)--(1.6,0)--(1.6,-2);
	\node at (.8,.15) {{\scriptsize{$#1-1$}}};
	\node at (.4,-.6) {{\scriptsize{$#1+1$}}};
	\node at (2,-.6) {{\scriptsize{$#1+1$}}};
	\draw[thick, unshaded] (0,0) circle (.4);
	\node at (0,0) {$#2$};
	\draw[thick, unshaded] (1.6,0) circle (.4);
	\node at (1.6,0) {$#3$};
	\node at (1.6,.55) {$\star$};
	\node at (0,.55) {$\star$};
	\draw[thick, unshaded] (-.4,-1.5)--(-.4,-.8)--(2,-.8)--(2,-1.5)--(-.4,-1.5);
	\node at (.8,-1.15) {$f^{(2 #1+2)}$};
	\node at (-.55,-1.15) {$\star$};
\end{tikzpicture}
}
\newcommand{\jellyfishDiagram}[2]{
\begin{tikzpicture}[baseline = -.3cm]
	\filldraw[unshaded] (-.7,-.8)--(-.7,0) arc (180:0:.7cm)--(.7,-.8);
	\draw (0,0)--(0,-.8);
	\node at (.2,-.6) {{\scriptsize{$2 #1$}}};
	\draw[thick, unshaded] (0,0) circle (.4);
	\node at (0,0) {$#2$};
	\node at (-.25,-.5) {$\star$};
\end{tikzpicture}
}
\newcommand{\boxjellyfishDiagram}[2]{
\begin{tikzpicture}[baseline = -.9cm]
	\filldraw[unshaded] (-.7,-1.7)--(-.7,0) arc (180:0:.7cm)--(.7,-1.7);
	\draw (0,0)--(0,-1.7);
	\node at (.2,-.6) {{\scriptsize{$2 #1$}}};
	\draw[thick, unshaded] (0,0) circle (.4);
	\node at (0,0) {$#2$};
	\node at (-.25,-.5) {$\star$};
	\draw[thick, unshaded] (-1,-1.5)--(-1,-.8)--(1,-.8)--(1,-1.5)--(-1,-1.5);	
	\node at (0,-1.15) {$f^{(2 #1+2)}$};
	\node at (-1.15,-1.15) {$\star$};
\end{tikzpicture}
}
\newcommand{\capunderDiagram}[2]{
\begin{tikzpicture}[baseline = -.3cm]
	\draw (.2,0)--(.4,-1);
	\draw (-.2,0)--(-.4,-1);
	\draw (-.2,-1) arc (180:0:.2cm);
	\node at (-.5,-.5) {{\scriptsize{$#1$}}};
	\node at (.5,-.5) {{\scriptsize{$#1$}}};
	\draw[thick, unshaded] (0,0) circle (.4);
	\node at (0,0) {$#2$};
	\node at (0,-.55) {$\star$};
\end{tikzpicture}
}
\newcommand{\jwoverDiagram}[2]{
\begin{tikzpicture}[baseline = -.3cm]
	\draw[thick,unshaded] (0,-.2) circle (1.3);
	\draw (.2,0)--(.52,-.82);
	\draw (-.2,0)--(-.52,-.82);
	\draw (-.27,-.95) arc (180:0:.27cm);
	\node at (-.55,-.5) {{\scriptsize{$#1$}}};
	\node at (.55,-.5) {{\scriptsize{$#1$}}};
	\draw[thick, unshaded] (0,0) circle (.4);
	\draw[thick] (0,-.2) circle (.8);
	\node at (0,.7) {$f^{(2n+2)}$};
	\node at (0,-.55) {$\star$};
	\node at (0,0) {$#2$};
	\node at (0,-.55) {$\star$};
	\fill[unshaded] (-.15,-1.6)--(-.15,-.9)--(.15,-.9)--(.15,-1.6);
	\draw[thick] (-.15,-1.51)--(-.15,-.97);
	\draw[thick] (.15,-1.51)--(.15,-.97);
	\draw (.27,-1.48)--(.27,-1.7);
	\draw (-.27,-1.48)--(-.27,-1.7);
	\draw (.55,-1.39)--(.55,-1.7);
	\draw (-.55,-1.39)--(-.55,-1.7);
\end{tikzpicture}
}
\newcommand{\jellyfishSquared}[2]{
\begin{tikzpicture}[baseline = -.3cm]
	\filldraw[unshaded] (-.9,-.8)--(-.9,0) arc (180:0:.9cm)--(.9,-.8);
	\filldraw[unshaded] (-.7,-.8)--(-.7,0) arc (180:0:.7cm)--(.7,-.8);
	\draw (0,0)--(0,-.8);
	\node at (.2,-.6) {{\scriptsize{$2 #1$}}};
	\draw[thick, unshaded] (0,0) circle (.4);
	\node at (0,0) {$#2$};
	\node at (-.25,-.5) {$\star$};
\end{tikzpicture}
}
\newcommand{\boxjellyfishSquared}[2]{
\begin{tikzpicture}[baseline = -.9cm]
	\filldraw[unshaded] (-.9,-1.7)--(-.9,0) arc (180:0:.9cm)--(.9,-1.7);
	\filldraw[unshaded] (-.7,-1.7)--(-.7,0) arc (180:0:.7cm)--(.7,-1.7);
	\draw (0,0)--(0,-1.7);
	\node at (.2,-.6) {{\scriptsize{$2 #1$}}};
	\draw[thick, unshaded] (0,0) circle (.4);
	\node at (0,0) {$#2$};
	\node at (-.25,-.5) {$\star$};
	\draw[thick, unshaded] (-1,-1.5)--(-1,-.8)--(1,-.8)--(1,-1.5)--(-1,-1.5);	
	\node at (0,-1.15) {$f^{(2 #1+4)}$};
	\node at (-1.15,-1.15) {$\star$};
\end{tikzpicture}
}

\newcommand{\gA}{{\textgoth{A}}}
\newcommand{\fB}{{\mathfrak{B}}}
\newcommand{\qtac}{{\mathcal{QT}\hspace{-.05in}\mathcal{AC}}}
\title{Constructing spoke subfactors using the jellyfish algorithm}
\author{Kim Morrison and David Penneys}
\begin{document}
\maketitle

\begin{abstract}
Using Jones' quadratic tangles formulas, we automate the construction of the 4442, 3333, 3311, and 2221 spoke subfactors by finding sets of 1-strand jellyfish generators. The 4442 spoke subfactor is new, and the 3333, 3311, and 2221 spoke subfactors were previously known.
% This is the published version of \arxiv{...}
\end{abstract}

%%%%%%%%%%%%%%%%%%%%%%%%%%%%%%%%%%%%%%%%%%%%%%%%%%%%%%%%%%%%%%%%%%%%%%%%%%%

\section{Introduction}

In this paper, we construct a number of subfactors by giving explicit generators and relations for the corresponding planar algebras. In particular, we use Jones' analysis of quadratic tangles to systematically derive `jellyfish relations', which are straightforward to analyze. 
We construct one new subfactor as well as three previously known subfactors.
The main result of the paper is:
\begin{thm}
There is a `4442' subfactor at index $3+\sqrt{5}$ with principal graphs 
$$
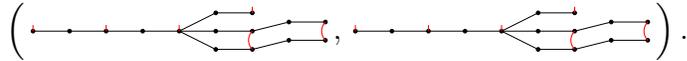
\left(\bigraph{bwd1v1v1v1v1p1p1v1x0x0p0x1x0p0x0x1v1x0x0p0x1x0v1x0p0x1duals1v1v1v2x1x3v2x1},\bigraph{bwd1v1v1v1v1p1p1v1x0x0p0x1x0p0x0x1v1x0x0p0x1x0v1x0p0x1duals1v1v1v2x1x3v2x1}\right).
$$
\end{thm}

At this point, there are very strong number theoretic and combinatorial constraints on the possible principal graphs of subfactors. Indeed, when we find a graph that satisfies all these constraints, we reasonably suspect that there are actually subfactors with that principal graph. Nevertheless, the final step of actually constructing such subfactors remains very difficult.

It is known that every subfactor planar algebra embeds in the graph planar algebra of its principal graph \cite{MR2812459,tvc}. Thus, a standard approach to constructing a planar algebra is to identify some candidate elements in the appropriate graph planar algebra, then prove that the subalgebra they generate is a subfactor planar algebra with the desired principal graph. In many cases, it is not that difficult to obtain the candidate elements, for example, by solving certain polynomial equations in the graph planar algebra, or finding flat elements in the graph planar algebra with respect to some connection. Note that the existence of flat elements does not imply the connection itself is flat; the flat elements could generate a subfactor planar algebra with different principal graphs. 

Thus, given some candidate generators, the challenge is to show that these elements generate the desired planar algebra. An initial difficult step is to show that they generate an evaluable planar algebra; that is, every closed diagram is a multiple of the empty diagram (equivalently, the zero box space is 1-dimensional). Of course, if we obtained our elements as flat elements with respect to a connection, this step is automatic.
Once we have established the planar algebra is evaluable, we have some subfactor planar algebra. We need to identify its combinatorial invariants. Often, the analysis of some small projections in the subalgebra and some combinatorial arguments suffice to determine the principal graph.
 
In \cite{0909.4099}, Bigelow-Morrison-Peters-Snyder constructed the extended Haagerup planar algebra, which had long been expected to exist, with principal graphs
$$
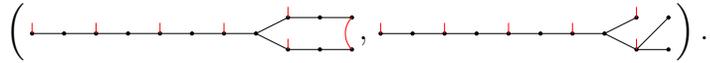
\left(\bigraph{bwd1v1v1v1v1v1v1v1p1v1x0p0x1v1x0p0x1duals1v1v1v1v1x2v2x1}, \bigraph{bwd1v1v1v1v1v1v1v1p1v1x0p1x0duals1v1v1v1v1x2}\right).
$$
The essential insight was the jellyfish algorithm, introduced therein, which provides a powerful framework for proving a planar algebra is evaluable and has the desired supertransitivity (the length of the initial arm of the principal graph).

Suppose we have a set of elements in a planar algebra, each a lowest weight rotational eigenvector, which we are thinking of as generators.
A `jellyfish relation' is an identity in which the left hand side is simply a single generator with some number of strands between it and the starred point on the boundary, and the right hand side is some linear combination, in every term of which every generator is adjacent to the starred point on the boundary. In \cite{0909.4099}, there were two jellyfish relations (here $n=4$ corresponds to the Haagerup planar algebra, and $n=8$ corresponds to the extended Haagerup planar algebra): 
%\nn{these are box jellyfish relations... we don't want to display the full relations here, so some more explanation is required }
%\input{diagrams/tikz/TikzStyles}
\begin{align*}
\boxjellyfishDiagram{n}{\check{S}} &= \frac{\sqrt{[n][n+2]}}{[n+1]}\boxquadraticDiagram{n}{S}{S}\\
\boxjellyfishSquared{n}{S} 
&= \frac{[2][2n+4]}{[n+1][n+2]}
\begin{tikzpicture}[baseline = -.9cm]
	\clip (-.7,-1.7)--(-.7,.7)--(4,.7)--(4,-1.7);
	\draw (0,-2)--(0,0);
	\draw (1.6,-2)--(1.6,0);
	\draw (3.2,-2)--(3.2,0);	
	\draw (0,0)--(3.2,0);
	\node at (.8,.15) {{\scriptsize{$n-1$}}};
	\node at (2.4,.15) {{\scriptsize{$n-1$}}};
	\node at (.4,-.6) {{\scriptsize{$n+1$}}};
	\node at (1.8,-.6) {{\scriptsize{$2$}}};
	\node at (3.6,-.6) {{\scriptsize{$n+1$}}};
	\draw[thick, unshaded] (0,0) circle (.4);
	\node at (0,0) {$S$};
	\draw[thick, unshaded] (1.6,0) circle (.4);
	\node at (1.6,0) {$S$};
	\draw[thick, unshaded] (3.2,0) circle (.4);
	\node at (3.2,0) {$S$};
	\node at (3.2,.55) {$\star$};
	\node at (1.6,.55) {$\star$};
	\node at (0,.55) {$\star$};
	\draw[thick, unshaded] (-.4,-1.5)--(-.4,-.8)--(3.6,-.8)--(3.6,-1.5)--(-.4,-1.5);
	\node at (1.6,-1.15) {$f^{(2 n+4)}$};
	\node at (-.55,-1.15) {$\star$};
\end{tikzpicture}
\end{align*}
These are actually \emph{box jellyfish relations}; it is easy enough to see that by expanding out the Jones-Wenzl idempotent on the left hand side in terms of Temperley-Lieb diagrams, and moving all non-identity terms to the right, that these relations become jellyfish relations as described above. See  Subsection \ref{sec:boxjellyfish} for more details.

In a $k$-strand jellyfish relation, the left hand side has $k$ strands above the generator. Above, we have a one-strand jellyfish relation for $\check{S}$, the Fourier transform, and a two-strand jellyfish relation for $S$.

A complete set of jellyfish relations is one such that by repeated application, we can rewrite any diagram as a linear combination of diagrams in which every generator is adjacent to the `outside' starred region. We picturesquely refer to this process as `the jellyfish algorithm': we gradually float all the jellyfish to the surface of the ocean, possibly creating new jellyfish along the way.
The pair of relations above is a complete set: the one-strand relation removes all the instances of $\check{S}$, then the two-strand relation allows us to float all the instances of $S$ to the surface.

It is typically easy to see that any closed diagram with all generators adjacent to the boundary is evaluable by iteratively finding an adjacent pair of generators which are connected by sufficiently many strands. The entire jellyfish algorithm is somewhat unusual amongst algorithms for simplifying a planar diagram; at intermediate steps, it requires making the diagram much more complicated.

In this paper, we present a systematic approach to identifying jellyfish relations using Jones' paper analyzing quadratic tangles \cite{math/1007.1158}. In particular, we show how to find all 1-strand jellyfish relations for which the right hand side involves at most two generators in each term. We find that for quite a number of potential principal graphs, the jellyfish relations obtained in this way constitute a complete set. This assures us that we have constructed some subfactor planar algebra, and a little separate work in each case identifies the principal graph as the intended graph.

It is worth noting, however, that this approach is far from uniformly successful! In \cite{0909.4099}, the authors needed 2-strand jellyfish relations. Indeed, a result of Bigelow and Penneys \cite{1208.1564} shows that having a complete set of 1-strand jellyfish relations implies that both the principal and dual principal graphs are spoke graphs. Thus by the triple point obstruction \cite{MR1317352,1007.2240}, any subfactor with principal graphs beginning with a triple point cannot have a complete set of 1-strand jellyfish relations. We anticipate the results of \cite{1208.1564} giving strong constraints on subfactors whose principal graphs are not both spokes. 

Even with these limitations, we do have a number of interesting examples. We give a jellyfish presentation of the new 4442 subfactor along with the 3311 subfactor \cite{MR999799}, the 2221 subfactor constructed in \cite{MR1832764} (reproducing all the work of Han's thesis \cite{1102.2052}, in an entirely automated fashion!), and one of the 3333 subfactors previously constructed by Izumi in unpublished work. These graphs appear in Figure \ref{fig:graphs}.

\begin{figure}[ht]
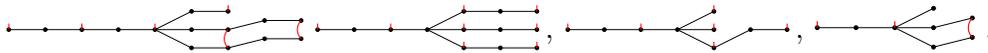

$$
\bigraph{bwd1v1v1v1v1p1p1v1x0x0p0x1x0p0x0x1v1x0x0p0x1x0v1x0p0x1duals1v1v1v2x1x3v2x1}
\bigraph{bwd1v1v1v1p1p1v1x0x0p0x1x0p0x0x1v1x0x0p0x1x0p0x0x1duals1v1v1x2x3v1x2x3},
\bigraph{bwd1v1v1v1p1p1v1x0x0v1duals1v1v1x2x3v1},
\bigraph{bwd1v1v1p1p1v1x0x0p0x1x0duals1v1v2x1}\,.
$$
\caption{The principal graphs of the four subfactors we construct in this paper. In each case, the dual principal graph is the same as the principal graph.}
\label{fig:graphs}
\end{figure}

Immediately after hearing about our construction of a 4442 subfactor, Izumi noticed that a 4442 subfactor can be constructed using a $\Integer/3\Integer$ quotient of the 3333 subfactor. This will appear in one of his forthcoming papers. Conversely, we constructed the 3333 subfactor only after hearing Izumi's  construction.

In addition to giving a generators and jellyfish relations presentation of each of these four subfactor planar algebras, we show that the 4442, 3333, and 2221 subfactors are each self-dual, and moreover symmetrically self-dual. This essentially means that one can ignore the shading in the planar algebras, and thus there exist fusion categories with the same principal graphs. We will investigate this further in a future paper \cite{4442equi}.

\paragraph{\textbf{Acknowledgements.}}
The authors would like to thank Vaughan Jones and Masaki Izumi for many helpful conversations. We appreciated the referee's careful reading of the manuscript and many helpful criticisms. Both authors were supported by DOD-DARPA grants HR0011-11-1-0001 and  HR0011-12-1-0009. Kim Morrison was at the Miller Institute for Basic Research and the Australian National University while writing this paper, and in part supported by the ARC DECRA grant DE120100232. David Penneys was also supported by NSF grant DMS-0856316.

\subsection{The {\tt FusionAtlas}}
This paper relies on some substantial calculations. In particular, our efforts to find the generators in the various graph planar algebras made use of a variety of techniques, some ad-hoc, some approximate, and some computationally expensive. This paper essentially does not address that work. Instead, we merely present the discovered generators and verify some relatively easy facts about them. In particular, the proofs presented in this paper rely on the computer in a much weaker sense. We need to calculate certain numbers of the form $\Tr(STR)$, where $S$, $T$, and $R$ are rather large martrices, and the computer does this for us. We also entered all the formulas derived in this paper into {\tt Mathematica}, and had the computer automatically evaluate the various quantities which appear in our derivation of jellyfish relations. As a reader may be interested in seeing these programs, we include a brief instruction on finding and running these programs.

The {\tt arXiv} sources of this article contain in the {\tt code} subdirectory a number of files, including:
\begin{itemize}
\item {\tt Generators.nb}, which reconstructs the generators from our terse descriptions of them in Appendix \ref{sec:generators}.
\item {\tt QuadraticTangles.nb}, which calculates the requisite moments of these generators, and performs the linear algebra necessary to derive the jellyfish relations.
\item {\tt GenerateLaTeX.nb}, which typesets each subsection of Section \ref{sec:relations} for each planar algebra, and many mathematical expressions in Appendices \ref{sec:generators} and \ref{sec:moments}.
\end{itemize}

The {\tt Mathematica} notebook {\tt Generators.nb} can be run by itself. The final cells of that notebook write the full generators to the disk; this must be done before running {\tt QuadraticTangles.nb}. The {\tt QuadraticTangles.nb} notebook relies on the {\tt FusionAtlas}, a substantial body of code the authors have developed along with  Emily Peters, Noah Snyder and James Tener to perform calculations with subfactors and fusion categories. To obtain a local copy, you first need to ensure that you have {\tt Mercurial}, the distributed version control system, installed on your machine. With that, the command 
\begin{quote}
{\tt hg clone https://bitbucket.org/fusionatlas/fusionatlas}
\end{quote}
will create a local directory called {\tt fusionatlas} containing the latest version. In the {\tt QuadraticTangles.nb} notebook, you will then need to adjust the paths appearing in the first input cell to ensure that your local copy is included. After that, running the entire notebook reproduces all the calculations described below.

We invite any interested readers to contact us with questions or queries about the use of these notebooks or the {\tt FusionAtlas} package.

%%%%%%%%%%%%%%%%%%%%%%%%%%%%%%%%%%%%%%%%%%%%%%%%%%%%%%%%%%%%%%%%%%%%%%%%%
%%%%%%%%%%%%%%%%%%%%%%%%%%%%%%%%%%%%%%%%%%%%%%%%%%%%%%%%%%%%%%%%%%%%%%%%%
%%%%%%%%%%%%%%%%%%%%%%%%%%%%%%%%%%%%%%%%%%%%%%%%%%%%%%%%%%%%%%%%%%%%%%%%%
\section{Background}

\subsection{Graph planar algebra embedding}\label{sec:FindGenerators}
A finite depth subfactor planar algebra embeds in the graph planar algebra of its principal graph \cite{MR2812459,tvc}. We begin by assuming that the desired planar algebra is $n-1$ supertransitive, so $P_{k,+} = TL_{k,+}$ for $k=0,\ldots, n-1$ and is generated by the orthogonal complement $P_{n,+} \ominus TL_{n,+}$ of the Temperley-Lieb algebra in the $n$-box space. If this orthogonal complement has dimension $k$, we say that the planar algebra is $n-1$ supertransitive of excess $k$. The excess can be read off from the principal graph; it's the sum of the squares of the multiplicities of the edges immediately beyond the branch point, minus one (or simply two less than the valence of the branch point, in the simply laced case).

To construct the subfactor planar algebras in these papers, we first find a \underline{connection} on the desired pair of principal graphs $(\Gamma,\Gamma')$. We do so by first computing a numerical approximation to high precision, then guessing the exact solution using {\tt Mathematica}'s {\tt RootApproximant} function, and finally verifying using exact arithmetic in a number field that the candidate exact solution satisfies the necessary equations. As we use this numerical method, we cannot say with certainty how many connections exist on these graphs. However, Izumi has determined all such connections, and we believe our list of possible connections coincides with his. As the main goal of this paper is to construct the new 4442 subfactor by the development of methods for deriving jellyfish relations, rather than proving uniqueness, we restrict our attention to a single connection.

Using this connection, it is possible to discover the rotational eigenvalues of the desired generators via  \cite[Theorem 1.7]{1109.3190}; in a certain gauge, $\tr{UU^t} = 2 + \sum \omega$, where $U$ is the `branch matrix' of connection entries going through the branch point, and the sum is over the $k$ rotational eigenvalues $\omega$, with multiplicity, of the lowest weight vectors in the $n$-box space.  We then compute the equations for flatness and find $k$ flat $n$-boxes with respect to the connection. (Note that these flat elements are not necessarily self-adjoint as in \cite{math/1007.1158}.)

However, merely having flat elements for a connection does not imply that the connection itself is flat. 
The planar algebra generated by the flat $n$-boxes is necessarily evaluable, and indeed a subfactor planar algebra, which we denote by $P_\bullet^\Gamma$. 
Thus we have constructed some subfactor by \cite{MR996454,MR1642584,JonesPANotes,tvc}.
This subfactor, however, need not have the principal graphs we started with; indeed, it could have smaller supertransitivity.
Determining even the supertransitivity of the principal graph of the subfactor planar algebra generated by the flat $n$-boxes requires some knowledge of the skein theory of the planar algebra, which is of comparable difficulty with directly computing jellyfish relations. In fact, by Lemma \ref{lem:Supertransitive}, having jellyfish relations ensures the supertransitivity is correct.
One could attempt to verify the connection itself is flat, which would also ensure the supertransitivity is correct, but this is extremely computationally expensive. 
The point of this paper is to avoid this work using planar algebras and the jellyfish algorithm.

Hence, we still need to analyze the subfactor planar algebra $P_\bullet^\Gamma$. In the analysis which follows, we never rely on the fact that we obtained these elements as flat elements with respect to some connection. In particular, nothing that follows relies on the correctness of the method described above for guessing the generators!

First, we need to calculate the principal graphs of $P_\bullet^\Gamma$ (which we are hoping are the graphs we started with, $(\Gamma, \Gamma')$). Second, we would like to know how to evaluate closed diagrams in the generators.

For the spoke graphs we consider in this paper, $\Gamma$ and $\Gamma'$ always coincide, and are drawn from the set $\{4442,3333,3311,2221\}$ (see Figure \ref{fig:graphs}), where the numbers refer to the length of the spokes emanating from the central vertex.
By \cite{1208.1564}, if such subfactor planar algebras were to exist with the desired principal graphs, they would have \emph{1-strand jellyfish generators} at depth $n$. We use Jones' quadratic tangles techniques \cite{math/1007.1158} to find these 1-strand jellyfish relations in $P_\bullet^\Gamma$. We then use the generators to determine that $P_\bullet^\Gamma$ has the correct principal graph.

%%%%%%%%%%%%%%%%%%%%%%%%%%%%%%%%%%%%%%%%%%%%%%%%%%%%%%%%%%%%%%%%%%%%%%%%%%%%
\subsection{Quadratic tangles}

In \cite{math/1007.1158}, Jones uses quadratic tangles techniques to prove a number of formulas about planar generators in a subfactor planar algebra. In fact, many of the formulas there hold in more generality.

\begin{nota}
Recall that the Fourier transform $\cF$ is given by
$$
\cF=
\begin{tikzpicture}[baseline = 0cm]
	\clip (0,0) circle (1cm);
	\draw[] (0,0) circle (1cm);
	\draw[] (0,0)--(-10:1.1cm) arc (-10:30:1.1cm) --(0,0); 
	\draw[] (0,0)--(70:1.1cm) arc (70:110:1.1cm) --(0,0); 
	\draw[] (0,0)--(150:1.1cm) arc (150:190:1.1cm) --(0,0); 
	\draw[] (0,0)--(230:1.1cm) arc (230:310:1.1cm) --(0,0); 
	\draw[thick, unshaded] (0,0) circle (.4cm);
	\node at (90:.52) {$\star$};
	\node at (130:.86) {$\star$};
	\node at (0,-.6) {$\cdots$};
	\node at (0,0) {};
	\draw[ultra thick] (0,0) circle (1cm);
\end{tikzpicture}\,.
$$
For a rotational eigenvector $S\in P_{n,\pm}$ corresponding to an eigenvalue $\omega_S=\sigma_S^2$, we define another rotational eigenvector $\check{S}\in P_{n,\mp}$ by $\check{S}=\sigma_S^{-1}\cF(S)$. Note that $\cF(\check{S})=\sigma_S S$, so $\check{\check{S}}=S$.
\end{nota}

\begin{defn}
Suppose $P_\bullet$ is a not necessarily evaluable $*$-planar algebra, i.e., we don't know if $\dim(P_{n,\pm})<\infty$ or if $\dim(P_{0,\pm})=1$. We call $P_\bullet$ \underline{unitary} if for all $n\geq 0$, the $P_{0,\pm}$-valued sesquilinear form on $P_{n,\pm}$ given by $\langle x,y\rangle=\Tr(y^*x)$ is positive definite (in the operator-valued sense).
\end{defn}

\begin{defn}
Suppose $P_\bullet$ is a unitary, spherical, shaded planar algebra with modulus $\delta>2$ which is not necessarily evaluable.
A finite  set $\fB\subset P_{n,+}$ is called a \emph{set of generators} if the elements of $\fB$ are linearly independent, self-adjoint, low-weight eigenvectors for the rotation, i.e, for all $S\in\fB$,
\begin{itemize}
\item $S=S^*$,
\item $S$ is uncappable, and
\item $\rho(S)=\omega_S S$ for some $n$-th root of unity $\omega_S$.
\end{itemize}
Given a set of generators $\fB$, we get a set of dual generators $\check{\fB}=\set{\check{S}}{S\in\fB}$. 

We say a set of generators $\fB$ has \emph{scalar moments} if $\Tr(R),\Tr(RS),\Tr(RST)$ and $\Tr(\check{R}),\Tr(\check{R}\check{S}),\Tr(\check{R}\check{S}\check{T})$ are scalar multiples of the empty diagram in $P_{0,+}$ and $P_{0,-}$ respectively for each $R,S,T\in\fB$. 

% probably false
%\begin{remark}
%Note that if the complex span of $\fB\cup\{f^{(n)}\}$ forms an algebra under the usual multiplication, then $\fB$ automatically has scalar moments.
%\end{remark}

If a set of generators $\fB$ has scalar moments, we say it is \emph{orthonormal} if for all $S,T\in \fB$, $\langle S,T\rangle =\Tr(ST)= \delta_{S,T}$.
\end{defn}

\begin{nota}
Since $\delta>2$, given a generator $R\in\fB$, the annular tangles
$$
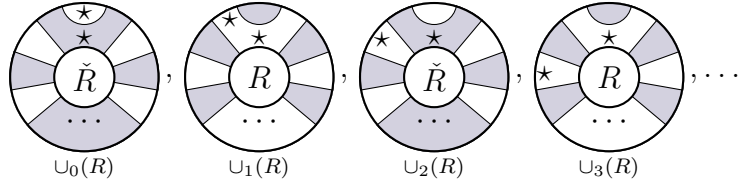
\underset{\cup_0(R)}{
\begin{tikzpicture}[baseline = 0cm]
	\clip (0,0) circle (1cm);
	\draw[shaded] (0,0) circle (1cm);
	\draw[unshaded] (.3,1) arc (0:-180:.3cm);
	\draw[unshaded] (0,0)--(130:1.1cm) arc (130:160:1.1cm) --(0,0); 
	\draw[unshaded] (0,0)--(190:1.1cm) arc (190:220:1.1cm) --(0,0); 
	\draw[unshaded] (0,0)--(50:1.1cm) arc (50:20:1.1cm) --(0,0); 
	\draw[unshaded] (0,0)--(-10:1.1cm) arc (-10:-40:1.1cm) --(0,0); 
	\draw[thick, unshaded] (0,0) circle (.4cm);
	\node at (90:.52) {$\star$};
	\node at (90:.86) {$\star$};
	\node at (0,-.6) {$\cdots$};
	\node at (0,0) {$\check{R}$};
	\draw[ultra thick] (0,0) circle (1cm);
\end{tikzpicture}
}
\,,\,
\underset{\cup_1(R)}{
\begin{tikzpicture}[baseline = 0cm]
	\clip (0,0) circle (1cm);
	\draw[unshaded] (0,0) circle (1cm);
	\draw[shaded] (.3,1) arc (0:-180:.3cm);
	\draw[shaded] (0,0)--(130:1.1cm) arc (130:160:1.1cm) --(0,0); 
	\draw[shaded] (0,0)--(190:1.1cm) arc (190:220:1.1cm) --(0,0); 
	\draw[shaded] (0,0)--(50:1.1cm) arc (50:20:1.1cm) --(0,0); 
	\draw[shaded] (0,0)--(-10:1.1cm) arc (-10:-40:1.1cm) --(0,0); 
	\draw[thick, unshaded] (0,0) circle (.4cm);
	\node at (90:.52) {$\star$};
	\node at (117:.86) {$\star$};
	\node at (0,-.6) {$\cdots$};
	\node at (0,0) {${R}$};
	\draw[ultra thick] (0,0) circle (1cm);
\end{tikzpicture}
}
\,,\,
\underset{\cup_2(R)}{
\begin{tikzpicture}[baseline = 0cm]
	\clip (0,0) circle (1cm);
	\draw[shaded] (0,0) circle (1cm);
	\draw[unshaded] (.3,1) arc (0:-180:.3cm);
	\draw[unshaded] (0,0)--(130:1.1cm) arc (130:160:1.1cm) --(0,0); 
	\draw[unshaded] (0,0)--(190:1.1cm) arc (190:220:1.1cm) --(0,0); 
	\draw[unshaded] (0,0)--(50:1.1cm) arc (50:20:1.1cm) --(0,0); 
	\draw[unshaded] (0,0)--(-10:1.1cm) arc (-10:-40:1.1cm) --(0,0); 
	\draw[thick, unshaded] (0,0) circle (.4cm);
	\node at (90:.52) {$\star$};
	\node at (144:.86) {$\star$};
	\node at (0,-.6) {$\cdots$};
	\node at (0,0) {$\check{R}$};
	\draw[ultra thick] (0,0) circle (1cm);
\end{tikzpicture}
}
\,,\,
\underset{\cup_3(R)}{
\begin{tikzpicture}[baseline = 0cm]
	\clip (0,0) circle (1cm);
	\draw[unshaded] (0,0) circle (1cm);
	\draw[shaded] (.3,1) arc (0:-180:.3cm);
	\draw[shaded] (0,0)--(130:1.1cm) arc (130:160:1.1cm) --(0,0); 
	\draw[shaded] (0,0)--(190:1.1cm) arc (190:220:1.1cm) --(0,0); 
	\draw[shaded] (0,0)--(50:1.1cm) arc (50:20:1.1cm) --(0,0); 
	\draw[shaded] (0,0)--(-10:1.1cm) arc (-10:-40:1.1cm) --(0,0); 
	\draw[thick, unshaded] (0,0) circle (.4cm);
	\node at (90:.52) {$\star$};
	\node at (177:.86) {$\star$};
	\node at (0,-.6) {$\cdots$};
	\node at (0,0) {${R}$};
	\draw[ultra thick] (0,0) circle (1cm);
\end{tikzpicture}
}
\,,\dots
$$
are a basis for $\gA_{n+1}(R)$, the annular consequences of $R$ inside $P_{n+1,+}$ \cite{MR1929335,math/1007.1158}. One calculates the dual annular basis $\set{\widehat{\cup}_i(R)}{i=0,\dots,2n+1}$ from $\langle \widehat{\cup}_i(R),\cup_j(R)\rangle = \delta_{i,j}$, where the inner product is linear on the \emph{right}. See Definition 4.2.6 of \cite{math/1007.1158} for an explicit formula.
\end{nota}

\begin{thm}
All the formulas of \S 4 of  \cite{math/1007.1158} hold in any unitary, spherical, shaded planar algebra with modulus $\delta>2$ for any orthonormal set of generators $\fB$ with scalar moments.
\end{thm}
\begin{proof}
Jones explicitly restricts to an evaluable planar algebra (in fact, he says `subfactor,' but evaluable is the only condition we're now leaving off) before proving these formulas, but upon reading through the proofs, it is clear that the generators having scalar moments is sufficient to work in the generality we need here.
\end{proof}

%%%%%%%%%%%%%%%%%%%%%%%%%%%%%%%%%%%%%%%%%%%%%%%%%%%%%%%%%%%%%%%%%%%%%%%%%%%%
\subsection{Spherical and lopsided planar algebras}

In \cite{1205.2742}, Morrison and Peters describe the \emph{spherical} and \emph{lopsided} conventions for planar algebras. In the spherical convention, both shaded and unshaded contractible closed loops count for a multiplicative factor of $\delta$; however, in the lopsided convention, shaded contractible closed loops count for $1$ while unshaded contractible closed loops count for $\delta^2$. The main advantage of working in the lopsided planar algebra is that there are fewer square roots, so arithmetic is easier.  In particular the number field in which we calculate is much smaller. Hence we use the lopsided planar algebra to compute the moments of our generators.

The map $\natural \colon P_\bullet^{spherical}\to P_\bullet^{lopsided}$ from \cite{1205.2742}  is not a planar algebra map, but it commutes with the action of the planar operad up to certain scalars. When we draw our tangles in the standard form where each input and output disk is a rectangle with the distinguished interval on the left and the same number of strings attach to the top and bottom of each rectangle, then there is a power of $\delta^{\pm1}$ for each critical point which is shaded above, and the power of $\delta$ corresponds to the sign of the critical point:
$$
\begin{tikzpicture}[baseline = -.15cm]
	\fill[unshaded] (-.5,-.5)--(-.5,.25)--(.5,.25)--(.5,-.5);
	\draw[shaded] (.25,.25)--(.25,0) arc (0:-180:.25cm) --(-.25,.25);
\end{tikzpicture}
\longleftrightarrow \delta
\hspace{1.5cm}
\begin{tikzpicture}[baseline = -.15cm]
	\clip (-.5,-.25)--(-.5,.5)--(.5,.5)--(.5,-.25);
	\fill[shaded] (-.5,-.5)--(-.5,.5)--(.5,.5)--(.5,-.5);
	\draw[unshaded] (-.25,-.5)--(-.25,0) arc (180:0:.25cm) --(.25,-.5);
\end{tikzpicture}
\longleftrightarrow \delta^{-1}.
$$
%The map $\natural$ also interacts with the adjoint in $P_\bullet$. As our generators are self-adjoint, we will not need the formula here, and we refer the reader to \cite{1205.2742} for more details. \nn{remove this para, unless we actually use the formula for adjoint}

\begin{ex}\label{ex:FourierTransform}
We will work out the correction factor arising when commuting $\natural$ and $\cF$. Note that
$$
\cF(S)
=
\begin{tikzpicture}[baseline = 0cm]
	\clip (0,0) circle (1cm);
	\draw[shaded] (0,0) circle (1cm);
	\draw[unshaded] (0,0)--(-10:1.1cm) arc (-10:30:1.1cm) --(0,0); 
	\draw[unshaded] (0,0)--(70:1.1cm) arc (70:110:1.1cm) --(0,0); 
	\draw[unshaded] (0,0)--(150:1.1cm) arc (150:190:1.1cm) --(0,0); 
	\draw[unshaded] (0,0)--(230:1.1cm) arc (230:310:1.1cm) --(0,0); 
	\draw[thick, unshaded] (0,0) circle (.4cm);
	\node at (90:.52) {$\star$};
	\node at (130:.86) {$\star$};
	\node at (0,-.6) {$\cdots$};
	\node at (0,0) {$S$};
	\draw[ultra thick] (0,0) circle (1cm);
\end{tikzpicture}
=
\begin{cases}
\begin{tikzpicture}[baseline = 0cm]
	\clip (-1.4,.8)--(1.4,.8)--(1.4,-.8)--(-1.4,-.8);
	\draw[shaded] (-.6,0)--(-.6,-.4) arc (0:-180:.2cm)--(-1,.9)--(-1.5,.9)--(-1.5,-.9)--(-.4,-.9)--(-.4,0);
	\draw[shaded] (-.6,0)--(-.6,.9)--(-.4,.9)--(-.4,0);
	\draw[shaded] (.6,0)--(.6,-.9)--(.4,-.9)--(.4,0);
	\draw[shaded] (.6,0)--(.6,.4) arc (180:0:.2cm)--(1,-.9)--(1.5,-.9)--(1.5,.9)--(.4,.9)--(.4,0);
	\draw[thick,unshaded] (-.8,.4)--(.8,.4)--(.8,-.4)--(-.8,-.4)--(-.8,.4);
	\draw[ultra thick] (-1.4,.8)--(1.4,.8)--(1.4,-.8)--(-1.4,-.8)--(-1.4,.8);
	\node at (0,0) {$S$};
	\node at (0,.6) {$\cdots$};
	\node at (0,-.6) {$\cdots$};
\end{tikzpicture}
&\text{if $n$ is even}
\\
\\
\begin{tikzpicture}[baseline = 0cm]
	\clip (-1.4,.8)--(1.4,.8)--(1.4,-.8)--(-1.4,-.8);
	\draw[shaded] (-.6,0)--(-.6,-.4) arc (0:-180:.2cm)--(-1,.9)--(-1.5,.9)--(-1.5,-.9)--(-.4,-.9)--(-.4,0);
	\draw[shaded] (-.6,0)--(-.6,.9)--(-.4,.9)--(-.4,0);
	\draw (.4,-.9)--(.4,.9);
	\draw[shaded] (.6,0)--(.6,.4) arc (180:0:.2cm)--(1,-.9)--(.6,-.9)--(.6,0);
	\draw[thick,unshaded] (-.8,.4)--(.8,.4)--(.8,-.4)--(-.8,-.4)--(-.8,.4);
	\draw[ultra thick] (-1.4,.8)--(1.4,.8)--(1.4,-.8)--(-1.4,-.8)--(-1.4,.8);
	\node at (0,0) {$S$};
	\node at (0,.6) {$\cdots$};
	\node at (0,-.6) {$\cdots$};
\end{tikzpicture}
&\text{if $n$ is odd.}
\end{cases}
$$
Hence we have 
$$
\natural \cF = 
\begin{cases}
\cF \natural & \text{if $n$ is even} \\
\delta^{-1} \cF \natural & \text{if $n$ is odd.} 
\end{cases}
$$
\end{ex}

\begin{ex}\label{ex:CorrectionFactors}
Similarly one can calculate that
\begin{align*}
\Tr(S) =\natural \Tr(S) &=
\begin{cases}
\Tr(\natural S) & \text{ if $n$ is even}\\
\delta\Tr(\natural S) & \text{ if $n$ is odd}
\end{cases}
\\
\Tr(\check{S})=\natural \Tr(\check{S}) &=
\begin{cases}
\Tr(\natural \check{S}) & \text{ if $n$ is even}\\
\delta^{-1}\Tr(\natural \check{S}) & \text{ if $n$ is odd.}
\end{cases}
\end{align*}
Hence if $S_1,\dots, S_k\in\fB$, we compute the moment
$$
\Tr(\cF(S_1)\cdots \cF(S_k))=
\begin{cases}
\delta^{-k} \Tr(\cF(\natural S_1)\cdots \cF(\natural S_k)) & \text{ if $n$ is even}\\
\delta^{-1} \Tr(\cF(\natural S_1)\cdots \cF(\natural S_k)) & \text{ if $n$ is odd}
\end{cases}
$$
which implies
$$
\Tr(\check{S_1}\cdots \check{S_k})
=
\begin{cases}
\delta^{-k} \Tr(\natural\check{S_1}\cdots \natural\check{S_k}) & \text{ if $n$ is even}\\
\delta^{-1} \Tr(\natural\check{S_1}\cdots \natural\check{S_k}) & \text{ if $n$ is odd.}
\end{cases}
$$
\end{ex}

%%%%%%%%%%%%%%%%%%%%%%%%%%%%%%%%%%%%%%%%%%%%%%%%%%%%%%%%%%%%%%%%%%%%%%%%%%%%
\subsection{The jellyfish algorithm}

The \emph{jellyfish algorithm} was invented in \cite{0909.4099} to construct the extended Haagerup subfactor planar algebra with principal graphs
$$
\left(\bigraph{bwd1v1v1v1v1v1v1v1p1v1x0p0x1v1x0p0x1duals1v1v1v1v1x2v2x1}, \bigraph{bwd1v1v1v1v1v1v1v1p1v1x0p1x0duals1v1v1v1v1x2}\right).
$$
One uses the jellyfish algorithm to evaluate closed diagrams on a set of generators. There are two ingredients:
\begin{enumerate}
\item[(1)] The generators in $P_{n,\pm}$ satisfy \emph{jellyfish relations}, i.e., for each generator $S,T$,
$$
j(\check{S})=
\jellyfishDiagram{n}{\check{S}}
\,,\, 
j^2(S')= 
\jellyfishSquared{n}{T}
$$
can be written as linear combinations of \emph{trains}, which are diagrams where any region meeting the distinguished interval of a generator meets the distinguished interval of the external disk, e.g.,
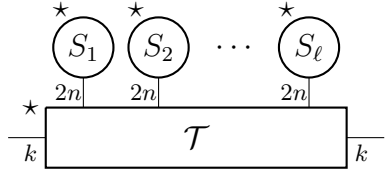
$$
\begin{tikzpicture}[baseline = -.3cm]
	\draw (-2,-1.2)--(3,-1.2);
	\draw (-1,0)--(-1,-1);
	\draw (0,0)--(0,-1);
	\draw (2,0)--(2,-1);
	\filldraw[thick, unshaded] (-1.5,-.8)--(-1.5,-1.6)--(2.5,-1.6)--(2.5,-.8)--(-1.5,-.8);
	\draw[thick, unshaded] (-1,0) circle (.4);
	\draw[thick, unshaded] (0,0) circle (.4);	
	\draw[thick, unshaded] (2,0) circle (.4);
	\node at (-1.7,-.8) {$\star$};
	\node at (-1.3,.5) {$\star$};	
	\node at (-.3,.5) {$\star$};
	\node at (1.7,.5) {$\star$};
	\node at (1,0) {$\cdots$};
	\node at (.5,-1.2) {$\cT$};
	\node at (-1,0) {$S_1$};
	\node at (0,0) {$S_2$};
	\node at (2,0) {$S_\ell$};
	\node at (-1.7,-1.4) {{\scriptsize{$k$}}};
	\node at (2.7,-1.4) {{\scriptsize{$k$}}};
	\node at (-1.2,-.6) {{\scriptsize{$2n$}}};
	\node at (-.2,-.6) {{\scriptsize{$2n$}}};
	\node at (1.8,-.6) {{\scriptsize{$2n$}}};
\end{tikzpicture}
$$
where $S_1,\dots, S_\ell$ are generators, and $\cT$ is a single Temperley-Lieb diagram.

(Note that $j(S),j(\check{S})$ means the same thing as $\cup_0(S),\cup_0(\check{S})$, but we will use the $j$ notation to emphasize its importance to the jellyfish algorithm.)
\item[(2)] The generators in $P_{n,\pm}$ are uncappable and together with the Jones-Wenzl projection $f^{(n)}$ form an algebra under the usual multiplication 
$$
ST=
\begin{tikzpicture}[baseline = .8cm]
	\draw (0,2.2)--(0,-.4);	
	\filldraw[unshaded,thick] (0,1.5) circle (.4cm);
	\node at (0,1.5) {$T$};
	\node at (-.55,1.5) {$\star$};
	\filldraw[unshaded, thick] (0,.3) circle (.4cm);
	\node at (0,.3) {$S$};
	\node at (-.55,.3) {$\star$};
	\node at (-.2,.9) {{\scriptsize{$n$}}};
	\node at (-.2,-.2) {{\scriptsize{$n$}}};
	\node at (-.2,2) {{\scriptsize{$n$}}};
\end{tikzpicture}
=
\sum_{R}
\alpha^R_{S,T}
\begin{tikzpicture}[baseline = -.1cm]
	\draw (0,.8)--(0,-.8);
	\filldraw[unshaded,thick] (0,0) circle (.4cm);
	\node at (0,0) {$R$};
	\node at (-.55,0) {$\star$};
	\node at (-.2,.6) {{\scriptsize{$n$}}};
	\node at (-.2,-.6) {{\scriptsize{$n$}}};
\end{tikzpicture}\,.
$$
(Note that the {\texttt{Mathematica}} package {\texttt{FusionAtlas}} multiplies in this order; reading from left to right in products corresponds to reading from bottom to top in planar composites.)
\end{enumerate}
Given these two ingredients, one can evaluate any closed diagram using the following two step process.
\begin{enumerate}
\item[(1)] Pull all generators $S$ to the outside of the diagram using the jellyfish relations, possibly getting diagrams with more $S$'s.
\item[(2)] Use uncappability and the algebra property to iteratively reduce the number of generators. Any non-zero train which is a closed diagram is either a Temperley-Lieb diagram, has a capped generator, or has two generators $S,T$ connected by at least $n$ strings, giving $ST$.
\end{enumerate}

Section \ref{sec:ComputingJellyfishRelations} is devoted to our procedure for computing the jellyfish relations necessary for the first part of the jellyfish algorithm, while the second part is rather easy. 

One can see that if $\{A,B,f^{(n)}\}$ span a subalgebra of $P_{n,+}$, their structure coefficients must be given by
$$
\alpha_{S,T}^R = \frac{\Tr(STR)}{\Tr(R^2)},
$$
and thus determined by the moments given in Appendix \ref{sec:moments}. We check that the algebra generated by $\{A,B,f^{(n)}\}$ is closed under multiplication directly in the graph planar algebra, in Lemma \ref{lem:check-structure-coefficients}.

Note that if we have an orthonormal set of generators, then in the notation of \cite{math/1007.1158},
$\alpha_{S,T}^R = \Tr(STR)=a^{ST}_R$.
There are similar easy calculations to determine the structure coefficients $\beta_{S,T}^R$ of the algebra generated by $\{\check{A},\check{B},\check{f}^{(n)}\}\subset P_{n,-}$, and in the orthonormal case, $\beta_{S,T}^R=b^{ST}_R$.

In \cite{0909.4099}, they found \emph{2-strand jellyfish relations}, i.e., $j(\check{S})$ and $j^2(S)$ lie in the span of the trains of $S$.
For spoke subfactors, we can find \emph{1-strand jellyfish relations} by \cite{1208.1564}, i.e., for each $S\in \fB$ and $\check{S}\in\check{\fB}$, $j(\check{S})$ lies in the span of the trains from $\fB$, and $j(S)$ lies in the span of the trains from $\check{\fB}$. 
These 1-strand jellyfish relations are sufficient to evaluate all closed diagrams from our generators, and thus $\fB$ generates some subfactor planar algebra. 

We need further arguments to prove that the resulting subfactor planar algebra has the desired principal graphs. 
It turns out that at the relatively low index of $3+\sqrt{5}$ these arguments are easy.

%%%%%%%%%%%%%%%%%%%%%%%%%%%%%%%%%%%%%%%%%%%%%%%%%%%%%%%%%%%%%%%%%%%%%%%%%
\subsection{Jellyfish to box jellyfish and back again}\label{sec:boxjellyfish}

Given a set of jellyfish relations, we may write them in a more compact form in which we multiply the diagrams by a Jones-Wenzl idempotent to get rid of simpler diagrams. We present the arguments back and forth in the case of one generator, and it is clear how to generalize to the multi-generator case.

If we know
$$
j(\check{S}) = \jellyfishDiagram{n}{\check{S}} = \gamma_{S,S}\quadraticDiagram{n}{S}{S} + \sum_{i=1}^{2n+1} \gamma_{S,i} \cup_i'(S) + X
$$
where $X\in TL_{n+1,+}$ with all strings turned down, and
$$
\cup_1'(S)= 
\begin{tikzpicture}[baseline = -.3cm]
	\draw (0,0)--(0,-.8);
	\draw (-.6,-.8) arc (180:0:.2cm);
	\node at (.2,-.6) {{\scriptsize{$2 n$}}};
	\draw[thick, unshaded] (0,0) circle (.4);
	\node at (0,0) {$S$};
	\node at (0,.55) {$\star$};
\end{tikzpicture}
\,,\,
\cup_i'(S)= 
\begin{tikzpicture}[baseline = -.3cm]
	\draw (-.2,0)--(-.6,-.8);
	\draw (.2,0)--(.6,-.8);
	\draw (-.2,-.8) arc (180:0:.2cm);
	\node at (-.9,-.6) {{\scriptsize{$i-1$}}};
	\node at (1.3,-.6) {{\scriptsize{$2n-i+1$}}};
	\draw[thick, unshaded] (0,0) circle (.4);
	\node at (0,0) {$S$};
	\node at (0,.55) {$\star$};
\end{tikzpicture}
\,,\text{ and }
\cup_{2n+1}'(S)= 
\begin{tikzpicture}[baseline = -.3cm]
	\draw (0,0)--(0,-.8);
	\draw (.2,-.8) arc (180:0:.2cm);
	\node at (-.2,-.6) {{\scriptsize{$2 n$}}};
	\draw[thick, unshaded] (0,0) circle (.4);
	\node at (0,0) {$S$};
	\node at (0,.55) {$\star$};
\end{tikzpicture}
$$
(note that the $\cup'_i(S)$'s for $i=1,\dots, 2n+1$ can be obtained from the $\cup_j(S)$'s for $j=1,\dots,2n+1$ by applying suitable powers of the rotation and multiplying by suitable powers of $\sigma_S$),
applying the Jones-Wenzl $f^{(2n+2)}$ to the bottom of the diagram gives the simpler \emph{box jellyfish relation}
$$
f\cdot j(\check{S}) = \boxjellyfishDiagram{n}{\check{S}} = \gamma_{S,S}\boxquadraticDiagram{n}{S}{S}.
$$

Conversely, given a box jellyfish relation of the above form, we get a jellyfish relation by expanding the Jones-Wenzl idempotents.
\begin{enumerate}
\item[(1)] First, note that the coefficient of the identity in the Jones-Wenzl idempotent $f^{(2n+2)}$ is 1. This term gives us a $j(\check{S})$ on the left hand side and an $S\circ S$ on the right hand side.
\item[(2)] On the left hand side, all other terms of the Jones-Wenzl either cap off the $\check{S}$, giving zero, or there is exactly one cup on the top and one cap on the bottom, and the cup on the top is between the $1$-st and $2$-nd strings or the $(2n+1)$-th and $(2n+2)$-th strings. In this case, we get a scalar multiple of an annular consequence of $S$ (which is not $j(\check{S})$), and these terms can be subtracted off to the right hand side of the equation. 
\item[(3)] On the right hand side, there are a few more options. First, if any $S$ is capped off, we get zero. Otherwise there is a cup between the $(n+1)$-th and $(n+2)$-th strings, and we get an an $S^2$ which can be written as a linear combination of $S$ and Temperley-Lieb diagrams. Now if there are more cups on top of the term in the Jones-Wenzl, the $S$ term vanishes, and we get a Temperley-Lieb diagram. If there are no extra cups, then we are left with some Temperley-Lieb diagrams and some scalar multiple of an annular consequence of $S$ (which is not $j(\check{S})$).
\end{enumerate}
Hence we get a jellyfish relation.

%%%%%%%%%%%%%%%%%%%%%%%%%%%%%%%%%%%%%%%%%%%%%%%%%%%%%%%%%%%%%%%%%%%%%%%%%
%%%%%%%%%%%%%%%%%%%%%%%%%%%%%%%%%%%%%%%%%%%%%%%%%%%%%%%%%%%%%%%%%%%%%%%%%
%%%%%%%%%%%%%%%%%%%%%%%%%%%%%%%%%%%%%%%%%%%%%%%%%%%%%%%%%%%%%%%%%%%%%%%%%
\section{Computing jellyfish relations with quadratic tangles}\label{sec:ComputingJellyfishRelations}

In Subsection \ref{sec:FindGenerators}, we explained how we obtained the subfactor planar algebra $P_\bullet^\Gamma$ generated by the flat elements at depth $n$ with respect to some connection in the graph planar algebra of $\Gamma$. We now describe how to calculate 1-strand jellyfish relations.

The calculations in these subsections, based on the techniques from \cite{math/1007.1158}, rely on knowing the cubic moments of the generators and the structure coefficients for the algebra generated by $\{A,B,f^{(n)}\}$. This requires a computer calculation, but it is no more difficult than multiplying some large matrices, with entries in a fixed number field, then taking a trace. These moments are given in Appendix \ref{sec:moments}.

We perform three calculations to derive the jellyfish relations. In Subsection \ref{sec:QTAC}, we find those linear combinations of the quadratic tangles which lie in annular consequences of the generators. In Subsection \ref{sec:JM}, we express these quadratic tangles in the basis of annular consequences of our generators. Finally in Subsection \ref{sec:JF}, we invert the relations found in Subsection \ref{sec:JM} to express the relevant annular consequences back in terms of the quadratic tangles.

\begin{remark}
In Subsections \ref{sec:QTAC} through \ref{sec:JF}, our formulas are for orthonormal sets of self-adjoint generators. However, the programs in the {\tt Mathematica} notebook {\tt QuadraticTangles.nb} are slightly more general and include the necessary correction factors allowing us to work with orthogonal generators with arbitrary norms as well.
\end{remark}

In Subsection \ref{sec:Checking}, we describe two ways that we verify our formulas. Since the computer is doing the arithmetic, we like to verify our calculations in as many ways that we can think of.

%%%%%%%%%%%%%%%%%%%%%%%%%%%%%%%%%%%%%%%%%%%%%%%%%%%%%%%%%%%%%%%%%%%%
\subsection{Identify quadratic tangles in annular consequences}\label{sec:QTAC}

Given our set of generators $\fB=\{A,B\}\subset P_{n,+}$, we have the dual generators $\check{\fB}=\{\check{A},\check{B}\}\subset P_{n,-}$, and $\fB$ and $\check{\fB}$ each give 4 quadratic tangles: 
\begin{align*}
\{A\circ A, A\circ B, B\circ A, B\circ B\}&\subset P_{n+1,+}\\\intertext{  and }
\{\check{A}\circ\check{A},\check{A}\circ\check{B},\check{B}\circ\check{A},\check{B}\circ\check{B}\}&\subset P_{n+1,-},
\end{align*}
where
$$
S\circ T = \quadraticDiagram{n}{S}{T}.
$$
Since we expect these generators to give a subfactor planar algebra with the principal graph being the underlying graph of the graph planar algebra in which they were found, we hope that some linear combinations of these quadratic tangles lie in the space $\textgoth{A}(A,B,\emptyset)$ of annular consequences of $A,B$, and the empty diagram (the annular consequences of the empty diagram are the diagrams in Temperley-Lieb), and the same for the checked generators: 
 \begin{align*}
\qtac&= \gA(A,B,\emptyset)\cap\spann(\{A\circ A, A\circ B, B\circ A, B\circ B\})\\
\qtac^\vee&= \gA(\check{A},\check{B},\emptyset)\cap \spann(\{\check{A}\circ\check{A},\check{A}\circ\check{B},\check{B}\circ\check{A},\check{B}\circ\check{B}\})
\end{align*}
($\qtac$ stands for ``quadratic tangles in annular consequences").

\begin{remark}
In our notation $\textgoth{A}(A,B,\emptyset)$ includes Temperley-Lieb. When we refer to the annular consequences of $A,B$ only, we will call this space $\gA(A,B)$ (which is what is referred to simply by $\textgoth{A}$ in \cite{math/1007.1158}).
\end{remark}

\begin{ex}
Starting with our $2$ generators found in the graph planar algebras of the 4442 or 3333 principal graphs, since 4442 and 3333 each have annular multiplicites $*22$, we hope that $\dim(\qtac)=\dim(\qtac^\vee)=2$, i.e., the quadratic tangles are as linearly independent as possible.
\end{ex} 

\begin{ex}
Starting with our $2$ generators found in the graph planar algebra of the 3311 principal graph, since 3311 has annular multiplicites $*20$, we hope that $\dim(\qtac)=\dim(\qtac^\vee)=4$.
\end{ex}

\begin{ex}
Starting with our $2$ generators found in the graph planar algebra of the 2221 principal graph, since 2221 has annular multiplicites $*21$, we hope that $\dim(\qtac)=\dim(\qtac^\vee)=3$.
\end{ex}

We use formulas from \cite{math/1007.1158} to calculate bases for $\qtac$ and $\qtac^\vee$. We first describe how to find a basis for $\qtac$.

First, we calculate the $4\times 4$ matrix of inner products modulo the annular consequences of $A,B$:
$$
\left(
\langle [S\circ T- P_{\gA(A,B)}(S\circ T)],P\circ Q\rangle 
\right)_{(S,T),(P,Q)\in\fB^2}.
$$
The inner products are given by the following formulas, where the second comes from Proposition 4.4.2 of  \cite{math/1007.1158}:
\begin{align*}
\langle S\circ T,P\circ Q\rangle &= \frac{1}{[n]}\Tr(QT)\Tr(SP)
\\
\langle P_{\gA(A,B)}S\circ T,P\circ Q\rangle&=\sum_{R\in\fB} \frac{1}{W_R}\bigg\lbrace
\left(\overline{a_R^{ST}}a_R^{PQ} + \sigma_T \overline{\sigma_S\sigma_Q}\sigma_P \overline{b_R^{ST}}b_R^{PQ}\right)
\left(\omega_R^{-1}+[2n+2]\right)\\
&+(-1)^{n+1}\sigma_R\left(\overline{\sigma_Q}\sigma_P\overline{a_R^{ST}}b_R^{PQ}+\sigma_T\overline{\sigma_S}\overline{b_R^{ST}}a_R^{PQ}\right)
\left(2\omega_R^{-1}[n+1]\right)\bigg\rbrace,
\end{align*}
where $a_R^{ST}=\Tr(RST)$, $b_R^{ST}=\Tr(\check{R}\check{S}\check{T})$, and $W_R=q^{2n+2}+q^{-2n+2}-\omega_R-\omega_R^{-1}$. Here, $q>1$ such that $\delta=[2]=q+q^{-1}$.

Taking a basis for the nullspace of this matrix gives us a basis for $\qtac$. 

\begin{remark}\label{rem:PassToDual}
To calculate a basis for $\qtac^\vee$, one passes to the dual planar algebra of the graph planar algebra (the same planar algebra, but with the shading reversed), and uses the same formulas above. 
This amounts to switching $a_R^{ST}$ and $b_R^{ST}$, since $S$ and $\check{S}$ have the same chirality for each $S\in\fB$.
\end{remark}

%%%%%%%%%%%%%%%%%%%%%%%%%%%%%%%%%%%%%%%%%%%%%%%%%%%%%%%%%%%%%%%%%%%%
\subsection{Find the jellyfish matrices}\label{sec:JM}

Now we want to write each basis element of $\qtac$ or $\qtac^\vee$ as a linear combination of the $\cup_i(R),\cup_i(\check{R})$'s for $R\in\fB,\check{R}\in\check{\fB}$ respectively. We describe the process for the basis elements of $\qtac$, and the checked versions are again computed by passing to the dual as in Remark \ref{rem:PassToDual}. Using Proposition $4.4.1.i$ in \cite{math/1007.1158}, we have
\begin{equation}\label{eqn:CIAB}
P_{\gA(A,B)}(S\circ T) = \sum_{R\in \mathfrak{B}} \sigma_R^n a_R^{ST}\widehat{\cup}_{n+1}(R) +\sigma_T^{-1}\sigma_Sb_R^{ST}\widehat{\cup}_0(R).
\end{equation}
We can express $\widehat{\cup}_{n+1}(R),\widehat{\cup}_0(R)$ (in the dual annular basis) in terms of the annular basis $\cup_i(R)$ for $R\in\fB$, using the formulas from Proposition 4.2.9 of \cite{math/1007.1158}:
\begin{align*}
\widehat{\cup}_0(R) &=\frac{1}{W_R}\left([2n+2] \cup_0(R)+((-\sigma)^{n+1}+(-\sigma)^{-n-1})[n+1]\cup_{n+1}(R)+X\right)\\
\widehat{\cup}_{n+1}(R) &=\frac{1}{W_R}\left([2n+2] \cup_{n+1}(R)+((-\sigma)^{n+1}+(-\sigma)^{-n-1})[n+1]\cup_0(R)+Y\right)
\end{align*}
where $X,Y\in \overline{\gA(R)}=\spann\set{\cup_i(R)}{i\neq 0,n+1}$.

\begin{remark}\label{rem:DualAnnularBasisToAnnularBasis}
Our calculations, available bundled with the {\tt arXiv} sources for this paper in the \code{Mathematica} notebook \code{QuadraticTangles.nb}, are slightly more complicated; we don't actually use the formulas from Proposition 4.2.9 above, but instead directly compute the change of basis matrix from the dual annular basis to the annular basis as an extra check of these formulas. The change of basis matrix is computed as follows.

First, letting $U$ and $\widehat{U}$ be the column vectors corresponding to the basis elements $\cup_i(R)$ and $\widehat{\cup}_i(R)$, there is some matrix $V\in M_{2n+2}(\Complex)$ such that $VU=\widehat{U}$. From the formulas $\langle \widehat{\cup}_i(R),\cup_j(R)\rangle = \delta_{i,j}$ and 
$$
W_{i,j}=\langle \cup_i(R),\cup_j(R)\rangle = 
\begin{cases}
\delta &\text{if }i=j\\
\sigma_R^{\pm 1} & \text{if } j=i\pm 1\\
0 &\text{else,}
\end{cases}
$$ 
we have $V=\overline{W^{-1}}$ (remember that the inner product is linear on the \emph{right}).
\end{remark}

Putting it all together, we can express $v$ in our basis of $\qtac$ as a linear combination of $\cup_0(A),\cup_0(B),\cup_{n+1}(A),\cup_{n+1}(B)$, plus another element in  
$\overline{\gA(A,B,\emptyset)}=TL_{n+1,\pm}\oplus \spann(\set{\cup_i(R)}{R\in\fB\text{ and }i\neq 0,n+1})$. 
\begin{align*}
v
&= 
\sum_{S,T\in\fB} \gamma_{S,T} \quadraticDiagram{n}{S}{T}\\
&=\mu_A\, \jellyfishDiagram{n}{\check{A}}+ \mu_B\, \jellyfishDiagram{n}{\check{B}}+\nu_A\capunderDiagram{n}{{A}^{\vee n}}+\nu_B\capunderDiagram{n}{{B}^{\vee n}}+Z,
\end{align*}
where $Z\in \overline{\gA(A,B,\emptyset)}$ and 
$$
S^{\vee n}=
\begin{cases}
S & \text{if $n$ is even}\\
\check{S} &\text{if $n$ is odd}
\end{cases}
\hspace{.5in}
\check{S}^{\vee n}=
\begin{cases}
\check{S} & \text{if $n$ is even}\\
S &\text{if $n$ is odd.}
\end{cases}
$$

Given such an equation, we can multiply by a Jones-Wenzl idempotent in two ways to find the following relations, from which we will derive the desired box jellyfish relations:
\begin{enumerate}
\item[(1)] We can multiply by $f^{(2n+2)}$ on the bottom to isolate the $\cup_{0}(A),\cup_{0}(B)$:
$$
\sum_{S,T\in\fB} \gamma_{S,T} \boxquadraticDiagram{n}{S}{T} =\mu_A\, \boxjellyfishDiagram{n}{\check{A}}+ \mu_B\, \boxjellyfishDiagram{n}{\check{B}}
$$
as any cap on top of $f^{(2n+2)}$ gives zero.

\item[(2)] We can bend $f^{(2n+2)}$ around the top to isolate the $\cup_{n+1}(A),\cup_{n+1}(B)$:
$$
\sum_{S,T\in\fB}
\gamma_{S,T}
\begin{tikzpicture}[baseline = -.1cm]
	\clip (-1.6,-1.6)--(-1.6,1.6)--(3,1.6)--(3,-1.6);
	\draw (0,0)--(1.6,0)--(1.6,-1.6);
	\draw (0,-1.6)--(0,0);
	\node at (.8,.15) {{\scriptsize{$n-1$}}};
	\node at (.4,-.5) {{\scriptsize{$n+1$}}};
	\node at (1.2,-.6) {{\scriptsize{$n+1$}}};
	\draw[thick, unshaded] (0,0) circle (.4);
	\node at (0,0) {$S$};
	\draw[thick, unshaded] (1.6,0) circle (.4);
	\node at (1.6,0) {$T$};
	\node at (1.6,.55) {$\star$};
	\node at (0,.55) {$\star$};
	\filldraw[thick, unshaded] (1.4,-.8)  .. controls ++(0:1.2cm) and ++(0:1.2cm) .. (1.4,.8)--(.2,.8).. controls ++(180:1.2cm) and ++(180:1.2cm) .. (.2,-.8)--
	(.2,-1.4).. controls ++(180:2cm) and ++(180:2cm) .. (.2,1.4)--(1.4,1.4).. controls ++(0:2cm) and ++(0:2cm) .. (1.4,-1.4)--(1.4,-.8);
	\node at (.8,1.1) {$f^{(2n+2)}$};
\end{tikzpicture}
=
\nu_A \;\jwoverDiagram{n}{{A}^{\vee n}} \;+ \;\nu_B\; \jwoverDiagram{n}{{B}^{\vee n}}
$$
which is equivalent to
$$
\sum_{S,T\in\fB}
\gamma_{S,T}\sigma_S^{1-n}\sigma_T^{n-1}
\boxquadraticDiagram{n}{\check{S}^{\vee n}}{\check{T}^{\vee n}}
=\nu_A\boxjellyfishDiagram{n}{{A}^{\vee n}}\; +\; \nu_B\boxjellyfishDiagram{n}{{B}^{\vee n}}\,.
$$
\end{enumerate} 

\begin{remark}\label{rem:ExtraBoxJellyfishRelations}
The second relation above is actually superfluous. 
It suffices to consider the relations above of the first type for both $\fB$ and $\check{\fB}$. 
On the other hand, the computer is doing the arithmetic, so we prefer to get a nice consistency check on all our formulas with little extra work.
\end{remark}

\begin{nota}
For $S,T\in\fB,\check{\fB}$, we use the notation
$$
j(S) = \jellyfishDiagram{n}{S},\, f\cdot j(S) = \boxjellyfishDiagram{n}{S}, \text{ and } f(S\circ T) = \boxquadraticDiagram{n}{S}{T}.
$$
\end{nota}

\begin{defn}
Suppose $\{v_1,\dots, v_k\}\subset \qtac, \{\check{v}_1,\dots, \check{v}_k\}\subset \qtac^\vee$ are bases (in our examples, $2\leq k\leq 4$). First, we use the above method to calculate coefficients $\mu_A^i,\mu_B^i,\check{\mu}_A^i,\check{\mu}_B^i$ so that
\begin{align*}
v_i &= \sum_{S,T\in\fB} \gamma_{S,T}^i f(S\circ T)=\mu_A^i [f\cdot j(\check{A})]+ \mu_B^i [f\cdot j(\check{B})]\text{ and}\\
\check{v}_i &= \sum_{\check{S},\check{T}\in\check{\fB}} \check{\gamma}_{{S},{T}}^i f(\check{S}\circ \check{T})= \check{\mu}_A^i [f\cdot j(A)]+ \check{\mu}_B^i [f\cdot j({B})].
\end{align*}
The \emph{jellyfish matrices} are the matrices $J,\check{J}$ whose $i$-th rows are $\begin{pmatrix} \mu_A^i & \mu_B^i \end{pmatrix}$, $\begin{pmatrix} \check{\mu}_A^i & \check{\mu}_B^i \end{pmatrix}$ respectively.
\\
The \emph{quadratic tangles matrices} are the matrices $K,\check{K}$ whose $i$-th rows are $\begin{pmatrix} \gamma_{A,A}^i & \gamma_{A,B}^i, \gamma_{B,A}^i, \gamma_{B,B}^i\end{pmatrix}$, $\begin{pmatrix} \check{\gamma}_{A,A}^i & \check{\gamma}_{A,B}^i, \check{\gamma}_{B,A}^i, \check{\gamma}_{B,B}^i\end{pmatrix}$ respectively.
\\
Note that
$$
K
\begin{pmatrix}
f(A\circ A)\\
f(A\circ B)\\
f(B\circ A)\\
f(B\circ B)
\end{pmatrix}
= J
\begin{pmatrix}
f\cdot j(\check{A})\\
f\cdot j(\check{B})
\end{pmatrix},
$$
and similarly for the checked version.
\end{defn}

%%%%%%%%%%%%%%%%%%%%%%%%%%%%%%%%%%%%%%%%%%%%%%%%%%%%%%%%%%%%%%%%%%%%
\subsection{Invert the jellyfish matrices to get box jellyfish relations}\label{sec:JF}

Given the jellyfish matrices $J,\check{J}$, we check if they have rank 2. If they do (and we know they should by \cite{1208.1564}), then we find a left inverse by the formula
$$
J^{L}= (J^*J)^{-1}J^*
$$
(and similarly for $\check{J}$), since $J^*J\in M_2(\Complex)$ has rank $2$. We then use $J^{L},\check{J}^{L}$ to get the \underline{box jellyfish relations} 
$$
\begin{pmatrix}
f\cdot j(\check{A})\\
f\cdot j(\check{B})
\end{pmatrix}
=
J^{L}K
\begin{pmatrix}
f(A\circ A)\\
f(A\circ B)\\
f(B\circ A)\\
f(B\circ B)
\end{pmatrix}
$$
which express $f\cdot j(\check{A}),f\cdot j(\check{B})$ as linear combinations of quadratic tangles, and similarly for the checked versions.

%%%%%%%%%%%%%%%%%%%%%%%%%%%%%%%%%%%%%%%%%%%%%%%%%%%%%%%%%%%%%%%%%%%%
\subsection{Checking our formulas}\label{sec:Checking}

As emphasized in Remarks \ref{rem:DualAnnularBasisToAnnularBasis} and \ref{rem:ExtraBoxJellyfishRelations}, since the computer is doing the arithmetic, we like to check our formulas in as many ways as we can think of. We perform two extra checks of the formulas obtained from the above calculations. 
\begin{enumerate}

\item[(1)] Given our set of generators $\fB$ in the graph planar algebra, we can compute the annular bases for $\gA(A)$, $\gA(B)$ and the quadratic tangles $S\circ T$ for $S,T\in\fB$ directly from the graph planar algebra. We can then use numerical linear algebra to compute an approximate basis for $\qtac$, and we can compare these results with those obtained in Subsection \ref{sec:QTAC}. We do a similar check for $\check{\fB}$ and $\qtac^\vee$.

Similarly, we can compute numerical approximations for the coefficients of the quadratic tangles in annular consequences with respect to the annular basis directly in the graph planar algebra. We compare these numbers with those computed from Equation \eqref{eqn:CIAB} using Remark \ref{rem:DualAnnularBasisToAnnularBasis}.

Finally, we can numerically find the jellyfish formulas $J^{L}K,\check{J}^{L}\check{K}$ directly from the generators in the graph planar algebra and compare them with the $J^{L}K,\check{J}^{L}\check{K}$ computed in Subsection \ref{sec:JM}. 

All of these checks are carried out in {\tt QuadraticTangles.nb}, in the ``QT Direct'' sections for the graphs 3333, 3311, and 2221.

\item[(2)]
It would be beneficial to check the actual jellyfish formulas directly in the graph planar algebra. Since the Jones-Wenzl idempotent, written in the graph planar algebra, is expensive to compute, this is only feasible for the smallest graph 2221, and even then, we need a clever trick introduced in \cite{0909.4099} by Stephen Bigelow. First, note that
$$
\begin{tikzpicture}[baseline = -.9cm]
	\clip (-.7,-1.8)--(-.7,.5)--(2.4,.5)--(2.4,-1.8);
	\draw (0,-2)--(0,-.5);
	\draw (0,-.5) arc (180:0:.8cm);
	\draw (1.6,-2)--(1.6,-.5);
	\node at (.8,.15) {{\scriptsize{$n-1$}}};
	\node at (.4,-.6) {{\scriptsize{$n+1$}}};
	\node at (2,-.6) {{\scriptsize{$n+1$}}};
	\draw[thick, unshaded] (.2,-.4)--(.2,.4)--(1.4,.4)--(1.4,-.4)--(.2,-.4);
	\node at (.8,0) {$f^{(n+1)}$};	
	\draw[thick, unshaded] (-.4,-1.5)--(-.4,-.8)--(2,-.8)--(2,-1.5)--(-.4,-1.5);
	\node at (.8,-1.15) {$f^{(2n+2)}$};
	\node at (-.55,-1.15) {$\star$};
	\node at (.05,.35) {$\star$};
\end{tikzpicture}
=
0,
$$
which implies that for $S,T\in\fB$ and all $\gamma$,
$$
f(S\circ T)=
\boxquadraticDiagram{n}{S}{T}
=
\boxquadraticDiagram{n}{S}{T}
-\gamma 
\begin{tikzpicture}[baseline = -.9cm]
	\clip (-.7,-1.8)--(-.7,.5)--(2.4,.5)--(2.4,-1.8);
	\draw (0,-2)--(0,-.5);
	\draw (0,-.5) arc (180:0:.8cm);
	\draw (1.6,-2)--(1.6,-.5);
	\node at (.8,.15) {{\scriptsize{$n-1$}}};
	\node at (.4,-.6) {{\scriptsize{$n+1$}}};
	\node at (2,-.6) {{\scriptsize{$n+1$}}};
	\draw[thick, unshaded] (.2,-.4)--(.2,.4)--(1.4,.4)--(1.4,-.4)--(.2,-.4);
	\node at (.8,0) {$f^{(n+1)}$};	
	\draw[thick, unshaded] (-.4,-1.5)--(-.4,-.8)--(2,-.8)--(2,-1.5)--(-.4,-1.5);
	\node at (.8,-1.15) {$f^{(2n+2)}$};
	\node at (-.55,-1.15) {$\star$};
	\node at (.05,.35) {$\star$};
\end{tikzpicture}.
$$
Note that $ST-\frac{\Tr(ST)}{[n+1]^2}f^{(n)}$ is uncappable as it lies in $\spann\{A,B\}$. Hence applying any two caps which do not enclose $\star$ to $S\circ T - \frac{\Tr(ST)}{[n+1][n+2]} f^{(n+1)}$ gives zero, and therefore $f(S\circ T)$ is equal to
$$
\quadraticDiagram{n}{S}{T}
+
\begin{tikzpicture}[baseline = -.9cm]
	\clip (-.7,-1.8)--(-.7,.7)--(2.4,.7)--(2.4,-1.8);
	\draw (0,-2)--(0,0)--(1.6,0)--(1.6,-2);
	\node at (.8,.15) {{\scriptsize{$n-1$}}};
	\node at (.4,-.6) {{\scriptsize{$n+1$}}};
	\node at (2,-.6) {{\scriptsize{$n+1$}}};
	\draw[thick, unshaded] (0,0) circle (.4);
	\node at (0,0) {$S$};
	\draw[thick, unshaded] (1.6,0) circle (.4);
	\node at (1.6,0) {$T$};
	\node at (1.6,.55) {$\star$};
	\node at (0,.55) {$\star$};
	\draw[thick, unshaded] (-.4,-1.5)--(-.4,-.8)--(2,-.8)--(2,-1.5)--(-.4,-1.5);
	\node at (.8,-1.15) {$f^{(2n+2)}_{1-cup}$};
	\node at (-.55,-1.15) {$\star$};
\end{tikzpicture}
-
\frac{\Tr(ST)}{[n+1][n+2]}
\begin{tikzpicture}[baseline = -.9cm]
	\clip (-.7,-1.8)--(-.7,.5)--(2.4,.5)--(2.4,-1.8);
	\draw (0,-2)--(0,-.5);
	\draw (0,-.5) arc (180:0:.8cm);
	\draw (1.6,-2)--(1.6,-.5);
	\node at (.8,.15) {{\scriptsize{$n-1$}}};
	\node at (.4,-.6) {{\scriptsize{$n+1$}}};
	\node at (2,-.6) {{\scriptsize{$n+1$}}};
	\draw[thick, unshaded] (.2,-.4)--(.2,.4)--(1.4,.4)--(1.4,-.4)--(.2,-.4);
	\node at (.8,0) {$f^{(n+1)}$};	
	\draw[thick, unshaded] (-.4,-1.5)--(-.4,-.8)--(2,-.8)--(2,-1.5)--(-.4,-1.5);
	\node at (.8,-1.15) {$f^{(2n+2)}_{1-cup}$};
	\node at (-.55,-1.15) {$\star$};
	\node at (.05,.35) {$\star$};
\end{tikzpicture}
$$
where $f^{(2n+2)}_{1-cup}$ is the sum of all terms in the Jones-Wenzl $f^{(2n+2)}$ with exactly one cup on the top and one cup on the bottom. A formula for the 1-cup Jones-Wenzl in terms of Temperley-Lieb diagrams can be deduced easily from \cite{MR2375712,morrison}:
\begin{align*}
f^{(k)}_{1-cup}
&=
-\sum_{a=0}^{k-2}
\frac{[a+1][k-a-1]}{[k]}
\begin{tikzpicture}[baseline = -.1cm]
	\draw (0,-.8)--(0,.8);
	\node at (.25,0) {\scriptsize{$\cdots$}};
	\draw (.5,-.8)--(.5,.8);
	\draw (.7,.8) arc (-180:0:.2cm);
	\draw (.7,-.8) arc (180:0:.2cm);
	\draw (1.3,-.8)--(1.3,.8);
	\node at (1.55,0) {\scriptsize{$\cdots$}};	
	\draw (1.8,-.8)--(1.8,.8);
	\node at (.25,-.4) {$\underbrace{}_{a}$};
\end{tikzpicture}
\\
+&\sum_{\substack{a+b+c=k-2\\c>0}}
(-1)^{c+1}
\frac{[a+1][b+1]}{[k]}
\left(
\begin{tikzpicture}[baseline = -.1cm]
	\draw (0,-.8)--(0,.8);
	\node at (.25,0) {\scriptsize{$\cdots$}};
	\draw (.5,-.8)--(.5,.8);
	\draw (.7,.8) arc (-180:0:.2cm);
	\draw (.7,-.8)--(1.3,.8);
	\node at (1.25,0) {\scriptsize{$\cdots$}};
	\draw (1.2,-.8)--(1.8,.8);
	\draw (1.4,-.8) arc (180:0:.2cm);
	\draw (2,-.8)--(2,.8);
	\node at (2.25,0) {\scriptsize{$\cdots$}};	
	\draw (2.5,-.8)--(2.5,.8);
	\node at (.25,-.4) {$\underbrace{}_{a}$};
	\node at (2.25,-.4) {$\underbrace{}_{b}$};
	\node at (1.15,-.4) {$\underbrace{}_{c}$};
\end{tikzpicture}
+
\begin{tikzpicture}[baseline = -.1cm]
	\draw (0,-.8)--(0,.8);
	\node at (.25,0) {\scriptsize{$\cdots$}};
	\draw (.5,-.8)--(.5,.8);
	\draw (.7,-.8) arc (180:0:.2cm);
	\draw (.7,.8)--(1.3,-.8);
	\node at (1.25,0) {\scriptsize{$\cdots$}};
	\draw (1.2,.8)--(1.8,-.8);
	\draw (1.4,.8) arc (-180:0:.2cm);
	\draw (2,-.8)--(2,.8);
	\node at (2.25,0) {\scriptsize{$\cdots$}};	
	\draw (2.5,-.8)--(2.5,.8);
	\node at (.25,-.4) {$\underbrace{}_{a}$};
	\node at (2.25,-.4) {$\underbrace{}_{b}$};
	\node at (1.35,-.4) {$\underbrace{}_{c}$};
\end{tikzpicture}
\right).
\end{align*}   
   
Note that the number of terms in the 1-cup Jones-Wenzl grows linearly, whereas the number of terms in the Jones-Wenzl idempotent, given by Catalan numbers, grows exponentially. 

Using this trick, we numerically check the jellyfish formulas calculated in Section \ref{sec:Relations2221} in the graph planar algebra of 2221.
\end{enumerate}

%%%%%%%%%%%%%%%%%%%%%%%%%%%%%%%%%%%%%%%%%%%%%%%%%%%%%%%%%%%%%%%%%%%%%%%%%
%%%%%%%%%%%%%%%%%%%%%%%%%%%%%%%%%%%%%%%%%%%%%%%%%%%%%%%%%%%%%%%%%%%%%%%%%
%%%%%%%%%%%%%%%%%%%%%%%%%%%%%%%%%%%%%%%%%%%%%%%%%%%%%%%%%%%%%%%%%%%%%%%%%
\section{Generators and relations}\label{sec:relations}

We now have a subsection for each of our subfactor planar algebras. The three lemmas in each section show the results of the calculations described above. The proofs are simply substituting the appropriate quantities (moments, chiralities, etc.) into the formulas above. You can verify all these calculations using the Mathematica notebooks included with the {\tt arXiv} sources of this paper.

Throughout, the notation  $\lambda_{a_n, \ldots, a_0}^{(z)}$ denotes the root of the polynomial $\sum_i a_i x^i$ which is closest to the approximate real number $z$. (The digits of precision of $z$ are in each case chosen so that this unambiguously identifies the root.) Thus for example $\lambda^{(0.3278)}_{1024,0,-864,0,81}$ denotes the root of $1024 x^4 - 864 x^2 + 81$ which is closest to $0.3278$.

In this section, $A,B$ denote the specific generators denoted by $A,B$ given in the appendix for each respective graph.

\begin{lem}
\label{lem:check-structure-coefficients}
For each of the graphs $\Gamma=4442,3333,3311$ or $2221$, the elements $A, B$ and $f^{(n)}$ in the $(n,+)$-box space of the graph planar algebra are closed under multiplication, and their structure coefficients
$$
ST = \alpha_{S,T}^A A+\alpha_{S,T}^B B +\alpha_{S,T}^{f} f^{(n)}
$$
are given the following ratio of moments:
$$
\alpha_{S,T}^R = \frac{\Tr(STR)}{\Tr(R^2)}.
$$
A similar result holds for the elements $\check{A},\check{B}$ and $f^{(n)}$ in the $(n,-)$-box space.
\end{lem}
\begin{proof}
The program {\tt VerifyClosedUnderMultiplication} in the {\tt Mathematica} notebook {\tt QuadraticTangles.nb} verifies that the algebra generated by the set $\{f^{(n)}, A, B\}$ is closed under multiplication directly in the graph planar algebra. Once we know this, the formula claimed for the structure coefficients follows by taking inner products.
\end{proof}

\subsection{4442}\label{sec:Relations4442}
% generated by the command 'PrepareLemmas[{BigraphWithDuals[GradedBigraph[{{1}}, {{1}}, {{1}}, {{1}}, {{1}, {1}, {1}}, {{1, 0, 0}, {0, 1, 0}, {0, 0, 1}}, {{1, 0, 0}, {0, 1, 0}}, {{1, 0}, {0, 1}}], DualData[{1}, {1}, {1}, {2, 1, 3}, {2, 1}]], BigraphWithDuals[GradedBigraph[{{1}}, {{1}}, {{1}}, {{1}}, {{1}, {1}, {1}}, {{1, 0, 0}, {0, 1, 0}, {0, 0, 1}}, {{1, 0, 0}, {0, 1, 0}}, {{1, 0}, {0, 1}}], DualData[{1}, {1}, {1}, {2, 1, 3}, {2, 1}]]}, 1]' in the Mathematica notebook code/GenerateLaTeX.nb available with the arXiv sources of this article.
\begin{lem}
The linear combinations
$$K\begin{pmatrix}
A\circ A\\
A\circ B\\
B\circ A\\
B\circ B
\end{pmatrix}
\text{ and }
\check{K}
\begin{pmatrix}
\check{A}\circ\check{A}\\
\check{A}\circ\check{B}\\
\check{B}\circ\check{A}\\
\check{B}\circ\check{B}
\end{pmatrix}
$$
lie in annular consequences, where
% generated by the command 'TypesetQuadraticTanglesInAnnularConsequencesMatrices[{BigraphWithDuals[GradedBigraph[{{1}}, {{1}}, {{1}}, {{1}}, {{1}, {1}, {1}}, {{1, 0, 0}, {0, 1, 0}, {0, 0, 1}}, {{1, 0, 0}, {0, 1, 0}}, {{1, 0}, {0, 1}}], DualData[{1}, {1}, {1}, {2, 1, 3}, {2, 1}]], BigraphWithDuals[GradedBigraph[{{1}}, {{1}}, {{1}}, {{1}}, {{1}, {1}, {1}}, {{1, 0, 0}, {0, 1, 0}, {0, 0, 1}}, {{1, 0, 0}, {0, 1, 0}}, {{1, 0}, {0, 1}}], DualData[{1}, {1}, {1}, {2, 1, 3}, {2, 1}]]}, 1]' in the Mathematica notebook code/GenerateLaTeX.nb available with the arXiv sources of this article.
\begin{align*}
K &= \begin{pmatrix}
1 & 0 & 0 & -1 \\[6pt]
0 & 1 & \lambda_{1,-1,1,-1,1}^{(0.809+0.588 i)} & 0\end{pmatrix}\\
\intertext{and}\check{K} &= \begin{pmatrix}
1 & 0 & 0 & -1 \\[6pt]
0 & 1 & \lambda_{1,-1,1,-1,1}^{(0.809+0.588 i)} & 0\end{pmatrix}.\\
\end{align*}
\end{lem}

\begin{lem}
In particular, we have
$$
K
\begin{pmatrix}
f(A\circ A)\\
f(A\circ B)\\
f(B\circ A)\\
f(B\circ B)
\end{pmatrix}
=J
\begin{pmatrix}
f\cdot j(\check{A})\\
f\cdot j(\check{B})
\end{pmatrix}
\text{ and }
\check{K}
\begin{pmatrix}
f(\check{A}\circ\check{A})\\
f(\check{A}\circ\check{B})\\
f(\check{B}\circ\check{A})\\
f(\check{B}\circ\check{B})
\end{pmatrix}
=
\check{J}
\begin{pmatrix}
f\cdot j(A)\\
f\cdot j(B)
\end{pmatrix},$$
where
% generated by the command 'TypesetJellyfishMatrices[{BigraphWithDuals[GradedBigraph[{{1}}, {{1}}, {{1}}, {{1}}, {{1}, {1}, {1}}, {{1, 0, 0}, {0, 1, 0}, {0, 0, 1}}, {{1, 0, 0}, {0, 1, 0}}, {{1, 0}, {0, 1}}], DualData[{1}, {1}, {1}, {2, 1, 3}, {2, 1}]], BigraphWithDuals[GradedBigraph[{{1}}, {{1}}, {{1}}, {{1}}, {{1}, {1}, {1}}, {{1, 0, 0}, {0, 1, 0}, {0, 0, 1}}, {{1, 0, 0}, {0, 1, 0}}, {{1, 0}, {0, 1}}], DualData[{1}, {1}, {1}, {2, 1, 3}, {2, 1}]]}, 1]' in the Mathematica notebook code/GenerateLaTeX.nb available with the arXiv sources of this article.
\begin{align*}
J &= \begin{pmatrix}
\lambda_{109,0,-5770,0,25}^{(7.275)} & \lambda_{109,0,-5770,0,25}^{(-7.275)} \\[6pt]
\lambda_{\begin{subarray}{l}11881,\\0,\\-966285,\\0,\\30007665,\\0,\\1366875,\\0,\\164025\\\mbox{}\end{subarray}}^{(6.745+2.191 i)} & \lambda_{\begin{subarray}{l}11881,\\0,\\-966285,\\0,\\30007665,\\0,\\1366875,\\0,\\164025\\\mbox{}\end{subarray}}^{(6.745+2.191 i)}\end{pmatrix}\\
\intertext{and}
\check{J} &= \begin{pmatrix}
\lambda_{109,0,-5770,0,25}^{(7.275)} & \lambda_{109,0,-5770,0,25}^{(-7.275)} \\[6pt]
\lambda_{\begin{subarray}{l}11881,\\0,\\-966285,\\0,\\30007665,\\0,\\1366875,\\0,\\164025\\\mbox{}\end{subarray}}^{(6.745+2.191 i)} & \lambda_{\begin{subarray}{l}11881,\\0,\\-966285,\\0,\\30007665,\\0,\\1366875,\\0,\\164025\\\mbox{}\end{subarray}}^{(6.745+2.191 i)}\end{pmatrix}.\\
\end{align*}
\end{lem}

\begin{lem}
The elements $A$ and $B$ satisfy the box jellyfish relations
$$
\begin{pmatrix}
f\cdot j(\check{A})\\
f\cdot j(\check{B})
\end{pmatrix}
=J^{L}K
\begin{pmatrix}
f(A\circ A)\\
f(A\circ B)\\
f(B\circ A)\\
f(B\circ B)
\end{pmatrix}
\text{ and }
\begin{pmatrix}
f\cdot j(A)\\
f\cdot j(B)
\end{pmatrix}
=
\check{J}^{L}\check{K}
\begin{pmatrix}
f(\check{A}\circ\check{A})\\
f(\check{A}\circ\check{B})\\
f(\check{B}\circ\check{A})\\
f(\check{B}\circ\check{B})
\end{pmatrix}
$$
where
% generated by the command 'TypesetJellyfishFormulasMatrices[{BigraphWithDuals[GradedBigraph[{{1}}, {{1}}, {{1}}, {{1}}, {{1}, {1}, {1}}, {{1, 0, 0}, {0, 1, 0}, {0, 0, 1}}, {{1, 0, 0}, {0, 1, 0}}, {{1, 0}, {0, 1}}], DualData[{1}, {1}, {1}, {2, 1, 3}, {2, 1}]], BigraphWithDuals[GradedBigraph[{{1}}, {{1}}, {{1}}, {{1}}, {{1}, {1}, {1}}, {{1, 0, 0}, {0, 1, 0}, {0, 0, 1}}, {{1, 0, 0}, {0, 1, 0}}, {{1, 0}, {0, 1}}], DualData[{1}, {1}, {1}, {2, 1, 3}, {2, 1}]]}, 1]' in the Mathematica notebook code/GenerateLaTeX.nb available with the arXiv sources of this article.
\begin{align*}
J^{L}K &= \begin{pmatrix}
\lambda_{400,0,-23080,0,109}^{(0.06872)} & \lambda_{\begin{subarray}{l}41990400,\\0,\\87480000,\\0,\\480122640,\\0,\\-3865140,\\0,\\11881\\\mbox{}\end{subarray}}^{(0.067054-0.021787 i)} & \lambda_{\begin{subarray}{l}41990400,\\0,\\87480000,\\0,\\480122640,\\0,\\-3865140,\\0,\\11881\\\mbox{}\end{subarray}}^{(0.067054+0.021787 i)} & \lambda_{400,0,-23080,0,109}^{(-0.06872)} \\[6pt]
\lambda_{400,0,-23080,0,109}^{(-0.06872)} & \lambda_{\begin{subarray}{l}41990400,\\0,\\87480000,\\0,\\480122640,\\0,\\-3865140,\\0,\\11881\\\mbox{}\end{subarray}}^{(0.067054-0.021787 i)} & \lambda_{\begin{subarray}{l}41990400,\\0,\\87480000,\\0,\\480122640,\\0,\\-3865140,\\0,\\11881\\\mbox{}\end{subarray}}^{(0.067054+0.021787 i)} & \lambda_{400,0,-23080,0,109}^{(0.06872)}\end{pmatrix}\\
\intertext{and}
\check{J}^{L}\check{K} &= \begin{pmatrix}
\lambda_{400,0,-23080,0,109}^{(0.06872)} & \lambda_{\begin{subarray}{l}41990400,\\0,\\87480000,\\0,\\480122640,\\0,\\-3865140,\\0,\\11881\\\mbox{}\end{subarray}}^{(0.067054-0.021787 i)} & \lambda_{\begin{subarray}{l}41990400,\\0,\\87480000,\\0,\\480122640,\\0,\\-3865140,\\0,\\11881\\\mbox{}\end{subarray}}^{(0.067054+0.021787 i)} & \lambda_{400,0,-23080,0,109}^{(-0.06872)} \\[6pt]
\lambda_{400,0,-23080,0,109}^{(-0.06872)} & \lambda_{\begin{subarray}{l}41990400,\\0,\\87480000,\\0,\\480122640,\\0,\\-3865140,\\0,\\11881\\\mbox{}\end{subarray}}^{(0.067054-0.021787 i)} & \lambda_{\begin{subarray}{l}41990400,\\0,\\87480000,\\0,\\480122640,\\0,\\-3865140,\\0,\\11881\\\mbox{}\end{subarray}}^{(0.067054+0.021787 i)} & \lambda_{400,0,-23080,0,109}^{(0.06872)}\end{pmatrix}.\\
\end{align*}
\end{lem}

\subsection{3333}
% generated by the command 'PrepareLemmas[{BigraphWithDuals[GradedBigraph[{{1}}, {{1}}, {{1}}, {{1}, {1}, {1}}, {{1, 0, 0}, {0, 1, 0}, {0, 0, 1}}, {{1, 0, 0}, {0, 1, 0}, {0, 0, 1}}], DualData[{1}, {1}, {1, 2, 3}, {1, 2, 3}]], BigraphWithDuals[GradedBigraph[{{1}}, {{1}}, {{1}}, {{1}, {1}, {1}}, {{1, 0, 0}, {0, 1, 0}, {0, 0, 1}}, {{1, 0, 0}, {0, 1, 0}, {0, 0, 1}}], DualData[{1}, {1}, {1, 2, 3}, {1, 2, 3}]]}, 2]' in the Mathematica notebook code/GenerateLaTeX.nb available with the arXiv sources of this article.
\begin{lem}
The linear combinations
$$K\begin{pmatrix}
A\circ A\\
A\circ B\\
B\circ A\\
B\circ B
\end{pmatrix}
\text{ and }
\check{K}
\begin{pmatrix}
\check{A}\circ\check{A}\\
\check{A}\circ\check{B}\\
\check{B}\circ\check{A}\\
\check{B}\circ\check{B}
\end{pmatrix}
$$
lie in annular consequences, where
% generated by the command 'TypesetQuadraticTanglesInAnnularConsequencesMatrices[{BigraphWithDuals[GradedBigraph[{{1}}, {{1}}, {{1}}, {{1}, {1}, {1}}, {{1, 0, 0}, {0, 1, 0}, {0, 0, 1}}, {{1, 0, 0}, {0, 1, 0}, {0, 0, 1}}], DualData[{1}, {1}, {1, 2, 3}, {1, 2, 3}]], BigraphWithDuals[GradedBigraph[{{1}}, {{1}}, {{1}}, {{1}, {1}, {1}}, {{1, 0, 0}, {0, 1, 0}, {0, 0, 1}}, {{1, 0, 0}, {0, 1, 0}, {0, 0, 1}}], DualData[{1}, {1}, {1, 2, 3}, {1, 2, 3}]]}, 2]' in the Mathematica notebook code/GenerateLaTeX.nb available with the arXiv sources of this article.
\begin{align*}
K &= \begin{pmatrix}
1 & 0 & 0 & \frac{1}{6} \left(-3-\sqrt{5}\right) \\
0 & 1 & 1 & 0\end{pmatrix}\\
\intertext{and}\check{K} &= \begin{pmatrix}
1 & 0 & 0 & \frac{1}{6} \left(-3-\sqrt{5}\right) \\
0 & 1 & 1 & 0\end{pmatrix}.\\
\end{align*}
\end{lem}

\begin{lem}
In particular, we have
$$
K
\begin{pmatrix}
f(A\circ A)\\
f(A\circ B)\\
f(B\circ A)\\
f(B\circ B)
\end{pmatrix}
=J
\begin{pmatrix}
f\cdot j(\check{A})\\
f\cdot j(\check{B})
\end{pmatrix}
\text{ and }
\check{K}
\begin{pmatrix}
f(\check{A}\circ\check{A})\\
f(\check{A}\circ\check{B})\\
f(\check{B}\circ\check{A})\\
f(\check{B}\circ\check{B})
\end{pmatrix}
=
\check{J}
\begin{pmatrix}
f\cdot j(A)\\
f\cdot j(B)
\end{pmatrix},$$
where
% generated by the command 'TypesetJellyfishMatrices[{BigraphWithDuals[GradedBigraph[{{1}}, {{1}}, {{1}}, {{1}, {1}, {1}}, {{1, 0, 0}, {0, 1, 0}, {0, 0, 1}}, {{1, 0, 0}, {0, 1, 0}, {0, 0, 1}}], DualData[{1}, {1}, {1, 2, 3}, {1, 2, 3}]], BigraphWithDuals[GradedBigraph[{{1}}, {{1}}, {{1}}, {{1}, {1}, {1}}, {{1, 0, 0}, {0, 1, 0}, {0, 0, 1}}, {{1, 0, 0}, {0, 1, 0}, {0, 0, 1}}], DualData[{1}, {1}, {1, 2, 3}, {1, 2, 3}]]}, 2]' in the Mathematica notebook code/GenerateLaTeX.nb available with the arXiv sources of this article.
\begin{align*}
J &= \begin{pmatrix}
\frac{1}{8} \left(\sqrt{5}-1\right) & \frac{1}{4} \left(-2-\sqrt{5}\right) \\
\frac{1}{8} \left(-3-3 \sqrt{5}\right) & \frac{1}{8} \left(1-\sqrt{5}\right)\end{pmatrix}\\
\intertext{and}
\check{J} &= \begin{pmatrix}
\lambda_{1024,0,-1344,0,121}^{(-1.102)} & \lambda_{1024,0,-96,0,1}^{(-0.2860)} \\
\lambda_{1024,0,-864,0,81}^{(-0.3278)} & \lambda_{1024,0,-1344,0,121}^{(1.102)}\end{pmatrix}.\\
\end{align*}
\end{lem}

\begin{lem}
The elements $A$ and $B$ satisfy the box jellyfish relations
$$
\begin{pmatrix}
f\cdot j(\check{A})\\
f\cdot j(\check{B})
\end{pmatrix}
=J^{L}K
\begin{pmatrix}
f(A\circ A)\\
f(A\circ B)\\
f(B\circ A)\\
f(B\circ B)
\end{pmatrix}
\text{ and }
\begin{pmatrix}
f\cdot j(A)\\
f\cdot j(B)
\end{pmatrix}
=
\check{J}^{L}\check{K}
\begin{pmatrix}
f(\check{A}\circ\check{A})\\
f(\check{A}\circ\check{B})\\
f(\check{B}\circ\check{A})\\
f(\check{B}\circ\check{B})
\end{pmatrix}
$$
where
% generated by the command 'TypesetJellyfishFormulasMatrices[{BigraphWithDuals[GradedBigraph[{{1}}, {{1}}, {{1}}, {{1}, {1}, {1}}, {{1, 0, 0}, {0, 1, 0}, {0, 0, 1}}, {{1, 0, 0}, {0, 1, 0}, {0, 0, 1}}], DualData[{1}, {1}, {1, 2, 3}, {1, 2, 3}]], BigraphWithDuals[GradedBigraph[{{1}}, {{1}}, {{1}}, {{1}, {1}, {1}}, {{1, 0, 0}, {0, 1, 0}, {0, 0, 1}}, {{1, 0, 0}, {0, 1, 0}, {0, 0, 1}}], DualData[{1}, {1}, {1, 2, 3}, {1, 2, 3}]]}, 2]' in the Mathematica notebook code/GenerateLaTeX.nb available with the arXiv sources of this article.
\begin{align*}
J^{L}K &= \begin{pmatrix}
\frac{1}{2} \left(\sqrt{5}-2\right) & \frac{1}{4} \left(-1-\sqrt{5}\right) & \frac{1}{4} \left(-1-\sqrt{5}\right) & \frac{1}{12} \left(1-\sqrt{5}\right) \\
\frac{1}{4} \left(3-3 \sqrt{5}\right) & \frac{1}{2} \left(2-\sqrt{5}\right) & \frac{1}{2} \left(2-\sqrt{5}\right) & \frac{1}{4} \left(1+\sqrt{5}\right)\end{pmatrix}\\
\intertext{and}
\check{J}^{L}\check{K} &= \begin{pmatrix}
\lambda_{64,0,-216,0,121}^{(-0.8422)} & \lambda_{64,0,-24,0,1}^{(-0.2185)} & \lambda_{64,0,-24,0,1}^{(-0.2185)} & \lambda_{5184,0,-3024,0,121}^{(0.7349)} \\
\lambda_{64,0,-1296,0,81}^{(-0.2504)} & \lambda_{64,0,-216,0,121}^{(0.8422)} & \lambda_{64,0,-216,0,121}^{(0.8422)} & \lambda_{64,0,-24,0,1}^{(0.2185)}\end{pmatrix}.\\
\end{align*}
\end{lem}

\subsection{3311}
% generated by the command 'PrepareLemmas[{BigraphWithDuals[GradedBigraph[{{1}}, {{1}}, {{1}}, {{1}, {1}, {1}}, {{1, 0, 0}}, {{1}}], DualData[{1}, {1}, {1, 2, 3}, {1}]], BigraphWithDuals[GradedBigraph[{{1}}, {{1}}, {{1}}, {{1}, {1}, {1}}, {{1, 0, 0}}, {{1}}], DualData[{1}, {1}, {1, 2, 3}, {1}]]}, 2]' in the Mathematica notebook code/GenerateLaTeX.nb available with the arXiv sources of this article.
\begin{lem}
The linear combinations
$$K\begin{pmatrix}
A\circ A\\
A\circ B\\
B\circ A\\
B\circ B
\end{pmatrix}
\text{ and }
\check{K}
\begin{pmatrix}
\check{A}\circ\check{A}\\
\check{A}\circ\check{B}\\
\check{B}\circ\check{A}\\
\check{B}\circ\check{B}
\end{pmatrix}
$$
lie in annular consequences, where
% generated by the command 'TypesetQuadraticTanglesInAnnularConsequencesMatrices[{BigraphWithDuals[GradedBigraph[{{1}}, {{1}}, {{1}}, {{1}, {1}, {1}}, {{1, 0, 0}}, {{1}}], DualData[{1}, {1}, {1, 2, 3}, {1}]], BigraphWithDuals[GradedBigraph[{{1}}, {{1}}, {{1}}, {{1}, {1}, {1}}, {{1, 0, 0}}, {{1}}], DualData[{1}, {1}, {1, 2, 3}, {1}]]}, 2]' in the Mathematica notebook code/GenerateLaTeX.nb available with the arXiv sources of this article.
\begin{align*}
K &= \begin{pmatrix}
1 & 0 & 0 & 0 \\[6pt]
0 & 1 & 0 & 0 \\[6pt]
0 & 0 & 1 & 0 \\[6pt]
0 & 0 & 0 & 1\end{pmatrix}\\
\intertext{and}\check{K} &= \begin{pmatrix}
1 & 0 & 0 & 0 \\[6pt]
0 & 1 & 0 & 0 \\[6pt]
0 & 0 & 1 & 0 \\[6pt]
0 & 0 & 0 & 1\end{pmatrix}.\\
\end{align*}
\end{lem}

\begin{lem}
In particular, we have
$$
K
\begin{pmatrix}
f(A\circ A)\\
f(A\circ B)\\
f(B\circ A)\\
f(B\circ B)
\end{pmatrix}
=J
\begin{pmatrix}
f\cdot j(\check{A})\\
f\cdot j(\check{B})
\end{pmatrix}
\text{ and }
\check{K}
\begin{pmatrix}
f(\check{A}\circ\check{A})\\
f(\check{A}\circ\check{B})\\
f(\check{B}\circ\check{A})\\
f(\check{B}\circ\check{B})
\end{pmatrix}
=
\check{J}
\begin{pmatrix}
f\cdot j(A)\\
f\cdot j(B)
\end{pmatrix},$$
where
% generated by the command 'TypesetJellyfishMatrices[{BigraphWithDuals[GradedBigraph[{{1}}, {{1}}, {{1}}, {{1}, {1}, {1}}, {{1, 0, 0}}, {{1}}], DualData[{1}, {1}, {1, 2, 3}, {1}]], BigraphWithDuals[GradedBigraph[{{1}}, {{1}}, {{1}}, {{1}, {1}, {1}}, {{1, 0, 0}}, {{1}}], DualData[{1}, {1}, {1, 2, 3}, {1}]]}, 2]' in the Mathematica notebook code/GenerateLaTeX.nb available with the arXiv sources of this article.
\begin{align*}
J &= \begin{pmatrix}
0 & \frac{1}{33} \left(-7-3 \sqrt{3}\right) \\[6pt]
\lambda_{144,0,312,0,121}^{(-0.7113 i)} & \frac{1}{6} \left(1-\sqrt{3}\right) \\[6pt]
\lambda_{144,0,312,0,121}^{(0.7113 i)} & \frac{1}{6} \left(1-\sqrt{3}\right) \\[6pt]
0 & \frac{1}{12} \left(7 \sqrt{3}-9\right)\end{pmatrix}\\
\intertext{and}
\check{J} &= \begin{pmatrix}
\lambda_{7776,0,-3672,0,121}^{(-0.1888)} & \lambda_{58806,0,-486,0,1}^{(-0.0621864)} \\[6pt]
\lambda_{1536,0,1632,0,121}^{(-0.2832 i)} & \lambda_{864,0,216,0,1}^{(0.4953 i)} \\[6pt]
\lambda_{1536,0,1632,0,121}^{(0.2832 i)} & \lambda_{864,0,216,0,1}^{(-0.4953 i)} \\[6pt]
\lambda_{24576,0,-2414976,0,1771561}^{(0.860)} & \lambda_{1536,0,-1632,0,121}^{(0.2832)}\end{pmatrix}.\\
\end{align*}
\end{lem}

\begin{lem}
The elements $A$ and $B$ satisfy the box jellyfish relations
$$
\begin{pmatrix}
f\cdot j(\check{A})\\
f\cdot j(\check{B})
\end{pmatrix}
=J^{L}K
\begin{pmatrix}
f(A\circ A)\\
f(A\circ B)\\
f(B\circ A)\\
f(B\circ B)
\end{pmatrix}
\text{ and }
\begin{pmatrix}
f\cdot j(A)\\
f\cdot j(B)
\end{pmatrix}
=
\check{J}^{L}\check{K}
\begin{pmatrix}
f(\check{A}\circ\check{A})\\
f(\check{A}\circ\check{B})\\
f(\check{B}\circ\check{A})\\
f(\check{B}\circ\check{B})
\end{pmatrix}
$$
where
% generated by the command 'TypesetJellyfishFormulasMatrices[{BigraphWithDuals[GradedBigraph[{{1}}, {{1}}, {{1}}, {{1}, {1}, {1}}, {{1, 0, 0}}, {{1}}], DualData[{1}, {1}, {1, 2, 3}, {1}]], BigraphWithDuals[GradedBigraph[{{1}}, {{1}}, {{1}}, {{1}, {1}, {1}}, {{1, 0, 0}}, {{1}}], DualData[{1}, {1}, {1, 2, 3}, {1}]]}, 2]' in the Mathematica notebook code/GenerateLaTeX.nb available with the arXiv sources of this article.
\begin{align*}
J^{L}K &= \begin{pmatrix}
0 & \lambda_{121,0,78,0,9}^{(0.7029 i)} & \lambda_{121,0,78,0,9}^{(-0.7029 i)} & 0 \\[6pt]
\frac{-137208-77672 \sqrt{3}}{172163} & \frac{-33708-32332 \sqrt{3}}{172163} & \frac{-33708-32332 \sqrt{3}}{172163} & \frac{52626+80142 \sqrt{3}}{172163}\end{pmatrix}\\
\intertext{and}
\check{J}^{L}\check{K} &= \begin{pmatrix}
\lambda_{\begin{subarray}{l}1559184260929,\\0,\\-65569974336,\\0,\\497664\\\mbox{}\end{subarray}}^{(-0.205052)} & \lambda_{3872,0,648,0,27}^{(0.27984 i)} & \lambda_{3872,0,648,0,27}^{(-0.27984 i)} & \lambda_{\begin{subarray}{l}1559184260929,\\0,\\-1362157739172,\\0,\\2305430424\\\mbox{}\end{subarray}}^{(0.93378)} \\[6pt]
\lambda_{\begin{subarray}{l}1559184260929,\\0,\\-23898282768,\\0,\\33872256\\\mbox{}\end{subarray}}^{(-0.11725)} & \lambda_{512,0,4896,0,3267}^{(-0.850 i)} & \lambda_{512,0,4896,0,3267}^{(0.850 i)} & \lambda_{\begin{subarray}{l}3118368521858,\\0,\\-1989820501362,\\0,\\313826716467\\\mbox{}\end{subarray}}^{(0.53393)}\end{pmatrix}.\\
\end{align*}
\end{lem}

\subsection{2221}\label{sec:Relations2221}
% generated by the command 'PrepareLemmas[{BigraphWithDuals[GradedBigraph[{{1}}, {{1}}, {{1}, {1}, {1}}, {{1, 0, 0}, {0, 1, 0}}], DualData[{1}, {1}, {2, 1}]], BigraphWithDuals[GradedBigraph[{{1}}, {{1}}, {{1}, {1}, {1}}, {{1, 0, 0}, {0, 1, 0}}], DualData[{1}, {1}, {2, 1}]]}, 1]' in the Mathematica notebook code/GenerateLaTeX.nb available with the arXiv sources of this article.
\begin{lem}
The linear combinations
$$K\begin{pmatrix}
A\circ A\\
A\circ B\\
B\circ A\\
B\circ B
\end{pmatrix}
\text{ and }
\check{K}
\begin{pmatrix}
\check{A}\circ\check{A}\\
\check{A}\circ\check{B}\\
\check{B}\circ\check{A}\\
\check{B}\circ\check{B}
\end{pmatrix}
$$
lie in annular consequences, where
% generated by the command 'TypesetQuadraticTanglesInAnnularConsequencesMatrices[{BigraphWithDuals[GradedBigraph[{{1}}, {{1}}, {{1}, {1}, {1}}, {{1, 0, 0}, {0, 1, 0}}], DualData[{1}, {1}, {2, 1}]], BigraphWithDuals[GradedBigraph[{{1}}, {{1}}, {{1}, {1}, {1}}, {{1, 0, 0}, {0, 1, 0}}], DualData[{1}, {1}, {2, 1}]]}, 1]' in the Mathematica notebook code/GenerateLaTeX.nb available with the arXiv sources of this article.
\begin{align*}
K &= \begin{pmatrix}
1 & 0 & 0 & \frac{1}{50} \left(-23-7 \sqrt{21}\right) \\[6pt]
0 & 1 & 0 & \lambda_{\begin{subarray}{l}625,\\0,\\2300,\\0,\\10464,\\0,\\-7360,\\0,\\6400\\\mbox{}\end{subarray}}^{(1.050-1.818 i)} \\[6pt]
0 & 0 & 1 & \lambda_{\begin{subarray}{l}625,\\0,\\2300,\\0,\\10464,\\0,\\-7360,\\0,\\6400\\\mbox{}\end{subarray}}^{(1.050+1.818 i)}\end{pmatrix}\\
\intertext{and}\check{K} &= \begin{pmatrix}
1 & 0 & 0 & \frac{1}{50} \left(-23-7 \sqrt{21}\right) \\[6pt]
0 & 1 & 0 & \lambda_{\begin{subarray}{l}625,\\0,\\2300,\\0,\\10464,\\0,\\-7360,\\0,\\6400\\\mbox{}\end{subarray}}^{(1.050-1.818 i)} \\[6pt]
0 & 0 & 1 & \lambda_{\begin{subarray}{l}625,\\0,\\2300,\\0,\\10464,\\0,\\-7360,\\0,\\6400\\\mbox{}\end{subarray}}^{(1.050+1.818 i)}\end{pmatrix}.\\
\end{align*}
\end{lem}

\begin{lem}
In particular, we have
$$
K
\begin{pmatrix}
f(A\circ A)\\
f(A\circ B)\\
f(B\circ A)\\
f(B\circ B)
\end{pmatrix}
=J
\begin{pmatrix}
f\cdot j(\check{A})\\
f\cdot j(\check{B})
\end{pmatrix}
\text{ and }
\check{K}
\begin{pmatrix}
f(\check{A}\circ\check{A})\\
f(\check{A}\circ\check{B})\\
f(\check{B}\circ\check{A})\\
f(\check{B}\circ\check{B})
\end{pmatrix}
=
\check{J}
\begin{pmatrix}
f\cdot j(A)\\
f\cdot j(B)
\end{pmatrix},$$
where
% generated by the command 'TypesetJellyfishMatrices[{BigraphWithDuals[GradedBigraph[{{1}}, {{1}}, {{1}, {1}, {1}}, {{1, 0, 0}, {0, 1, 0}}], DualData[{1}, {1}, {2, 1}]], BigraphWithDuals[GradedBigraph[{{1}}, {{1}}, {{1}, {1}, {1}}, {{1, 0, 0}, {0, 1, 0}}], DualData[{1}, {1}, {2, 1}]]}, 1]' in the Mathematica notebook code/GenerateLaTeX.nb available with the arXiv sources of this article.
\begin{align*}
J &= \begin{pmatrix}
\frac{1}{3} \left(-6-\sqrt{21}\right) & \lambda_{225,0,-393,0,-5}^{(1.326)} \\[6pt]
\lambda_{\begin{subarray}{l}2025,\\0,\\90855,\\0,\\1616571,\\0,\\931200,\\0,\\160000\\\mbox{}\end{subarray}}^{(1.680-4.996 i)} & \lambda_{9,9,12,-3,1}^{(-0.6319-1.0945 i)} \\[6pt]
\lambda_{\begin{subarray}{l}2025,\\0,\\90855,\\0,\\1616571,\\0,\\931200,\\0,\\160000\\\mbox{}\end{subarray}}^{(1.680+4.996 i)} & \lambda_{9,9,12,-3,1}^{(-0.6319+1.0945 i)}\end{pmatrix}\\
\intertext{and}
\check{J} &= \begin{pmatrix}
\frac{1}{3} \left(-6-\sqrt{21}\right) & \lambda_{225,0,-393,0,-5}^{(1.326)} \\[6pt]
\lambda_{\begin{subarray}{l}2025,\\0,\\90855,\\0,\\1616571,\\0,\\931200,\\0,\\160000\\\mbox{}\end{subarray}}^{(1.680-4.996 i)} & \lambda_{9,9,12,-3,1}^{(-0.6319-1.0945 i)} \\[6pt]
\lambda_{\begin{subarray}{l}2025,\\0,\\90855,\\0,\\1616571,\\0,\\931200,\\0,\\160000\\\mbox{}\end{subarray}}^{(1.680+4.996 i)} & \lambda_{9,9,12,-3,1}^{(-0.6319+1.0945 i)}\end{pmatrix}.\\
\end{align*}
\end{lem}

\begin{lem}
The elements $A$ and $B$ satisfy the box jellyfish relations
$$
\begin{pmatrix}
f\cdot j(\check{A})\\
f\cdot j(\check{B})
\end{pmatrix}
=J^{L}K
\begin{pmatrix}
f(A\circ A)\\
f(A\circ B)\\
f(B\circ A)\\
f(B\circ B)
\end{pmatrix}
\text{ and }
\begin{pmatrix}
f\cdot j(A)\\
f\cdot j(B)
\end{pmatrix}
=
\check{J}^{L}\check{K}
\begin{pmatrix}
f(\check{A}\circ\check{A})\\
f(\check{A}\circ\check{B})\\
f(\check{B}\circ\check{A})\\
f(\check{B}\circ\check{B})
\end{pmatrix}
$$
where
% generated by the command 'TypesetJellyfishFormulasMatrices[{BigraphWithDuals[GradedBigraph[{{1}}, {{1}}, {{1}, {1}, {1}}, {{1, 0, 0}, {0, 1, 0}}], DualData[{1}, {1}, {2, 1}]], BigraphWithDuals[GradedBigraph[{{1}}, {{1}}, {{1}, {1}, {1}}, {{1, 0, 0}, {0, 1, 0}}], DualData[{1}, {1}, {2, 1}]]}, 1]' in the Mathematica notebook code/GenerateLaTeX.nb available with the arXiv sources of this article.
\begin{align*}
J^{L}K &= \begin{pmatrix}
\frac{1}{51} \left(33-8 \sqrt{21}\right) & \lambda_{\begin{subarray}{l}4228250625,\\0,\\18810887175,\\0,\\36065983311,\\0,\\203997780,\\0,\\960400\\\mbox{}\end{subarray}}^{(0.034194+0.063236 i)} & \lambda_{\begin{subarray}{l}4228250625,\\0,\\18810887175,\\0,\\36065983311,\\0,\\203997780,\\0,\\960400\\\mbox{}\end{subarray}}^{(0.034194-0.063236 i)} & \frac{659 \sqrt{21}-2049}{2550} \\[6pt]
\lambda_{2601,0,885,0,-125}^{(0.3277)} & \lambda_{2601,4896,18981,5700,925}^{(-0.1561+0.1682 i)} & \lambda_{2601,4896,18981,5700,925}^{(-0.1561-0.1682 i)} & \lambda_{65025,0,1149168,0,-6845}^{(-0.07717)}\end{pmatrix}\\
\intertext{and}
\check{J}^{L}\check{K} &= \begin{pmatrix}
\frac{1}{51} \left(33-8 \sqrt{21}\right) & \lambda_{\begin{subarray}{l}4228250625,\\0,\\18810887175,\\0,\\36065983311,\\0,\\203997780,\\0,\\960400\\\mbox{}\end{subarray}}^{(0.034194+0.063236 i)} & \lambda_{\begin{subarray}{l}4228250625,\\0,\\18810887175,\\0,\\36065983311,\\0,\\203997780,\\0,\\960400\\\mbox{}\end{subarray}}^{(0.034194-0.063236 i)} & \frac{659 \sqrt{21}-2049}{2550} \\[6pt]
\lambda_{2601,0,885,0,-125}^{(0.3277)} & \lambda_{2601,4896,18981,5700,925}^{(-0.1561+0.1682 i)} & \lambda_{2601,4896,18981,5700,925}^{(-0.1561-0.1682 i)} & \lambda_{65025,0,1149168,0,-6845}^{(-0.07717)}\end{pmatrix}.\\
\end{align*}
\end{lem}

%%%%%%%%%%%%%%%%%%%%%%%%%%%%%%%%%%%%%%%%%%%%%%%%%%%%%%%%%%%%%%%%%%%%%%%%%
%%%%%%%%%%%%%%%%%%%%%%%%%%%%%%%%%%%%%%%%%%%%%%%%%%%%%%%%%%%%%%%%%%%%%%%%%
%%%%%%%%%%%%%%%%%%%%%%%%%%%%%%%%%%%%%%%%%%%%%%%%%%%%%%%%%%%%%%%%%%%%%%%%%
\section{Self-duality and calculating principal graphs}

We now know that our elements $A,B\in \cP\cA(\Gamma)$ generate an evaluable planar subalgebra $P^\Gamma_\bullet$, and hence a subfactor planar algebra. 
By the next lemma, we know that the principal graphs have the desired supertransitivity since we have 1-strand jellyfish relations.

\begin{lem}\label{lem:Supertransitive}
Suppose a planar algebra $P_\bullet$ is generated by uncappable elements $A_1,\dots, A_k\in P_{n,+}$ such that
\begin{enumerate}[(1)]
\item
the $A_j$'s satisfy 1-strand jellyfish relations, and
\item
the complex span of $\{ A_1,\dots, A_k,f^{(n)}\}$ forms an algebra under the usual multiplication. 
\end{enumerate}
Then $P_\bullet$ is $(n-1)$ supertransitive.
\end{lem}
\begin{proof}
Suppose $x\in P_{n-1,+}$. Use the jellyfish relations to get a linear combination of diagrams in jellyfish form. Use the facts that the $A_j$'s are uncappable and that the $A_j$'s and $f^{(n)}$ form an algebra to reduce the linear combination of diagrams so that no two generators are connected by more than $n$ strands, and no generator is connected to itself. Since there are only $2n-2$ external boundary points, it is easy to see that the diagrams in the resulting linear combination have no generators, i.e., $x\in TL_{n-1,+}$.
\end{proof}

We now show that the principal graphs of the $P_\bullet^{\Gamma}$ planar algebra are, in fact, $(\Gamma,\Gamma)$.

%%%%%%%%%%%%%%%%%%%%%%%%%%%%%%%%%%%%%%%%%%%%%%%%%%%%%%%%%%%%%%%%%%%%%%%%%%
\subsection{Self duality}
In this subsection, we show that $P_\bullet^{4442}, P_\bullet^{3333}$ and  $P_\bullet^{2221}$ are self-dual, i.e., there is a planar algebra isomorphism $\Phi\colon P_\bullet\to P_\bullet^{\vee}$ where $P_\bullet^\vee$ is the dual planar algebra obtained from $P_\bullet$ by reversing the shading. Note that this means for all $k$, there is a map $\Phi_{k,\pm}\colon P_{k,\pm}\to P_{k,\pm}^\vee=P_{k,\mp}$, and these maps commute with the action of the planar operad.

In fact, these three subfactor planar algebras are more than self-dual; they are \emph{symmetrically self-dual}, i.e., for every $n$, $\Phi_{n,\mp}\circ \Phi_{n,\pm}=\id_{n,\pm}$. Hence by \cite{4442equi}, we can lift the shading to get \emph{fantastic planar algebras}, i.e., unshaded, spherical, evaluable $C^*$-planar algebras. 

Given a fantastic planar algebra $P_\bullet$, we have an associated rigid $C^*$-tensor category $\cC_{P_\bullet}$ whose objects are the projections of $P_\bullet$ and a morphism in $\Hom{\cC}{p}{q}$ is an element $x\in P_\bullet$ such that $x=pxq$ (see \cite{MR2559686} for more details). Note further that $\cC_{P_\bullet}$ is generated by a single self-dual object $X$ (the strand), and $\cC_{P_\bullet}$ is $\Integer/2$-graded. The fusion graph with respect to $X$ is exactly the principal graph of $P_\bullet$. If the fusion graph is finite, then $\cC_{P_\bullet}$ is a unitary fusion category.

Hence the subfactor planar algebras $P_\bullet^{4442}, P_\bullet^{3333}, P_\bullet^{2221}$ give rise to $\Integer/2$-graded unitary fusion categories generated by a single self-dual object with fusion graphs $4442,3333,2221$ respectively. Note that a fusion category with fusion graph 2221 has previously been constructed by Ostrik in the appendix to \cite{1004.0665}.

\begin{thm}\label{thm:4442SelfDual}
The map $P_{5,+}^{4442}\leftrightarrow P_{5,-}^{4442}$ exchanging $A\leftrightarrow\check{A}$ and $B\leftrightarrow\check{B}$ gives a symmetric self-duality of planar algebras.
\end{thm}
\begin{proof}
By the symmetry of the moments in Appendix \ref{sec:moments4442}, the map clearly preserves the moments. Hence the box jellyfish relations of Subsection \ref{sec:Relations4442} are preserved under $\Phi$.
Moreover, the structure coefficients in the algebra $P_{5,\pm}$ are also preserved, so by the jellyfish algorithm the map preserves the evaluation of all closed diagrams so is an isomorphism.
\end{proof}

\begin{thm}\label{thm:3333SelfDual}
The map $\Phi_{3,+}\colon P_{4,+}^{3333}\to P_{4,-}^{3333}$ by 
$$
\begin{pmatrix}
A\\
B
\end{pmatrix}
\mapsto
M
\begin{pmatrix}
\check{A}\\
\check{B}
\end{pmatrix}
\text{ where }
M=
\begin{pmatrix}
\frac{\sqrt{\frac{5}{2}}}{2} & \frac{\sqrt{3+\sqrt{5}}}{4} \\
\frac{3 \sqrt{3-\sqrt{5}}}{4} & -\frac{\sqrt{\frac{5}{2}}}{2}
\end{pmatrix}
$$ 
gives a symmetric self-duality of planar algebras.
\end{thm}
\begin{remark}
It might be possible to choose generators so that $(A,B) \mapsto (\check{A}, \check{B})$ is already a symmetric self-duality of planar algebras, but it seems that one would have to work in a larger number field for this to be possible.
\end{remark}

\begin{proof}
One can easily verify that this map preserves the moments given in Appendix \ref{sec:moments3333}, so the argument from the proof of Theorem \ref{thm:4442SelfDual} applies. (In fact, we do this verification in the {\tt Mathematica} notebook {\tt QuadraticTangles.nb}.)
Finally, note that $M^2=\id$, so $\Phi_{n,\mp}\circ \Phi_{n,\pm}=\id_{n,\pm}$.
\end{proof}

\begin{thm}\label{thm:2221SelfDual}
The map $P_{3,+}^{2221}\leftrightarrow P_{3,-}^{2221}$ swapping $A\leftrightarrow\check{A}$ and $B\leftrightarrow\check{B}$ gives a symmetric self-duality of planar algebras.
\end{thm}
\begin{proof}
Similar to the proof of Theorem \ref{thm:4442SelfDual}.
\end{proof}

Note, however, that the result is not true for 3311. This result has been known since \cite{MR1355948}, and we give another proof for completeness.

\begin{thm}\label{thm:3311NotSelfDual}
$P_\bullet^{3311}$ is not self-dual.
\end{thm}
\begin{proof}
Suppose there were a self-duality $\Phi$. Since $\Phi$ is a map of planar algebras, $\Phi$ must preserve the Temperley-Lieb planar subalgebras $TL_{k,\pm}$ and also the low weight spaces for the rotation in $P_{4,\pm}^{3311}$. Since the rotational eigenvalues are $\omega_A=-1$ and $\omega_B=1$, we know that $\Phi(A)=\mu \check{A}$ for some $\mu\in\Complex^\times$. However, this clearly violates
$$
\Tr(A^3) = \frac{1}{27} \left(-6-4 \sqrt{3}\right)\text{ and }\Tr(\check{A}^3) = 0
$$
(see Appendix \ref{sec:moments3311}).
\end{proof}

%%%%%%%%%%%%%%%%%%%%%%%%%%%%%%%%%%%%%%%%%%%%%%%%%%%%%%%%%%%%%%%%%%%%%%%%%%
\subsection{Identifying principal graphs}

\begin{thm}\label{thm:4442graph}
The principal graphs of $P_\bullet^{4442}$ are 
$$
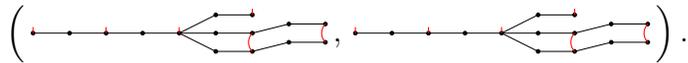
\left(
\bigraph{bwd1v1v1v1v1p1p1v1x0x0p0x1x0p0x0x1v1x0x0p0x1x0v1x0p0x1duals1v1v1v2x1x3v2x1},
\bigraph{bwd1v1v1v1v1p1p1v1x0x0p0x1x0p0x0x1v1x0x0p0x1x0v1x0p0x1duals1v1v1v2x1x3v2x1}
\right).
$$
\end{thm}
\begin{proof}
First, note that the modulus of the given subfactor planar algebra is $\sqrt{3+\sqrt{5}} \simeq 2.28825$. Let $\Gamma$ denote the principal graph. Using the moments calculated from our generators in Appendix \ref{sec:moments4442}, we find all solutions to the equation
$$
(aA+bB+cf^{(5)})^2=aA+bB+cf^{(5)}
$$
to get that the minimal projections one past the branch point are given by the equations:
$$
(a,b,c)=
% computed in part by the command 'FindCoefficientsOfProjections[GPALopsidedGenerators[G, 1], GPAChiralities[G], \[Lambda]]' in the Mathematica notebook code/QuadraticTangles.nb available with the arXiv sources of this article.
\begin{cases}
\begin{gathered}
\lambda_{\begin{subarray}{l}332150625,\\-7418030625,\\60616258875,\\-208296167625,\\266656517775,\\-15736631400,\\328666680,\\13223880,\\190096\\\mbox{}\end{subarray}}^{(0.042252+0.036141 i)},
\lambda_{\begin{subarray}{l}332150625,\\-7418030625,\\60616258875,\\-208296167625,\\266656517775,\\-15736631400,\\328666680,\\13223880,\\190096\\\mbox{}\end{subarray}}^{(-0.011821-0.010111 i)},
\frac{1}{5} \left(-5+3 \sqrt{5}\right) 
\end{gathered} \\
\begin{gathered}
\lambda_{\begin{subarray}{l}332150625,\\-7418030625,\\60616258875,\\-208296167625,\\266656517775,\\-15736631400,\\328666680,\\13223880,\\190096\\\mbox{}\end{subarray}}^{(-0.011821-0.010111 i)},
\lambda_{\begin{subarray}{l}332150625,\\-7418030625,\\60616258875,\\-208296167625,\\266656517775,\\-15736631400,\\328666680,\\13223880,\\190096\\\mbox{}\end{subarray}}^{(0.042252+0.036141 i)},
\frac{1}{5} \left(-5+3 \sqrt{5}\right) 
\end{gathered} \\
\begin{gathered}
\lambda_{2025,45225,274635,16620,436}^{(-0.030431-0.026030 i)},
\lambda_{2025,45225,274635,16620,436}^{(-0.030431-0.026030 i)},
\frac{1}{5} \left(15-6 \sqrt{5}\right).
\end{gathered}
\end{cases}
$$
Note that the dimensions of these projections agree with the dimensions of the vertices on the $4442$ graph, so all the arms on the principal graph $\Gamma$ must continue. Two branches cannot merge since 
$$
\left\|\bigraph{gbg1v1v1v1v1p1p1v1x1x0}\right\| \approx 2.33743,
$$
no branch can split since 
$$
\left\|\bigraph{gbg1v1v1v1v1p1p1v1x0x0p1x0x0p0x1x0p0x0x1} \right\| \approx 2.31725,
$$
and no branch can have a double edge since
$$
\left\|\bigraph{gbg1v1v1v1v1p1p1v0x2x0} \right\| \approx 2.4761,
$$
and all of these numbers are already too large.
Hence all branches continue simply. By counting Frobenius-Perron dimensions, one arm must stop, but the other two must continue. Again, the branches cannot merge or split as 
\begin{align*}
\left\|\bigraph{gbg1v1v1v1v1p1p1v1x0x0p0x1x0p0x0x1v1x1x0}\right\| &\approx 2.32033,\\
\left\|\bigraph{gbg1v1v1v1v1p1p1v1x0x0p0x1x0p0x0x1v1x0x0p1x0x0} \right\| &\approx 2.29079, \text{ and }\\
\left\|\bigraph{gbg1v1v1v1v1p1p1v1x0x0p0x1x0p0x0x1v0x2x0}\right\| &\approx 2.41976.
\end{align*}
Again, both branches must continue simply. Once more by counting dimensions, the two remaining arms must continue, and again they cannot merge or split as
\begin{align*}
\left\|\bigraph{gbg1v1v1v1v1p1p1v1x0x0p0x1x0p0x0x1v0x1x0p0x0x1v1x1}\right\| &\approx 2.30231,\\
\left\|\bigraph{gbg1v1v1v1v1p1p1v1x0x0p0x1x0p0x0x1v0x1x0p0x0x1v1x0p1x0}\right\| &\approx 2.29193,\text{ and }\\
\left\|\bigraph{gbg1v1v1v1v1p1p1v1x0x0p0x1x0p0x0x1v0x1x0p0x0x1v2x0}\right\| &\approx 2.37309.
\end{align*}
We conclude by counting dimensions again that $\Gamma$ is 4442.

The dual principal graph is also $4442$ by Theorem \ref{thm:4442SelfDual}.

To determine the dual data, we run the program {\tt FindGraphPartners} on the graph
$$
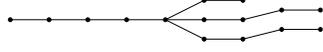
\bigraph{gbg1v1v1v1v1p1p1v1x0x0p0x1x0p0x0x1v1x0x0p0x1x0v1x0p0x1}
$$
which determines all possible pairs of principal graphs and dual data for which one of the graphs is 4442. The only possibilities for which the principal and dual principal graph are both 4442 are
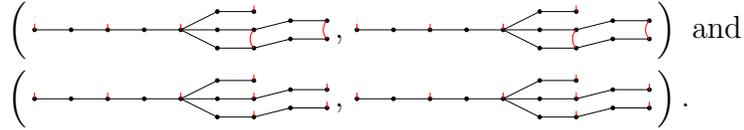
\begin{align*}
&\left(\bigraph{bwd1v1v1v1v1p1p1v1x0x0p0x1x0p0x0x1v1x0x0p0x1x0v1x0p0x1duals1v1v1v2x1x3v2x1},\bigraph{bwd1v1v1v1v1p1p1v1x0x0p0x1x0p0x0x1v1x0x0p0x1x0v1x0p0x1duals1v1v1v2x1x3v2x1}\right)
\text{ and }\\
&\left(\bigraph{bwd1v1v1v1v1p1p1v1x0x0p0x1x0p0x0x1v1x0x0p0x1x0v1x0p0x1duals1v1v1v1x2x3v1x2},\bigraph{bwd1v1v1v1v1p1p1v1x0x0p0x1x0p0x0x1v1x0x0p0x1x0v1x0p0x1duals1v1v1v1x2x3v1x2}\right).
\end{align*}
We rule out the second pair by noting that the dimension 1 bimodules form a group, and thus the vertices at the far right must be dual to each other.
\end{proof}

If we already knew $P_\bullet^{4442}$ was finite depth, we could give an alternative argument identifying the principal graph as 4442 as follows. We omit details, as this is redundant with the argument in Theorem \ref{thm:4442graph}.
\begin{thm}
The only finite depth subfactor principal graphs with index $3+\sqrt{5}$ starting like 4111 are 4442.
\end{thm}
\begin{proof}
If any of the edges above the quadruple point end immediately, that vertex has dimension which is not an algebraic integer (a root of $81 - 126 x^2 + 4 x^4$).
Otherwise, graph enumeration in the style of \cite{index5-part1} shows that two of the legs end, with lengths 2, 2 or with lengths 2, 3. In second case, there's a dimension which isn't an algebraic integer. In the other case, the graph contains 4422. (There are two possibilities for the dual data.) We can look for connections on this graph (without assuming that it ends), and show that there are no bi-unitary connections.
\end{proof}

\begin{thm}\label{thm:3333graph}
The principal graphs of $P_\bullet^{3333}$ are 
$$
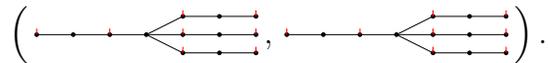
\left(
\bigraph{bwd1v1v1v1p1p1v1x0x0p0x1x0p0x0x1v1x0x0p0x1x0p0x0x1duals1v1v1x2x3v1x2x3},
\bigraph{bwd1v1v1v1p1p1v1x0x0p0x1x0p0x0x1v1x0x0p0x1x0p0x0x1duals1v1v1x2x3v1x2x3}
\right).
$$
\end{thm}
\begin{proof}
The proof is similar to that of Theorem \ref{thm:4442graph}. Again, the modulus is $\sqrt{3+\sqrt{5}} \simeq 2.28825$, and we find that the minimal projections one past the branch are given by $aA +bB+cf^{(4)}$ where
$$
(a,b,c)=
% computed in part by the command 'FindCoefficientsOfProjections[GPALopsidedGenerators[G, 1], GPAChiralities[G], \[Lambda]]' in the Mathematica notebook code/QuadraticTangles.nb available with the arXiv sources of this article.
\begin{cases}
\displaystyle\left(\frac{1}{4} \left(-1+\sqrt{5}\right),  -\frac{\sqrt{5}}{6},\frac{1}{3}\right) \\
\displaystyle\left(-\frac{1}{2} , \frac{1}{12} \left(-3+\sqrt{5}\right),\frac{1}{3}\right) \\
\displaystyle\left(\frac{1}{4} \left(3-\sqrt{5}\right), \frac{1}{12} \left(3+\sqrt{5}\right),\frac{1}{3}\right).
\end{cases}
$$
Since $\Tr(f^{(4)})=6+3 \sqrt{5}$, all the minimal projections have trace $2+\sqrt{5}$, which agree with the Frobenius-Perron dimensions of the vertices of
$$
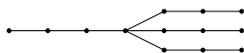
\bigraph{gbg1v1v1v1p1p1v1x0x0p0x1x0p0x0x1v1x0x0p0x1x0p0x0x1}
$$
at depth $4$. Once again, counting Frobenius-Perron dimensions and noting that
\begin{align*}
\left\|\bigraph{gbg1v1v1v1p1p1v1x1x0}\right\| &\approx 2.33441,\\
\left\|\bigraph{gbg1v1v1v1p1p1v1x0x0p1x0x0p0x1x0p0x0x1} \right\| &\approx 2.31384,\\
\left\|\bigraph{gbg1v1v1v1p1p1v0x2x0} \right\| &\approx 2.47485,\\
\left\|\bigraph{gbg1v1v1v1p1p1v1x0x0p0x1x0p0x0x1v1x1x0}\right\| &\approx 2.31725,\\
\left\|\bigraph{gbg1v1v1v1p1p1v1x0x0p0x1x0p0x0x1v1x0x0p1x0x0p0x1x0}\right\| &\approx 2.29813,\text{ and }\\
\left\|\bigraph{gbg1v1v1v1p1p1v1x0x0p0x1x0p0x0x1v0x2x0}\right\| &\approx 2.41856
\end{align*}
yields the result.

The dual principal graph is also $3333$ by Theorem \ref{thm:3333SelfDual}.

To determine the dual data, we run {\tt FindGraphPartners} as in the proof of Theorem \ref{thm:4442graph} on the graph
$$
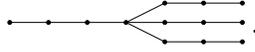
\bigraph{gbg1v1v1v1p1p1v1x0x0p0x1x0p0x0x1v1x0x0p0x1x0p0x0x1}.
$$
The only possibilities for which the principal and dual principal graph are both 3333 are
\begin{align*}
&\left(\bigraph{bwd1v1v1v1p1p1v1x0x0p0x1x0p0x0x1v1x0x0p0x1x0p0x0x1duals1v1v1x2x3v1x2x3},\bigraph{bwd1v1v1v1p1p1v1x0x0p0x1x0p0x0x1v1x0x0p0x1x0p0x0x1duals1v1v1x2x3v1x2x3}\right)
\text{ and }\\
&\left(\bigraph{bwd1v1v1v1p1p1v1x0x0p0x1x0p0x0x1v1x0x0p0x1x0p0x0x1duals1v1v2x1x3v2x1x3},\bigraph{bwd1v1v1v1p1p1v1x0x0p0x1x0p0x0x1v1x0x0p0x1x0p0x0x1duals1v1v2x1x3v2x1x3}\right).
\end{align*}
Since $\omega_A=\omega_B=1$ (see Appendix \ref{sec:generators3333}), the two click rotation $\rho$ must be the identity on $P^{3333}_{4,+}\ominus TL_{4,+}$, and thus all vertices at depth $4$ must be self-dual.
\end{proof}

\begin{thm}\label{thm:3311graph}
The principal graphs of $P_\bullet^{3311}$  are
$$
\left(
\bigraph{bwd1v1v1v1p1p1v1x0x0v1duals1v1v1x2x3v1},
\bigraph{bwd1v1v1v1p1p1v1x0x0v1duals1v1v1x2x3v1}
\right).
$$
\end{thm}
\begin{proof}
Similar to the proofs of Theorems \ref{thm:4442graph} and \ref{thm:3333graph}. The modulus is $\sqrt{3+\sqrt{3}}\simeq 2.17533$, and the minimal projections are given by $aA +bB+cf^{(4)}$ where
$$
(a,b,c)=
% computed in part by the command 'FindCoefficientsOfProjections[GPALopsidedGenerators[G, 2], GPAChiralities[G], \[Lambda]]' in the Mathematica notebook code/QuadraticTangles.nb available with the arXiv sources of this article.
\begin{cases}
\left(-\frac{\sqrt{3}}{2}, \frac{1}{11} \left(4-3 \sqrt{3}\right),  \frac{1}{11} \left(1+2 \sqrt{3}\right) \right)\\
\left(\frac{1}{4} \left(3-\sqrt{3}\right), \frac{1}{22} \left(7+3 \sqrt{3}\right), \frac{1}{11} \left(5-\sqrt{3}\right)\right) \\
\left(\frac{1}{4} \left(-3+3 \sqrt{3}\right), \frac{1}{22} \left(-15+3 \sqrt{3}\right), \frac{1}{11} \left(5-\sqrt{3}\right)\right).
\end{cases}
$$ 
Since $\Tr(f^{(4)})=4+3\sqrt{3}$, the traces of the minimal projections are $2+\sqrt{3},1+\sqrt{3},1+\sqrt{3}$ respectively, which agree with the Frobenius-Perron dimensions of the vertices of
$$
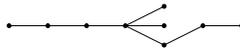
\bigraph{gbg1v1v1v1p1p1v1x0x0v1}
$$
at depth $4$ reading from bottom to top. Finally,
\begin{align*}
\left\|\bigraph{gbg1v1v1v1p1p1v1x1x0}\right\| &\approx 2.33441,\\
\left\| \bigraph{gbg1v1v1v1p1p1v1x0x0p1x0x0}\right\| &\approx 2.23607 \text{ and}\\
\left\| \bigraph{gbg1v1v1v1p1p1v0x2x0}\right\| &\approx 2.47485.
\end{align*}
Hence the principal graph is 3311.

For the dual graph, note that the minimal projections in $P_{4,-}$ are given by $\check{a}\check{A} +\check{b}\check{B}+\check{c}\check{f}^{(4)}$ where
$$
(\check{a},\check{b},\check{c})=
% computed in part by the command 'FindCoefficientsOfProjections[GPALopsidedGenerators[G, 2], GPAChiralities[G], \[Lambda]]' in the Mathematica notebook code/QuadraticTangles.nb available with the arXiv sources of this article.
\begin{cases}
\left(0,\frac{1}{11} \sqrt{2 \left(27-\sqrt{3}\right)}, \frac{1}{11} \left(1+2 \sqrt{3}\right)\right) \\
\left(-\frac{1}{2} \sqrt{\frac{3}{2} \left(3-\sqrt{3}\right)}, -\frac{1}{11} \sqrt{\frac{1}{2} \left(27-\sqrt{3}\right)}, \frac{1}{11} \left(5-\sqrt{3}\right)\right)\\
\left(\frac{1}{2} \sqrt{\frac{3}{2} \left(3-\sqrt{3}\right)}, -\frac{1}{11} \sqrt{\frac{1}{2} \left(27-\sqrt{3}\right)}, \frac{1}{11} \left(5-\sqrt{3}\right)\right).
\end{cases}
$$
Since $\Tr(\check{f}^{(4)})=\Tr(f^{(4)})$, the same argument as above applies, and the dual principal graph is 3311.

To determine the dual data, we note there are only two possibilities for each graph: either the singly valent vertices at depth 4 are self-dual or they are dual to each other. Since $\omega_A=-1$ and $\omega_B=1$ (see Subsection \ref{sec:generators3311}), we know that $\rho^2=\id{}$ on $P^{3311}_{4,+}\ominus TL_{4,+}$. Hence the singly valent vertices must be self dual.
\end{proof}

\begin{thm}
The principal graphs of $P_\bullet^{2221}$ are 
$$
\left(
\bigraph{bwd1v1v1p1p1v1x0x0p0x1x0duals1v1v2x1},
\bigraph{bwd1v1v1p1p1v1x0x0p0x1x0duals1v1v2x1}
\right).
$$
\end{thm}
\begin{proof}
Similar to the proof of Theorems \ref{thm:4442graph}, \ref{thm:3333graph}, and \ref{thm:3311graph}. The modulus is $\sqrt{(5+\sqrt{21})/2}\approx 2.1889$, and the minimal projections are given by $aA +bB+cf^{(3)}$ where
$$
(a,b,c)=
% computed in part by the command 'FindCoefficientsOfProjections[GPALopsidedGenerators[G, 2], GPAChiralities[G], \[Lambda]]' in the Mathematica notebook code/QuadraticTangles.nb available with the arXiv sources of this article.
\begin{cases}
\begin{gathered}
\lambda_{\begin{subarray}{l}729,\\-4374,\\9477,\\-8586,\\2943,\\-270,\\153,\\-6,\\1\\\mbox{}\end{subarray}}^{(-0.0476+0.2030 i)},
 \lambda_{\begin{subarray}{l}455625,\\364500,\\1293975,\\-439020,\\1378809,\\119772,\\31257,\\-570,\\25\\\mbox{}\end{subarray}}^{(0.010242-0.025430 i)},
\frac{1}{10} \left(-1+\sqrt{21}\right)
\end{gathered}\\
\begin{gathered}
\lambda_{\begin{subarray}{l}729,\\-4374,\\9477,\\-8586,\\2943,\\-270,\\153,\\-6,\\1\\\mbox{}\end{subarray}}^{(0.02010-0.08568 i)},
\lambda_{\begin{subarray}{l}455625,\\364500,\\1293975,\\-439020,\\1378809,\\119772,\\31257,\\-570,\\25\\\mbox{}\end{subarray}}^{(-0.05749+0.14274 i)},
\frac{1}{10} \left(-1+\sqrt{21}\right) \\
\end{gathered}\\
\begin{gathered}
\lambda_{9,54,66,-3,1}^{(0.02753-0.11731 i)},
\lambda_{225,-180,-294,27,-5}^{(0.04725-0.11731 i)},
\frac{1}{5} \left(6-\sqrt{21}\right).
\end{gathered}
\end{cases}
$$
Since $\Tr(f^{(3)})=\sqrt{19+4 \sqrt{21}}$, the traces of the minimal projections are $\sqrt{\frac{5}{2}+\frac{\sqrt{21}}{2}},\sqrt{\frac{5}{2}+\frac{\sqrt{21}}{2}},\sqrt{3}$ respectively, which agree with the Frobenius-Perron dimensions of the vertices of
$$
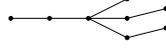
\bigraph{gbg1v1v1p1p1v1x0x0p0x1x0}
$$
at depth $3$ reading from bottom to top. Finally,
\begin{align*}
\left\| \bigraph{gbg1v1v1p1p1v1x1x0}\right\| &\approx 2.32437,\\
\left\| \bigraph{gbg1v1v1p1p1v1x0x0p1x0x0}\right\| &\approx 2.22158, \text{ and}\\
\left\| \bigraph{gbg1v1v1p1p1v0x2x0}\right\| &\approx 2.46991.
\end{align*}
The dual principal graph is also $2221$ by Theorem \ref{thm:2221SelfDual}.

Finally, since the dimension one bimodules form a group, the dual data must be as claimed.
\end{proof}

%%%%%%%%%%%%%%%%%%%%%%%%%%%%%%%%%%%%%%%%%%%%%%%%%%%%%%%%%%%%%%%%%%%%%%%%%
%%%%%%%%%%%%%%%%%%%%%%%%%%%%%%%%%%%%%%%%%%%%%%%%%%%%%%%%%%%%%%%%%%%%%%%%%
%%%%%%%%%%%%%%%%%%%%%%%%%%%%%%%%%%%%%%%%%%%%%%%%%%%%%%%%%%%%%%%%%%%%%%%%%

\appendix
\section{Generators in the graph planar algebra}\label{sec:generators}
Specifying an element of the graph planar algebra of a large graph can be a somewhat cumbersome process; the element is a function on loops of a certain length on the graph, and we need to specify each value. Fortunately, if the element is a lowest weight vector, this can be significantly abbreviated.

Throughout this section, we assume that $\Gamma$ is a spoke graph with $m+1$ arms, and the initial arm is at least as long as any of the other arms. (This obviously holds for the graphs in which we are interested.)

\begin{lem}\mbox{}%
\label{lem:collapsing}
\begin{enumerate}
\item[(1)]
A lowest weight vector in $A \in \cG(\Gamma)_{n,+}$ is determined by its values on `collapsed' loops which stay within distance one of the central vertex. 
\item[(2)]
In fact, it is determined by its values on such loops which never enter one of the spokes of our choice.
\item[(3)] 
Further, the value on any loop which has more than $2k+1$ consecutive vertices which either lie on a particular arm of the graph of length $k-1$ or are the central vertex is zero.
\end{enumerate}
\end{lem}
\begin{proof}
We'll work in the spherical graph planar algebra, as it is somewhat easier to state the requisite formulas there. Obviously the lemma holds in the spherical planar algebra if and only if it holds in the lopsided graph planar algebra.

Call the central vertex of $\Gamma$ $c$.
We'll write $\norm{\gamma} = \sum_i d(\gamma(i), c)$. For a collapsed loop, $\norm{\gamma} = n/2$, while for any other loop $\norm{\gamma} > n/2$. We'll show that for any non-collapsed loop $\gamma$, if $A$ is a lowest weight vector, then $A(\gamma)$ is determined by the value of $A$ on loops of strictly smaller norm. Inductively, this gives the result.

Suppose $\gamma$ is a loop of length $n$ on $\Gamma$, with 
$d(\gamma(i), c) \geq 2$ and $d(\gamma(i\pm1), c) = d(\gamma(i), c) -1$. (That is, $i$ is a position on the loop where $\gamma$ reaches a local maximum distance from the centre.)
Consider the modified loop $\gamma'$, which agrees with $\gamma$ except at position $i$, where it passes through the vertex 2 closer to the centre than $\gamma(i)$ (possibly the central vertex itself). Consider also the `snipped' loop $\pi$ of length $n-2$, obtained from $\gamma$ or $\gamma'$ by removing the $i$-th and $i+1$-th positions.
We name the vertices as $s = \gamma(i)$, $r = \gamma(i\pm1)$, and $t = \gamma'(i)$.
 
Applying a cap at position $i$ to $A$, we have $\cap_i (A) = 0$. Evaluating this at $\pi$ gives
\begin{align*}
0 = \sqrt{\dim(r)}^{k_i} \;{\cap_i(A)(\pi)} & = \sqrt{\dim(s)}^{k_i} A(\gamma) + \sqrt{\dim(t)}^{k_i} A(\gamma').
\end{align*}
(Here $k_i$ is the number of critical points in the cap strand, either $1$, $2$ or $3$ depending on the position of the point $i$ around the boundary of the rectangular box, as follows
\begin{equation*}
k_i = \begin{cases}
1 & \text{when we have $\tikz[baseline=7]{\draw[fill=gray] (0,0) rectangle (1,0.5); \draw (0.3,0.5) arc (180:90:0.2) node[above left] {$i$} arc(90:0:0.2);}$ or $\tikz[baseline=0]{\draw[fill=gray] (0,0) rectangle (1,0.5); \draw (0.3,0) arc (180:270:0.2) node[right=6pt] {$i$} arc(270:360:0.2);}$} \\
2 & \text{when we have $\tikz[baseline=4]{\draw[fill=gray] (0,0) rectangle (1,0.5); \draw (0.8,0.5) arc (180:0:0.2) -- node[right] {$i$} (1.2,0) arc(0:-180:0.2);}$ or $\tikz[baseline=4]{\draw[fill=gray] (0,0) rectangle (1,0.5); \draw (0.2,0.5) arc (0:180:0.2) -- node[left] {$i$} (-0.2,0) arc(180:360:0.2);}$} \\
3 & \text{when we have $\tikz[baseline=5]{\path (0,0) -- (0,1); \draw[fill=gray] (0,0) rectangle (1,0.5); \draw (0.2,0.5) arc (0:180:0.2) -- (-0.2,0) arc(180:270:0.7) node[above] {$i$} arc (270:360:0.7) -- (1.2,0.5) arc (0:180:0.2);}$ or $\tikz[y=-1cm,baseline=5]{\path (0,0) -- (0,1); \draw[fill=gray] (0,0) rectangle (1,0.5); \draw (0.2,0.5) arc (0:180:0.2) -- (-0.2,0) arc(180:270:0.7) node[below] {$i$} arc (270:360:0.7) -- (1.2,0.5) arc (0:180:0.2);}$}
\end{cases}
\end{equation*}
although the case $k=3$ never occurs for us as we always consider boxes with equal numbers of strands above and below.)

For the second statement, consider some collapsed loop $\alpha$ of length $n-2$. Write $\alpha_{i,j}$ for the collapsed loop of length $n$ which makes an extra visit to the $j$-th spoke of the graph at position $i$, and write $u_j$ for the vertex adjacent to the central vertex on the $j$-th spoke. (If the graph has even supertransitivity, then $i$ ought to be odd, while if the graph has odd supertransitivity, $i$ is even.) Now,
\begin{equation}
\label{eq:cap-at-central-vertex}
0 = \sqrt{\dim(c)}^{k_i} \;{\cap_i(A)(\pi)}  = \sum_j \sqrt{\dim(u_j)}^{k_i} A(\alpha_{i,j}).
\end{equation}
Using this formula, we can express the value of $A$ on any collapsed loop which visits some spoke in terms of other collapsed loops which visit that spoke strictly fewer times.

For the final statement, note that $A$ has value zero on any loop which visits a univalent vertex at positions $i$ and $i+2$ by a similar argument as above. The value of $A$ on a loop with $2k+1$ consecutive vertices either in a fixed spoke of length $k-1$ or at the central vertex is then a multiple of the value of $A$ on a loop that visits the end of that spoke twice consecutively by our first argument, and is thus also zero.
\end{proof}

\begin{cor}
If $A$ is a lowest weight vector in a spherical graph planar algebra, 
and $\hat{\gamma}$ denotes the `collapsed' loop corresponding to $\gamma$, then
\begin{equation}
\label{eq:collapsing}
A(\gamma) = (-1)^{(\norm{\gamma}-\norm{\hat{\gamma}})/2} \left( \prod_{i} \sqrt{\frac{\dim(\hat{\gamma}(i))}{\dim(\gamma(i))}}^{k_i} \right) A(\hat{\gamma}).
\end{equation}
\end{cor}
The corresponding formula for a lowest weight vector in a lopsided graph planar algebra is
\begin{equation}
\label{eq:collapsing-lopsided}
A(\gamma) = (-1)^{(\norm{\gamma}-\norm{\hat{\gamma}})/2} \left( \prod_{i} \left(\frac{\dim(\hat{\gamma}(i))}{\dim(\gamma(i))}\right)^{\ell_i} \right) A(\hat{\gamma})
\end{equation}
where 
\begin{equation*}
\ell_i = \begin{cases}
0 & \text{when we have $\tikz[baseline=7]{\draw[fill=gray] (0,0) rectangle (1,0.5); \draw (0.3,0.5) arc (180:90:0.2) node[above left] {$i$} arc(90:0:0.2);}$} \\
1 & \text{when we have $\tikz[baseline=0]{\draw[fill=gray] (0,0) rectangle (1,0.5); \draw (0.3,0) arc (180:270:0.2) node[right=6pt] {$i$} arc(270:360:0.2);}$, $\tikz[baseline=4]{\draw[fill=gray] (0,0) rectangle (1,0.5); \draw (0.8,0.5) arc (180:0:0.2) -- node[right] {$i$} (1.2,0) arc(0:-180:0.2);}$, $\tikz[baseline=4]{\draw[fill=gray] (0,0) rectangle (1,0.5); \draw (0.2,0.5) arc (0:180:0.2) -- node[left] {$i$} (-0.2,0) arc(180:360:0.2);}$ or $\tikz[baseline=5]{\path (0,0) -- (0,1); \draw[fill=gray] (0,0) rectangle (1,0.5); \draw (0.2,0.5) arc (0:180:0.2) -- (-0.2,0) arc(180:270:0.7) node[above] {$i$} arc (270:360:0.7) -- (1.2,0.5) arc (0:180:0.2);}$} \\
2 & \text{when we have $\tikz[y=-1cm,baseline=5]{\path (0,0) -- (0,1); \draw[fill=gray] (0,0) rectangle (1,0.5); \draw (0.2,0.5) arc (0:180:0.2) -- (-0.2,0) arc(180:270:0.7) node[below] {$i$} arc (270:360:0.7) -- (1.2,0.5) arc (0:180:0.2);}$}\,.
\end{cases}
\end{equation*}
(The exponent $\ell_i$ is just the number of minima on the strand.)

Rotation acts on the set of collapsed loops, so if we are trying to specify a lowest weight vector $A$ which is also a rotational eigenvector, then it suffices to specify $A$ only on a representative of each such orbit.

\begin{lem}
\label{lem:verify-generator}
Fix $\omega$ an $n$-th root of unity.
Suppose we have specified the values of an element $A$ on a rotation representative of each collapsed loop which avoids the initial arm, and further that
\begin{enumerate}[(a)]
\item if a representative is fixed by the $k$-fold rotation, and $\omega^k \neq 1$, the corresponding value of $A$ is zero, and
\item condition (3) of Lemma \ref{lem:collapsing} holds where appropriate on these values, i.e., if a loop visits an arm of length $k-1$ at least $k$ consecutive times, the corresponding value of $A$ is zero.
\end{enumerate}
Then we can make three consecutive well-defined extensions, defining, in turn, the values of $A$ 
\begin{enumerate}[(1)]
\item  on every collapsed loop avoiding the initial arm, using the condition that $A$ is a rotational eigenvector with some eigenvalue $\omega$,
\item on every collapsed loop, via Equation \eqref{eq:cap-at-central-vertex}, and finally
\item on every loop, via Equation \eqref{eq:collapsing}.
\end{enumerate}

The resulting element $A$ is a rotational eigenvector, and is a lowest weight vector if and only if
$$A(00\cdots0) = (-1)^n \sum_{a_i=1,\ldots,m} \left(\prod_{i=1}^n \left( \frac{\dim{u_{a_i}}}{\dim{u_0}}\right)^{k_i/2}\right) A(a_1a_2\cdots a_n) = 0$$
(here we denote the value of $A$ on the collapsed loop which successively visits arms $a_1, a_2, \ldots, a_n$ by $A(a_1a_2\cdots a_n)$).
\end{lem}

\begin{ex}
Consider the Haagerup principal graph, with $n=4$ and three arms of length 3. There are 5 representatives under rotation of collapsed loops avoiding the initial arm, namely $1111, 1112, 1122, 1212, 1222$ and $2222$. Condition (3) of Lemma \ref{lem:collapsing} says that $A(1111)$ and $A(2222)$ must be zero if we want to extend $A$ to a lowest weight vector. Further, if $\omega = \pm i$ then $A(1212)$ must be zero also. Finally, the condition $A(0000) = 0$ is trivially true when $\omega \neq 1$, since there $A(0000)= \omega A(0000)$, while it is non-trivial when $\omega = 1$. Thus the the rotational eigenspaces of lowest weight $4$-boxes in the graph planar algebra are 3 dimensional for $\omega = 1, i, -i$, and 4 dimensional for $\omega = -1$.
\end{ex}

Thus, in each of the following subsections, we list representatives of the rotational orbits of the collapsed loops avoiding the initial spoke. 
We then list the values of our generators on these loops, and refer to the function {\tt lowestWeightCondition} in the {\tt Mathematica} notebook {\tt Generators.nb} (included with the {\tt arXiv} sources of this article) for the elementary check of Lemma \ref{lem:verify-generator} that these actually determine a lowest weight vector with the desired rotational eigenvector. Note that this notebook takes quite a while to run on 4442, as it needs to lift the specified values, expressed as particular roots of their minimal polynomials, back to the fixed number field $\mathbb{Q}(\mu_{4442})$ described below. This notebook also regenerates all the values in accordance with the above lemma, in a format compatible with the {\tt FusionAtlas} package. This notebook, however, is completely independent of that package.
We note that this method of describing lowest weight vectors was implicitly used in \cite{0909.4099}, but without explanation of why it is always possible.

\subsection{4442}\label{sec:generators4442}

We give here two generators $A_0$ and $B_0$ which are rotational eigenvectors and lowest weight vectors, but are not self adjoint. We'll correct them in Subsection \ref{sec:moments4442} by a phase to obtain self-adjoint elements. 

The two generators $A_0,B_0$ for 4442 have rotational eigenvalues 
\begin{align*}
\omega_A &=\exp\left(2 \pi i\frac{3}{5}\right)  & \omega_B&=\exp\left(2 \pi i\frac{2}{5}\right) \\
\intertext{and we chose square roots $\sigma_A,\sigma_B$:}
\sigma_A &=\exp\left(2 \pi i \frac{8}{10}\right) & \sigma_B &=\exp\left(2 \pi i\frac{2}{10}\right).
\end{align*}

We express $A_0,B_0$ here by giving their coefficients on representatives of the rotational orbits of the collapsed loops. We write these coefficients as algebraic numbers, that is, roots of certain integer coefficient polynomials. Recall that the notation $\lambda_{a_n, \ldots, a_0}^{(z)}$ indicates the root of the polynomial $\sum a_i x^i$ which is approximately equal to $z$. (We give $z$ to sufficiently high precision that it clearly distinguishes amongst the roots.) In fact, we know that all of these numbers lie in a single number field, but it is a terrifying one: $\mathbb{Q}(\mu_{4442})$, where $\mu_{4442}$ is the root of
\begin{dmath*}
 x^{16}-137624 x^{15}+8933996874 x^{14}-350479594607884 x^{13}+9011981487580477099 x^{12}-153965505437561352450336 x^{11}+1677614319697333636358399288 x^{10}-10020910478354387137539071365292 x^9+14924720493503061057816631839615921 x^8+106599644350188183735570113321511839620 x^7+190438222993363250317220258290893662331668 x^6+186721586395717486269136749424837851231468948 x^5+116885475358299597752661285079387479235051171199 x^4+48661380081758389460116404892672194953761868153880 x^3+13292626196670539867585523753990034583320535147867050 x^2+2196889811116840388811767217587569058102181055950438988 x+171603625204099732635330174190840922237651443646904957481
 \end{dmath*}
which is approximately $ 17589.4 + 13246.7 i $. The polynomials in $\mu_{4442}$ required to express these numbers themselves tend to be horrific (coefficients whose numerator and denominator may have hundreds of digits), and we chose to spare the reader from the danger of trying to read them. The overenthusiastic may of course view them in the \code{Mathematica} notebook. Nevertheless, it is important to remember that the calculation of moments in \S \ref{sec:moments} must be performed inside this number field in order to be tractable.

The reason we work with the non-self-adjoint generators $A_0$ and $B_0$ is simply that the coefficients of the phase corrected generators require an even larger number field; sufficiently large, in fact, that our computers can't perform the necessary calculations there!

{\small
% generated by the command 'TypesetGenerators[{BigraphWithDuals[GradedBigraph[{{1}}, {{1}}, {{1}}, {{1}}, {{1}, {1}, {1}}, {{1, 0, 0}, {0, 1, 0}, {0, 0, 1}}, {{1, 0, 0}, {0, 1, 0}}, {{1, 0}, {0, 1}}], DualData[{1}, {1}, {1}, {2, 1, 3}, {2, 1}]], BigraphWithDuals[GradedBigraph[{{1}}, {{1}}, {{1}}, {{1}}, {{1}, {1}, {1}}, {{1, 0, 0}, {0, 1, 0}, {0, 0, 1}}, {{1, 0, 0}, {0, 1, 0}}, {{1, 0}, {0, 1}}], DualData[{1}, {1}, {1}, {2, 1, 3}, {2, 1}]]}, 1]' in the Mathematica notebook code/GenerateLaTeX.nb available with the arXiv sources of this article.
\begin{align*}
A_0(11112) & = \lambda_{3941830656,370676736,12897792,-5087232,111600,36792,852,-6,1}^{(0.009547+0.023000 i)} \displaybreak[1]\\
A_0(11113) & = \lambda_{1247219856,379505736,55366092,5211378,385155,25407,1422,54,1}^{(0.026510+0.063867 i)} \displaybreak[1]\\
A_0(11122) & = \lambda_{3941830656,2911168512,562837248,25208064,6989040,233784,7428,42,1}^{(-0.37586+0.02986 i)} \displaybreak[1]\\
A_0(11123) & = \lambda_{77951241,-40286727,21922488,-4650534,645165,-27324,2988,-27,31}^{(0.13774+0.13191 i)} \displaybreak[1]\\
A_0(11132) & = \lambda_{1247219856,652215888,202145868,6602796,6554925,997056,165888,9768,496}^{(0.11518-0.15201 i)} \displaybreak[1]\\
A_0(11133) & = \lambda_{77951241,-34088769,8046702,-1582902,130005,-10098,2412,24,16}^{(0.22307-0.01772 i)} \displaybreak[1]\\
A_0(11212) & = \lambda_{15397776,-13490712,3167748,-6096654,2141865,640836,414528,2883,961}^{(-0.15040+0.62850 i)} \displaybreak[1]\\
A_0(11213) & = \lambda_{1247219856,3177168624,5503926672,4667265522,2501734815,231250113,72236142,4203666,122461}^{(-0.57158-1.17710 i)} \displaybreak[1]\\
A_0(11222) & = \lambda_{3941830656,2911168512,562837248,25208064,6989040,233784,7428,42,1}^{(0.024549-0.102591 i)} \displaybreak[1]\\
A_0(11223) & = \lambda_{1247219856,-598818096,213923592,-131814378,67404555,-2266407,34767,-279,1}^{(-0.181705-0.437753 i)} \displaybreak[1]\\
A_0(11232) & = \lambda_{1247219856,1315874160,1776121020,856453500,225182025,32508000,1964250,13125,625}^{(-0.23036+0.96265 i)} \displaybreak[1]\\
A_0(11233) & = \lambda_{77951241,-49106898,33170958,-32103216,9349425,-303966,101448,1392,16}^{(0.48303-0.03838 i)} \displaybreak[1]\\
A_0(11312) & = \lambda_{1247219856,-884877696,195611112,-129102498,2093568525,564703758,25239672,-2431059,122461}^{(1.04250-0.79086 i)} \displaybreak[1]\\
A_0(11313) & = \lambda_{77951241,74852262,35328798,11706039,2207655,166914,12168,432,16}^{(-0.08426+0.35214 i)} \displaybreak[1]\\
A_0(11322) & = \lambda_{35316,-24678,7389,168,1}^{(0.360178-0.308085 i)} \displaybreak[1]\\
A_0(11323) & = \lambda_{77951241,-155425716,96700392,-79082163,68641425,-4964112,278712,-144,16}^{(-0.18246+0.76250 i)} \displaybreak[1]\\
A_0(11332) & = \lambda_{77951241,121336947,113029263,60444549,17861715,2067174,73368,-1248,16}^{(-0.41334-0.25287 i)} \displaybreak[1]\\
A_0(12122) & = \lambda_{15397776,-13490712,3167748,-6096654,2141865,640836,414528,2883,961}^{(0.77203-0.06134 i)} \displaybreak[1]\\
A_0(12123) & = \lambda_{81,108,144,72,16}^{(-0.3333+1.0259 i)} \displaybreak[1]\\
A_0(12132) & = \lambda_{77951241,-75805794,99558072,-91700667,145832805,-61946748,27924912,-9145536,3041536}^{(-0.4911-0.9604 i)} \displaybreak[1]\\
A_0(12133) & = \lambda_{77951241,-151373205,49269465,15230025,11397915,918000,319500,6000,400}^{(1.15226-0.09155 i)} \displaybreak[1]\\
A_0(12213) & = \lambda_{1247219856,-2526859800,1974773520,-491005800,55811025,-10587375,2076750,-45000,625}^{(-0.06347-0.15291 i)} \displaybreak[1]\\
A_0(12222) & = \lambda_{3941830656,370676736,12897792,-5087232,111600,36792,852,-6,1}^{(0.075927+0.046450 i)} \displaybreak[1]\\
A_0(12223) & = \lambda_{77951241,31228173,13579083,7130106,2059425,324216,29448,1458,31}^{(0.13497-0.38939 i)} \displaybreak[1]\\
A_0(12232) & = \lambda_{1247219856,1670588064,-268120368,1191806622,1842147765,-527829912,98070912,-6203229,122461}^{(0.49113-0.87531 i)} \displaybreak[1]\\
A_0(12233) & = \lambda_{77951241,83195667,22443723,14206509,7526925,-899316,106308,-48,16}^{(-0.70611+0.05610 i)} \displaybreak[1]\\
A_0(12313) & = \lambda_{77951241,-175449888,329852088,-416257056,333376560,-158331456,44336448,-6899328,476416}^{(0.07491+1.38254 i)} \displaybreak[1]\\
A_0(12322) & = \lambda_{1247219856,-884877696,195611112,-129102498,2093568525,564703758,25239672,-2431059,122461}^{(-0.55570+0.83581 i)} \displaybreak[1]\\
A_0(12323) & = \lambda_{77951241,325154412,587305728,611166384,426416400,193325184,52769088,7788672,476416}^{(-0.24343-0.67826 i)} \displaybreak[1]\\
A_0(12332) & = \lambda_{77951241,-134686395,185147775,-112077675,34665165,-5968350,586800,-24000,400}^{(0.031564+0.019310 i)} \displaybreak[1]\\
A_0(13132) & = \lambda_{190096,-298224,48892,140748,-31635,79392,9592,14424,29776}^{(1.1727-0.5217 i)} \displaybreak[1]\\
A_0(13133) & = \lambda_{11881,31174,44152,52112,25680,11968,17152,-1024,4096}^{(-1.3384+0.1063 i)} \displaybreak[1]\\
A_0(13222) & = \lambda_{3041536,-390656,219552,-53008,49325,-14677,3417,-389,31}^{(-0.26442+0.27576 i)} \displaybreak[1]\\
A_0(13223) & = \lambda_{190096,419432,437068,227614,74265,-12256,2458,97,1}^{(-0.90634-0.55448 i)} \displaybreak[1]\\
A_0(13232) & = \lambda_{190096,-298224,48892,140748,-31635,79392,9592,14424,29776}^{(-0.6626+0.0854 i)} \displaybreak[1]\\
A_0(13233) & = \lambda_{11881,-15587,11038,-6514,1605,-374,268,8,16}^{(0.6692-0.0532 i)} \displaybreak[1]\\
A_0(13322) & = \lambda_{190096,155216,93912,6998,-10335,-2908,1022,39,1}^{(-0.25174-0.60649 i)} \displaybreak[1]\\
A_0(13323) & = \lambda_{11881,-3052,1228,1861,255,-1174,628,-152,16}^{(0.2574+0.6200 i)} \displaybreak[1]\\
A_0(22223) & = \lambda_{1247219856,379505736,55366092,5211378,385155,25407,1422,54,1}^{(0.026510+0.063867 i)} \displaybreak[1]\\
A_0(22233) & = \lambda_{77951241,-34088769,8046702,-1582902,130005,-10098,2412,24,16}^{(0.22307-0.01772 i)} \displaybreak[1]\\
A_0(22323) & = \lambda_{77951241,74852262,35328798,11706039,2207655,166914,12168,432,16}^{(-0.08426+0.35214 i)} \displaybreak[1]\\
A_0(23233) & = \lambda_{11881,31174,44152,52112,25680,11968,17152,-1024,4096}^{(-1.3384+0.1063 i)}
\end{align*}\begin{align*}
B_0(11112) & = \lambda_{3941830656,370676736,12897792,-5087232,111600,36792,852,-6,1}^{(0.075927+0.046450 i)} \displaybreak[1]\\
B_0(11113) & = \lambda_{1247219856,379505736,55366092,5211378,385155,25407,1422,54,1}^{(-0.058988-0.036087 i)} \displaybreak[1]\\
B_0(11122) & = \lambda_{3941830656,2911168512,562837248,25208064,6989040,233784,7428,42,1}^{(0.024549-0.102591 i)} \displaybreak[1]\\
B_0(11123) & = \lambda_{77951241,-40286727,21922488,-4650534,645165,-27324,2988,-27,31}^{(0.11968+0.39436 i)} \displaybreak[1]\\
B_0(11132) & = \lambda_{1247219856,652215888,202145868,6602796,6554925,997056,165888,9768,496}^{(-0.28523+0.29746 i)} \displaybreak[1]\\
B_0(11133) & = \lambda_{77951241,-34088769,8046702,-1582902,130005,-10098,2412,24,16}^{(0.05208-0.21763 i)} \displaybreak[1]\\
B_0(11212) & = \lambda_{15397776,-13490712,3167748,-6096654,2141865,640836,414528,2883,961}^{(0.77203-0.06134 i)} \displaybreak[1]\\
B_0(11213) & = \lambda_{1247219856,3177168624,5503926672,4667265522,2501734815,231250113,72236142,4203666,122461}^{(-0.68071-0.73758 i)} \displaybreak[1]\\
B_0(11222) & = \lambda_{3941830656,2911168512,562837248,25208064,6989040,233784,7428,42,1}^{(-0.37586+0.02986 i)} \displaybreak[1]\\
B_0(11223) & = \lambda_{1247219856,-598818096,213923592,-131814378,67404555,-2266407,34767,-279,1}^{(0.404307+0.247347 i)} \displaybreak[1]\\
B_0(11232) & = \lambda_{1247219856,1315874160,1776121020,856453500,225182025,32508000,1964250,13125,625}^{(-0.16504+0.01311 i)} \displaybreak[1]\\
B_0(11233) & = \lambda_{77951241,-49106898,33170958,-32103216,9349425,-303966,101448,1392,16}^{(-0.16485+0.68888 i)} \displaybreak[1]\\
B_0(11312) & = \lambda_{1247219856,-884877696,195611112,-129102498,2093568525,564703758,25239672,-2431059,122461}^{(-0.55570+0.83581 i)} \displaybreak[1]\\
B_0(11313) & = \lambda_{77951241,74852262,35328798,11706039,2207655,166914,12168,432,16}^{(-0.36094+0.02868 i)} \displaybreak[1]\\
B_0(11322) & = \lambda_{35316,-24678,7389,168,1}^{(0.360178-0.308085 i)} \displaybreak[1]\\
B_0(11323) & = \lambda_{77951241,-155425716,96700392,-79082163,68641425,-4964112,278712,-144,16}^{(1.14251-0.09077 i)} \displaybreak[1]\\
B_0(11332) & = \lambda_{77951241,121336947,113029263,60444549,17861715,2067174,73368,-1248,16}^{(-0.27155-0.65421 i)} \displaybreak[1]\\
B_0(12122) & = \lambda_{15397776,-13490712,3167748,-6096654,2141865,640836,414528,2883,961}^{(-0.15040+0.62850 i)} \displaybreak[1]\\
B_0(12123) & = \lambda_{81,108,144,72,16}^{(-0.3333-1.0259 i)} \displaybreak[1]\\
B_0(12132) & = \lambda_{77951241,-75805794,99558072,-91700667,145832805,-61946748,27924912,-9145536,3041536}^{(0.7616-0.7639 i)} \displaybreak[1]\\
B_0(12133) & = \lambda_{77951241,-151373205,49269465,15230025,11397915,918000,319500,6000,400}^{(-0.008611+0.035986 i)} \displaybreak[1]\\
B_0(12213) & = \lambda_{1247219856,-2526859800,1974773520,-491005800,55811025,-10587375,2076750,-45000,625}^{(0.84435+0.51656 i)} \displaybreak[1]\\
B_0(12222) & = \lambda_{3941830656,370676736,12897792,-5087232,111600,36792,852,-6,1}^{(0.009547+0.023000 i)} \displaybreak[1]\\
B_0(12223) & = \lambda_{77951241,31228173,13579083,7130106,2059425,324216,29448,1458,31}^{(-0.18897-0.02576 i)} \displaybreak[1]\\
B_0(12232) & = \lambda_{1247219856,1670588064,-268120368,1191806622,1842147765,-527829912,98070912,-6203229,122461}^{(-1.29611+0.17986 i)} \displaybreak[1]\\
B_0(12233) & = \lambda_{77951241,83195667,22443723,14206509,7526925,-899316,106308,-48,16}^{(0.11277-0.47125 i)} \displaybreak[1]\\
B_0(12313) & = \lambda_{77951241,-175449888,329852088,-416257056,333376560,-158331456,44336448,-6899328,476416}^{(0.56984-0.44112 i)} \displaybreak[1]\\
B_0(12322) & = \lambda_{1247219856,-884877696,195611112,-129102498,2093568525,564703758,25239672,-2431059,122461}^{(1.04250-0.79086 i)} \displaybreak[1]\\
B_0(12323) & = \lambda_{77951241,325154412,587305728,611166384,426416400,193325184,52769088,7788672,476416}^{(-1.29172+0.49847 i)} \displaybreak[1]\\
B_0(12332) & = \lambda_{77951241,-134686395,185147775,-112077675,34665165,-5968350,586800,-24000,400}^{(0.44313+1.06757 i)} \displaybreak[1]\\
B_0(13132) & = \lambda_{190096,-298224,48892,140748,-31635,79392,9592,14424,29776}^{(-0.6626+0.0854 i)} \displaybreak[1]\\
B_0(13133) & = \lambda_{11881,31174,44152,52112,25680,11968,17152,-1024,4096}^{(-0.3125+1.3058 i)} \displaybreak[1]\\
B_0(13222) & = \lambda_{3041536,-390656,219552,-53008,49325,-14677,3417,-389,31}^{(0.10678-0.14092 i)} \displaybreak[1]\\
B_0(13223) & = \lambda_{190096,419432,437068,227614,74265,-12256,2458,97,1}^{(-0.27865-0.67130 i)} \displaybreak[1]\\
B_0(13232) & = \lambda_{190096,-298224,48892,140748,-31635,79392,9592,14424,29776}^{(1.1727-0.5217 i)} \displaybreak[1]\\
B_0(13233) & = \lambda_{11881,-15587,11038,-6514,1605,-374,268,8,16}^{(0.1562-0.6529 i)} \displaybreak[1]\\
B_0(13322) & = \lambda_{190096,155216,93912,6998,-10335,-2908,1022,39,1}^{(-0.38319-0.23443 i)} \displaybreak[1]\\
B_0(13323) & = \lambda_{11881,-3052,1228,1861,255,-1174,628,-152,16}^{(-0.5727-0.3503 i)} \displaybreak[1]\\
B_0(22223) & = \lambda_{1247219856,379505736,55366092,5211378,385155,25407,1422,54,1}^{(-0.058988-0.036087 i)} \displaybreak[1]\\
B_0(22233) & = \lambda_{77951241,-34088769,8046702,-1582902,130005,-10098,2412,24,16}^{(0.05208-0.21763 i)} \displaybreak[1]\\
B_0(22323) & = \lambda_{77951241,74852262,35328798,11706039,2207655,166914,12168,432,16}^{(-0.36094+0.02868 i)} \displaybreak[1]\\
B_0(23233) & = \lambda_{11881,31174,44152,52112,25680,11968,17152,-1024,4096}^{(-0.3125+1.3058 i)}
\end{align*}
}
\subsection{3333}\label{sec:generators3333}
The self-adjoint generators $A,B$ of $P_\bullet^{3333}$ have chiralities $\sigma_A=1$ and $\sigma_B=1$.
Their values on collapsed loops are as follows:
{\small
% generated by the command 'TypesetGenerators[{BigraphWithDuals[GradedBigraph[{{1}}, {{1}}, {{1}}, {{1}, {1}, {1}}, {{1, 0, 0}, {0, 1, 0}, {0, 0, 1}}, {{1, 0, 0}, {0, 1, 0}, {0, 0, 1}}], DualData[{1}, {1}, {1, 2, 3}, {1, 2, 3}]], BigraphWithDuals[GradedBigraph[{{1}}, {{1}}, {{1}}, {{1}, {1}, {1}}, {{1, 0, 0}, {0, 1, 0}, {0, 0, 1}}, {{1, 0, 0}, {0, 1, 0}, {0, 0, 1}}], DualData[{1}, {1}, {1, 2, 3}, {1, 2, 3}]]}, 2]' in the Mathematica notebook code/GenerateLaTeX.nb available with the arXiv sources of this article.
\begin{align*}
A(1112) & = \frac{1}{4} \left(2-\sqrt{5}\right) &
A(1113) & = \frac{1}{8} \left(3 \sqrt{5}-7\right) \displaybreak[1]\\
A(1122) & = \frac{1}{4} \left(\sqrt{5}-3\right) &
A(1123) & = \frac{1}{8} \left(3-\sqrt{5}\right) \displaybreak[1]\\
A(1132) & = \frac{1}{8} \left(3-\sqrt{5}\right) &
A(1133) & = \frac{1}{4} \left(3-\sqrt{5}\right) \displaybreak[1]\\
A(1212) & = \frac{1}{4} \left(\sqrt{5}-1\right) &
A(1213) & = \frac{1}{8} \left(-1-\sqrt{5}\right) \displaybreak[1]\\
A(1222) & = \frac{1}{4} \left(2-\sqrt{5}\right) &
A(1223) & = \frac{1}{8} \left(\sqrt{5}-1\right) \displaybreak[1]\\
A(1232) & = \frac{1}{8} \left(\sqrt{5}-1\right) &
A(1233) & = -\frac{1}{4} \displaybreak[1]\\
A(1313) & = \frac{1}{4} \left(3-\sqrt{5}\right) &
A(1322) & = \frac{1}{8} \left(\sqrt{5}-1\right) \displaybreak[1]\\
A(1323) & = \frac{1}{4} &
A(1332) & = -\frac{1}{4} \displaybreak[1]\\
A(1333) & = \frac{1}{8} \left(3 \sqrt{5}-7\right) &
A(2223) & = \frac{1}{8} \left(3-\sqrt{5}\right) \displaybreak[1]\\
A(2233) & = 0 &
A(2323) & = -\frac{1}{2} \displaybreak[1]\\
A(2333) & = \frac{1}{8} \left(3-\sqrt{5}\right)
\end{align*}\begin{align*}
B(1112) & = \frac{1}{8} \left(7 \sqrt{5}-15\right) &
B(1113) & = \frac{1}{8} \left(\sqrt{5}-3\right) \displaybreak[1]\\
B(1122) & = \frac{1}{2} \left(2-\sqrt{5}\right) &
B(1123) & = \frac{1}{4} \displaybreak[1]\\
B(1132) & = \frac{1}{4} &
B(1133) & = \frac{1}{2} \left(2-\sqrt{5}\right) \displaybreak[1]\\
B(1212) & = \frac{1}{4} \left(5-3 \sqrt{5}\right) &
B(1213) & = \frac{1}{4} \left(2-\sqrt{5}\right) \displaybreak[1]\\
B(1222) & = \frac{1}{8} \left(7 \sqrt{5}-15\right) &
B(1223) & = \frac{1}{8} \left(5-3 \sqrt{5}\right) \displaybreak[1]\\
B(1232) & = \frac{1}{8} \left(1+\sqrt{5}\right) &
B(1233) & = \frac{1}{8} \left(3 \sqrt{5}-7\right) \displaybreak[1]\\
B(1313) & = \frac{1}{2} &
B(1322) & = \frac{1}{8} \left(5-3 \sqrt{5}\right) \displaybreak[1]\\
B(1323) & = \frac{1}{8} \left(\sqrt{5}-5\right) &
B(1332) & = \frac{1}{8} \left(3 \sqrt{5}-7\right) \displaybreak[1]\\
B(1333) & = \frac{1}{8} \left(\sqrt{5}-3\right) &
B(2223) & = \frac{1}{4} \left(9-4 \sqrt{5}\right) \displaybreak[1]\\
B(2233) & = \sqrt{5}-2 &
B(2323) & = \frac{1}{4} \left(3 \sqrt{5}-7\right) \displaybreak[1]\\
B(2333) & = \frac{1}{4} \left(9-4 \sqrt{5}\right)
\end{align*}
}
Clearly all of these entries lie in the field $\mathbb{Q}(\sqrt{5})$.

%\todo{The number field generator is not $\sqrt{5}$ for 3333, so shouldn't the number field be
%$$
%\mathbb{Q}\left(i \sqrt{1+\sqrt{5}}\right)
%$$
%?
%and why do we have to work in a number field larger than $\mathbb{Q}(\sqrt{5})$? Is it to get the rest of the loop coefficients? this is what i said down below for 3311, but I don't know if that's true.
%}

\subsection{3311}\label{sec:generators3311}
The self-adjoint generators $A,B$ of $P_\bullet^{3311}$ have chiralities $\sigma_A=i$, $\omega_A=\sigma_A^2 = -1$, and $\sigma_B=\omega_B=1$.
Their values on collapsed loops are as follows:
{\small
% generated by the command 'TypesetGenerators[{BigraphWithDuals[GradedBigraph[{{1}}, {{1}}, {{1}}, {{1}, {1}, {1}}, {{1, 0, 0}}, {{1}}], DualData[{1}, {1}, {1, 2, 3}, {1}]], BigraphWithDuals[GradedBigraph[{{1}}, {{1}}, {{1}}, {{1}, {1}, {1}}, {{1, 0, 0}}, {{1}}], DualData[{1}, {1}, {1, 2, 3}, {1}]]}, 2]' in the Mathematica notebook code/GenerateLaTeX.nb available with the arXiv sources of this article.
\begin{align*}
A(1112) & = \frac{1}{6} \left(9-5 \sqrt{3}\right) &
A(1113) & = \frac{1}{6} \left(3 \sqrt{3}-5\right) \displaybreak[1]\\
A(1123) & = \frac{1}{3} \left(\sqrt{3}-1\right) &
A(1132) & = \frac{1}{3} \left(1-\sqrt{3}\right) \displaybreak[1]\\
A(1212) & = \frac{1}{6} \left(5 \sqrt{3}-9\right) &
A(1213) & = \frac{1}{6} \left(\sqrt{3}-3\right) \displaybreak[1]\\
A(1232) & = \frac{1}{3} \left(\sqrt{3}-1\right) &
A(1313) & = \frac{1}{6} \left(7-3 \sqrt{3}\right) \displaybreak[1]\\
A(1323) & = \frac{1}{3} \left(1-\sqrt{3}\right) &
A(2323) & = 0
\end{align*}\begin{align*}
B(1112) & = \frac{1}{4} \left(9-5 \sqrt{3}\right) &
B(1113) & = \frac{1}{12} \left(\sqrt{3}-3\right) \displaybreak[1]\\
B(1123) & = \frac{1}{6} \left(3-\sqrt{3}\right) &
B(1132) & = \frac{1}{6} \left(3-\sqrt{3}\right) \displaybreak[1]\\
B(1212) & = \frac{1}{4} \left(13 \sqrt{3}-21\right) &
B(1213) & = \frac{1}{4} \left(\sqrt{3}-3\right) \displaybreak[1]\\
B(1232) & = \frac{1}{6} \left(7 \sqrt{3}-15\right) &
B(1313) & = \frac{1}{12} \left(1-\sqrt{3}\right) \displaybreak[1]\\
B(1323) & = \frac{1}{6} \left(\sqrt{3}-1\right) &
B(2323) & = \frac{1}{3} \left(2 \sqrt{3}-2\right)
\end{align*}
}
Clearly these entries all lie in $\mathbb{Q}(\sqrt{3})$.
%, \nn{but to get the rest of the coefficients as described in the beginning of Appendix \ref{sec:generators}, we need to work in the number field
%$$
%\mathbb{Q}\left(\sqrt{-85+47 \sqrt{3}-4 \sqrt{3 \left(3 \sqrt{3}-5\right)}}\right).
%$$
%(TRUE? -D)} \nn{This is an excellent question, and it seems rather strange that we need a still bigger field --S}

\subsection{2221}\label{sec:generators2221}
As in Subsection \ref{sec:generators4442}, we work with non-self-adjoint generators for $2221$. We will correct them by phases in Subsection \ref{sec:moments2221} to get self-adjoint elements.

The generators $A_0,B_0$ have rotational eigenvalues
$$
\omega_A = \exp\left(2\pi i \frac{1}{3}\right) \hspace{2cm}\omega_B=1
$$
for which we choose square roots
$$
\sigma_A=\exp\left(2\pi i \frac{4}{6}\right)\hspace{2cm} \sigma_B=1.
$$
Their values on collapsed loops are as follows:
{\small
% generated by the command 'TypesetGenerators[{BigraphWithDuals[GradedBigraph[{{1}}, {{1}}, {{1}, {1}, {1}}, {{1, 0, 0}, {0, 1, 0}}], DualData[{1}, {1}, {2, 1}]], BigraphWithDuals[GradedBigraph[{{1}}, {{1}}, {{1}, {1}, {1}}, {{1, 0, 0}, {0, 1, 0}}], DualData[{1}, {1}, {2, 1}]]}, 1]' in the Mathematica notebook code/GenerateLaTeX.nb available with the arXiv sources of this article.
\begin{align*}
A_0(112) & = \lambda_{9,9,9,-18,21,-12,12,-6,1}^{(0.5762-0.5412 i)} &
A_0(113) & = \lambda_{81,81,54,-27,-27,-9,6,3,1}^{(-0.4208+0.3953 i)} \displaybreak[1]\\
A_0(122) & = \lambda_{9,9,9,-18,21,-12,12,-6,1}^{(0.3194+0.0964 i)} &
A_0(123) & = \lambda_{9,9,12,-3,1}^{(-0.6319-1.0945 i)} \displaybreak[1]\\
A_0(132) & = \lambda_{9,27,30,12,1}^{(-1.1319+0.5621 i)} &
A_0(223) & = \lambda_{81,81,54,-27,-27,-9,6,3,1}^{(-0.4208+0.3953 i)}
\end{align*}\begin{align*}
B_0(112) & = \lambda_{225,-45,288,-45,201,-39,36,3,1}^{(-0.0605-0.1502 i)} &
B_0(113) & = \lambda_{45,36,42,3,-1}^{(-0.3528-0.8759 i)} \displaybreak[1]\\
B_0(122) & = \lambda_{225,-45,288,-45,201,-39,36,3,1}^{(0.3396+0.8433 i)} &
B_0(123) & = \lambda_{45,54,24,-39,5}^{(-0.9110+0.8759 i)} \displaybreak[1]\\
B_0(132) & = \frac{1}{6} \left(3+\sqrt{21}\right) &
B_0(223) & = \lambda_{45,36,42,3,-1}^{(-0.3528-0.8759 i)}
\end{align*}
}
These entries all lie in the number field $\mathbb{Q}(\mu_{2221})$ where $\mu_{2221}$ is the root of 
\begin{dmath*}
x^8-18 x^7-345 x^6+7146 x^5+84726 x^4-1458918 x^3-13821786 x^2+101759328 x+1245393549
\end{dmath*}
which is approximately $14.85+9.90 i$.

In fact, we could have tried working directly with self-adjoint generators, but this would have required a degree 32 number field. We could still calculate the requisite moments; however, we could not express the values of the generators on collapsed loops in the above compact form, and  instead, we would have to write out the 32 coefficients in the number field for each value.

%%%%%%%%%%%%%%%%%%%%%%%%%%%%%%%%%%%%%%%%%%%%%%%%%%%%%%%%%%%%%%%%%%%%%%%%%
%%%%%%%%%%%%%%%%%%%%%%%%%%%%%%%%%%%%%%%%%%%%%%%%%%%%%%%%%%%%%%%%%%%%%%%%%
%%%%%%%%%%%%%%%%%%%%%%%%%%%%%%%%%%%%%%%%%%%%%%%%%%%%%%%%%%%%%%%%%%%%%%%%%
\section{Moments}\label{sec:moments}

In the following subsections, we give the quadratic and cubic moments for our generators calculated directly from the graph planar algebra. Recall that the generators given in Appendix \ref{sec:generators} are not normalized, so $\Tr(A^2)$ and $\Tr(B^2)$ need not be equal to $1$.

\subsection{4442}\label{sec:moments4442}
Recall from Appendix \ref{sec:generators4442} that the generators given for $4442$ are not self-adjoint, since it would require working in an even larger number field (as if $\mathbb{Q}(\mu_{4442})$ isn't terrifying enough)! However, for each non self-adjoint $S_0\in\fB_0$, there is a $\gamma_S\in U(1)$ such that $S=\gamma_S S_0$ is self-adjoint. Hence we compute the moments in the graph planar algebra using the non self-adjoint generators, and we correct them afterward to agree with what we would get from first making our generators self-adjoint. This amounts to multiplying the moment by the correction factor for each generator that appears, e.g., 
$$
\Tr(A^2B)=\Tr((\gamma_AA_0)^2\gamma_BB_0)= \gamma_A^2\gamma_B \Tr(A_0^2 B_0).
$$
In the case of 4442, the correction factors for $A_0,B_0$ are equal and given by
$$
\gamma_A=\gamma_B=\sqrt{\frac{437}{872}-\frac{135 \sqrt{5}}{872}+\frac{1}{436} i \sqrt{\frac{15}{2} \left(15943+3933 \sqrt{5}\right)}},
$$
and the corrected moments are as follows: 
% generated by the command 'TypesetMomentsOneColumn[{BigraphWithDuals[GradedBigraph[{{1}}, {{1}}, {{1}}, {{1}}, {{1}, {1}, {1}}, {{1, 0, 0}, {0, 1, 0}, {0, 0, 1}}, {{1, 0, 0}, {0, 1, 0}}, {{1, 0}, {0, 1}}], DualData[{1}, {1}, {1}, {2, 1, 3}, {2, 1}]], BigraphWithDuals[GradedBigraph[{{1}}, {{1}}, {{1}}, {{1}}, {{1}, {1}, {1}}, {{1, 0, 0}, {0, 1, 0}, {0, 0, 1}}, {{1, 0, 0}, {0, 1, 0}}, {{1, 0}, {0, 1}}], DualData[{1}, {1}, {1}, {2, 1, 3}, {2, 1}]]}, 1]' in the Mathematica notebook code/GenerateLaTeX.nb available with the arXiv sources of this article.
\begin{align*}
\Tr(A^2) &= \sqrt{\frac{12706131015}{11881}+\frac{5682354525 \sqrt{5}}{11881}}\displaybreak[1]\\
\Tr(AB) &= 0\displaybreak[1]\\
\Tr(B^2) &= \sqrt{\frac{12706131015}{11881}+\frac{5682354525 \sqrt{5}}{11881}}\displaybreak[1]\\
\Tr(A^3) &= \sqrt{\frac{36768021311025}{2590058}+\frac{16443159009075 \sqrt{5}}{2590058}}\displaybreak[1]\\
\Tr(A^2B) &= -\sqrt{\frac{25201176715665}{1295029}+\frac{11270308849830 \sqrt{5}}{1295029}}\displaybreak[1]\\
\Tr(AB^2) &= -\sqrt{\frac{25201176715665}{1295029}+\frac{11270308849830 \sqrt{5}}{1295029}}\displaybreak[1]\\
\Tr(B^3) &= \sqrt{\frac{36768021311025}{2590058}+\frac{16443159009075 \sqrt{5}}{2590058}}\displaybreak[1]\\
\displaybreak[1]\\
\Tr(\check{A}^2) &= \sqrt{\frac{12706131015}{11881}+\frac{5682354525 \sqrt{5}}{11881}}\displaybreak[1]\\
\Tr(\check{A}\check{B}) &= 0\displaybreak[1]\\
\Tr(\check{B}^2) &= \sqrt{\frac{12706131015}{11881}+\frac{5682354525 \sqrt{5}}{11881}}\displaybreak[1]\\
\Tr(\check{A}^3) &= \sqrt{\frac{36768021311025}{2590058}+\frac{16443159009075 \sqrt{5}}{2590058}}\displaybreak[1]\\
\Tr(\check{A}^2\check{B}) &= -\sqrt{\frac{25201176715665}{1295029}+\frac{11270308849830 \sqrt{5}}{1295029}}\displaybreak[1]\\
\Tr(\check{A}\check{B}^2) &= -\sqrt{\frac{25201176715665}{1295029}+\frac{11270308849830 \sqrt{5}}{1295029}}\displaybreak[1]\\
\Tr(\check{B}^3) &= \sqrt{\frac{36768021311025}{2590058}+\frac{16443159009075 \sqrt{5}}{2590058}}
\end{align*}

\subsection{3333}\label{sec:moments3333}
Since the generators of 3333 are self-adjoint, no corrections are needed. The moments are as follows:
% generated by the command 'TypesetMomentsTwoColumns[{BigraphWithDuals[GradedBigraph[{{1}}, {{1}}, {{1}}, {{1}, {1}, {1}}, {{1, 0, 0}, {0, 1, 0}, {0, 0, 1}}, {{1, 0, 0}, {0, 1, 0}, {0, 0, 1}}], DualData[{1}, {1}, {1, 2, 3}, {1, 2, 3}]], BigraphWithDuals[GradedBigraph[{{1}}, {{1}}, {{1}}, {{1}, {1}, {1}}, {{1, 0, 0}, {0, 1, 0}, {0, 0, 1}}, {{1, 0, 0}, {0, 1, 0}, {0, 0, 1}}], DualData[{1}, {1}, {1, 2, 3}, {1, 2, 3}]]}, 1]' in the Mathematica notebook code/GenerateLaTeX.nb available with the arXiv sources of this article.
\begin{align*}
\Tr(A^2) &= \frac{1}{2} \left(11+5 \sqrt{5}\right)&\Tr(\check{A}^2) &= \frac{1}{2} \left(11+5 \sqrt{5}\right)\\
\Tr(AB) &= 0&\Tr(\check{A}\check{B}) &= 0\\
\Tr(B^2) &= 6+3 \sqrt{5}&\Tr(\check{B}^2) &= 6+3 \sqrt{5}\\
\Tr(A^3) &= \frac{1}{8} \left(-27-12 \sqrt{5}\right)&\Tr(\check{A}^3) &= 0\\
\Tr(A^2B) &= \frac{1}{16} \left(-25-11 \sqrt{5}\right)&\Tr(\check{A}^2\check{B}) &= -\sqrt{\frac{123}{4}+\frac{55 \sqrt{5}}{4}}\\
\Tr(AB^2) &= \frac{1}{16} \left(63+27 \sqrt{5}\right)&\Tr(\check{A}\check{B}^2) &= 0\\
\Tr(B^3) &= \frac{1}{8} \left(15+6 \sqrt{5}\right)&\Tr(\check{B}^3) &= 3 \sqrt{\frac{1}{2} \left(9+4 \sqrt{5}\right)}
\end{align*}

\subsection{3311}\label{sec:moments3311}
Since the generators of 3311 are self-adjoint, no corrections are needed. The moments are as follows:
% generated by the command 'TypesetMomentsTwoColumns[{BigraphWithDuals[GradedBigraph[{{1}}, {{1}}, {{1}}, {{1}, {1}, {1}}, {{1, 0, 0}}, {{1}}], DualData[{1}, {1}, {1, 2, 3}, {1}]], BigraphWithDuals[GradedBigraph[{{1}}, {{1}}, {{1}}, {{1}, {1}, {1}}, {{1, 0, 0}}, {{1}}], DualData[{1}, {1}, {1, 2, 3}, {1}]]}, 2]' in the Mathematica notebook code/GenerateLaTeX.nb available with the arXiv sources of this article.
\begin{align*}
\Tr(A^2) &= \frac{1}{9} \left(12+8 \sqrt{3}\right)&\Tr(\check{A}^2) &= \frac{1}{9} \left(12+8 \sqrt{3}\right)\\
\Tr(AB) &= 0&\Tr(\check{A}\check{B}) &= 0\\
\Tr(B^2) &= \frac{1}{3} \left(9+4 \sqrt{3}\right)&\Tr(\check{B}^2) &= \frac{1}{3} \left(9+4 \sqrt{3}\right)\\
\Tr(A^3) &= \frac{1}{27} \left(-6-4 \sqrt{3}\right)&\Tr(\check{A}^3) &= 0\\
\Tr(A^2B) &= \frac{1}{9} \left(-3-2 \sqrt{3}\right)&\Tr(\check{A}^2\check{B}) &= -\sqrt{\frac{44}{27}+\frac{76}{27 \sqrt{3}}}\\
\Tr(AB^2) &= \frac{1}{18} \left(15+14 \sqrt{3}\right)&\Tr(\check{A}\check{B}^2) &= 0\\
\Tr(B^3) &= \frac{1}{12} \left(23-2 \sqrt{3}\right)&\Tr(\check{B}^3) &= \sqrt{\frac{9}{4}+\frac{1}{4 \sqrt{3}}}
\end{align*}

\subsection{2221}\label{sec:moments2221}
As in Subsection \ref{sec:moments4442} for 4442, we now correct our non-self-adjoint generators for 2221 by multiplying by the following phases:
\begin{align*}
\gamma_A&=
\frac{1}{2} \sqrt{1-\sqrt{21}-i \sqrt{2 \left(\sqrt{21}-3\right)}}
\\
\gamma_B&=
\sqrt{-\frac{19}{20}+\frac{\sqrt{21}}{20}-\frac{1}{10} i \sqrt{\frac{1}{2} \left(9+19 \sqrt{21}\right)}},
\end{align*}
which yield the corrected moments:
% generated by the command 'TypesetMomentsTwoColumns[{BigraphWithDuals[GradedBigraph[{{1}}, {{1}}, {{1}, {1}, {1}}, {{1, 0, 0}, {0, 1, 0}}], DualData[{1}, {1}, {2, 1}]], BigraphWithDuals[GradedBigraph[{{1}}, {{1}}, {{1}, {1}, {1}}, {{1, 0, 0}, {0, 1, 0}}], DualData[{1}, {1}, {2, 1}]]}, 1]' in the Mathematica notebook code/GenerateLaTeX.nb available with the arXiv sources of this article.
\begin{align*}
\Tr(A^2) &= 3 \sqrt{55+12 \sqrt{21}}&\Tr(\check{A}^2) &= 3 \sqrt{55+12 \sqrt{21}}\\
\Tr(AB) &= 0&\Tr(\check{A}\check{B}) &= 0\\
\Tr(B^2) &= \sqrt{\frac{60093}{50}+\frac{13113 \sqrt{21}}{50}}&\Tr(\check{B}^2) &= \sqrt{\frac{60093}{50}+\frac{13113 \sqrt{21}}{50}}\\
\Tr(A^3) &= -18-4 \sqrt{21}&\Tr(\check{A}^3) &= -18-4 \sqrt{21}\\
\Tr(A^2B) &= \sqrt{\frac{4413}{10}+\frac{963 \sqrt{21}}{10}}&\Tr(\check{A}^2\check{B}) &= \sqrt{\frac{4413}{10}+\frac{963 \sqrt{21}}{10}}\\
\Tr(AB^2) &= \frac{1}{5} \left(198+43 \sqrt{21}\right)&\Tr(\check{A}\check{B}^2) &= \frac{1}{5} \left(198+43 \sqrt{21}\right)\\
\Tr(B^3) &= -\sqrt{\frac{11667}{250}+\frac{2547 \sqrt{21}}{250}}&\Tr(\check{B}^3) &= -\sqrt{\frac{11667}{250}+\frac{2547 \sqrt{21}}{250}}
\end{align*}

\bibliographystyle{alpha}
%Included for winedt:
%input "bibliography/bibliography.bib"
\bibliography{../../bibliography/bibliography}

\end{document}